\documentclass[a4paper,reqno,11pt]{amsart}
\usepackage{amsmath,amsthm,amsfonts,amssymb,color,enumerate,enumitem,xfrac,tikz,comment, pgfplots,soul}
\usepackage[colorlinks,citecolor=blue,urlcolor=blue]{hyperref}
\usetikzlibrary{shapes.geometric,positioning}
\excludecomment{codes}

\usepackage[backend=biber,
          style=alphabetic,
          natbib=true,
          abbreviate=true,
          maxbibnames=99,
          sorting=nyt,
          backref=true,
          ]{biblatex}
\AtEveryBibitem{
 \clearfield{doi}
 \clearfield{url}
 \clearfield{issn}
 \clearlist{location}
 \clearfield{month}
 \clearfield{series}

 \ifentrytype{book}{}{
  \clearlist{publisher}
  \clearname{editor}
 }
}
\def\polhk#1{\setbox0=\hbox{#1}{\ooalign{\hidewidth
  \lower1.5ex\hbox{`}\hidewidth\crcr\unhbox0}}} 

\theoremstyle{definition}
\newtheorem{theorem}{Theorem}[section]

\newtheorem{example}[theorem]{Example}
\newtheorem{proposition}[theorem]{Proposition}
\newtheorem{lemma}[theorem]{Lemma}
\newtheorem{definition}[theorem]{Definition}
\newtheorem{remark}[theorem]{Remark}
\numberwithin{equation}{section}

\def\cB{\mathcal{B}}
\def\cD{\mathcal{D}}

\def\cF{\mathcal{F}}

\def\cL{\mathcal{L}}
\def\cN{\mathcal{N}}

\def\bD{\mathbb{D}}
\def\bR{\mathbb{R}}

\def\e{\varepsilon}
\def\E{\mathbb{E}}

\newcommand{\DRL}[2]{{}_t D_{#1}^{#2}}
\newcommand{\lMr}[3]{\:{}_{#1} #2_{#3}}
\newcommand{\Norm}[1]{\left|\left|  #1   \right|\right|}
\newcommand*{\one}{{{\rm 1\mkern-1.5mu}\!{\rm I}}}

\DeclareMathOperator*{\esssup}{ess\,sup}

\newcommand{\R}{\mathbb{R}}

\newcommand{\W}{\dot{W}}
\newcommand{\ud}{\ensuremath{ \mathrm{d}} }
\newcommand{\Ceil}[1]{\left\lceil #1 \right\rceil}
\newcommand{\Floor}[1]{\left\lfloor #1 \right\rfloor}

\newcommand{\FoxH}[5]{H_{#2}^{#1}\left(#3\:\middle\vert\: \begin{subarray}{l}#4\\[0.4em] #5\end{subarray}\right)}
\begin{document}

\title[Moments and asymptotics for a class of SPDEs]{Moments and asymptotics for a class of SPDEs with space-time white noise}

\author[L. Chen]{Le Chen}
\address{L. Chen: Department of Mathematics and Statistics,
  Auburn University, Auburn, USA.}
  \email{\url{le.chen@auburn.edu}}

\author[Y. Guo]{Yuhui Guo}
\address{Y. Guo:
 School of Mathematics, Shandong University, Jinan, China.}
 \email{\url{guoyuhui@mail.sdu.edu.cn}}

\author[J. Song]{Jian Song}
\address{J. Song: Research Center for Mathematics and Interdisciplinary Sciences, Shandong University, Qingdao, China;
 School of Mathematics, Shandong University, Jinan, China.}
 \email{\url{txjsong@sdu.edu.cn}}




\date{}
\begin{abstract}
  %
  %
  In this article, we consider the nonlinear stochastic partial differential equation of fractional order in both space
  and time variables with constant initial condition:
  \begin{equation*}
    \left(\partial^{\beta}_t+\dfrac{\nu}{2}\left(-\Delta\right)^{\alpha / 2}\right) u(t, x)= \: I_{t}^{\gamma}\left[\lambda u(t, x) \dot{W}(t, x)\right]
    \quad t>0,\: x\in\R^d,
  \end{equation*}
  where $\dot{W}$ is space–time white noise, $\alpha>0$, $\beta\in(0,2]$, $\gamma \ge 0$, $\lambda\neq0$ and $\nu>0$.
  The existence and uniqueness of solution in the It\^o-Skorohod sense is obtained under Dalang’s condition. We obtain
  explicit formulas for both the second moment and the second moment Lyapunov exponent. We derive the $p$-th moment  upper bounds and find the matching lower bounds. Our results solve a large class of conjectures regarding the order of
  the $p$-th moment Lyapunov exponents. In particular, by letting $\beta=2$, $\alpha=2$, $\gamma=0$, and $d=1$, we
  confirm the following standing conjecture for the stochastic wave equation:
  \begin{align*}
    t^{-1}\log\E[u(t,x)^p] \asymp p^{3/2}, \quad \text{for $p\ge 2$ as $t\to \infty$.}
  \end{align*}
  The method for the lower bounds is inspired by a recent work by Hu and Wang \cite{hu.wang:21:intermittency}, where the
  authors focus on the space-time colored Gaussian noise.
\end{abstract}

\maketitle
\tableofcontents
\section{Introduction}

Let $\W$ be a {\it space-time white noise}, namely, a centered Gaussian noise with covariance
\begin{align} \label{E:NoiseWW}
  \E[\dot W(t, x) \dot W(s,y) ]=\delta(t-s) \delta(x-y),
\end{align}
where $\delta(\cdot)$ is the Dirac delta function. The following {\it stochastic heat equation}
\begin{equation}\label{E:SHE}
  \text{(SHE)} \quad
  \begin{cases}
    \left(\dfrac{\partial}{\partial t}-\dfrac{\nu}{2}\dfrac{\partial^2}{\partial x^2}\right) u(t, x)= \: \lambda u(t, x) \dot{W}(t, x) , & t>0, x \in \mathbb{R}, \\
    u(0,\cdot)=u_0,                                                                                            &
  \end{cases}
\end{equation}
and {\it stochastic wave equation}
\begin{equation}\label{E:SWE}
  \text{(SWE)} \quad
  \begin{cases}
    \left(\dfrac{\partial^2}{\partial t^2}-\dfrac{\nu}{2}\dfrac{\partial^2}{\partial x^2}\right) u(t, x)= \: \lambda u(t, x) \dot{W}(t, x) , & t>0, x \in \mathbb{R}, \\
    u(0,\cdot)=u_0, \quad \dfrac{\partial}{\partial t} u(0, \cdot)=u_1,
  \end{cases}
\end{equation}
with $\lambda\ne 0$, $\nu>0$, $u_0,u_1\in\R$ are two canonical stochastic partial differential equations. SHE
\eqref{E:SHE} has been widely and extensively studied with many fine properties, among which the probability moments
enjoy the following explicit asymptotics:
\begin{subequations} \label{E:Karkar}
  \begin{align} \label{E:Kardar-p}
    \lim_{p \to \infty} p^{-3}   & \log\E\left[ u(t,x)^p \right] = \frac{\lambda^4}{24},          &  & \text{for all $t>0$ and $x\in\R$ and}, \\
      \lim_{t \to \infty} t^{-1} & \log\E\left[ u(t,x)^p \right] = \frac{1}{24}p(p^2-1)\lambda^4, &  & \text{for all $p\ge 2$ and $x\in\R$}; \label{E:Kardar-t}
  \end{align}
\end{subequations}
see \cite{chen:15:precise} and references therein. One major tool in studying this parabolic
equation is the {\it Feynman-Kac representation} of the moments as being used in {\it ibid}. Note
that the quantity on the right-hand side of \eqref{E:Kardar-t} is called the {\it $p$-th moment
Lyapunov exponent}, which characterizes the intermittency property of the solution; see
\cite{carmona.molchanov:94:parabolic}.

In contrast, much less is known for the hyperbolic counterpart --- SWE \eqref{E:SWE}. The lack of
Feynman-Kac representation for the moments poses a major difficulty in this study. To our best
knowledge, only the following upper bound of the $p$-th moment Lyapunov exponent is known: for some
constant $C>0$,
\begin{align} \label{E:SWE-Ly}
  \limsup_{t \to \infty} t^{-1} \log\E\left[ u(t,x)^p \right]
  \le C p^{3/2}\quad \text{for all $p\ge 2$ and $x\in\R$};
\end{align}
see, e.g., \cite{chen.dalang:15:moment}. It has been long conjectured that the exponent $3/2$ in
\eqref{E:SWE-Ly} is sharp; yet there lacks of a rigorous proof. Dalang and Mueller
\cite{dalang.mueller:09:intermittency} studied the three-dimensional SWE with a Gaussian noise that
is white-in-time and colored-in-space with a bounded correlation function, namely,
\begin{align} \label{E:NoiseWC}
  \E\left[\dot W(t, x) \dot W(s,y) \right]=\delta(t-s) f(x-y),
\end{align}
where $f$ is a nonnegative, nonnegative definite and bounded function. Using an earlier developed
Feynman-Kac-type formula for moments in \cite{dalang.mueller.ea:08:feynman-kac-type}, they
established the following large-time asymptotics:
\begin{align} \label{E:DM-t}
      C_1 p^{4/3}
  \le \liminf_{t\to\infty} \frac{\log\E\left[|u(t,x)|^p\right]}{t}
  \le \limsup_{t\to\infty} \frac{\log\E\left[|u(t,x)|^p\right]}{t}
  \le C_2 p^{4/3},
\end{align}
for all $p\ge 2$ and  $x\in\R^3$; see Theorem 1.1 ({\it ibid.}). To obtain the lower bound in
\eqref{E:DM-t},  their arguments crucially depend on the property that one can find a small
indicator function below $f(x)$ near the origin, i.e.,  $c 1_{\{|x|\le r\} } \le f(x)$ for all
$x\in\R^3$. This requirement prevents the application to the space-time white noise case. Recently,
Hu and Wang \cite{hu.wang:21:intermittency} obtained the matching lower and upper $p$-th moment
Lyapunov exponents for a wide range of SPDEs with space-time colored Gaussian noise
\begin{align} \label{E:NoiseCC}
  \E\left[\dot W(t, x) \dot W(s,y) \right]=\gamma(t-s) \Lambda(x-y).
\end{align}
Certain choices or limits of the parameters of the noise in \eqref{E:NoiseCC} may suggest the
correct moment asymptotics for SWE \eqref{E:SWE}. However, just as the SHE case, the SWE with space-time white noise needs a separate treatment (see Remark \ref{R:Hu-Wang} for more details). One
of the major contributions of this paper is to carry out such arguments and confirm the conjecture
about the moment asymptotics of SWE \eqref{E:SWE} by showing that if $u_0>0$ and $ u_1\ge 0$, then
\begin{subequations} \label{E:AsymSWE}
  \begin{align} \label{E:AsymSWE-t}
         C_1 p^{3/2}
     \le & \liminf_{t\to \infty} \frac{\log \E[u(t,x)^p]}{t}
     \le   \limsup_{t\to \infty} \frac{\log \E[u(t,x)^p]}{t}
     \le C_2 p^{3/2}, & p\ge 2, \\
         C_3 \: t
     \le & \liminf_{p\to \infty} \frac{\log \E[u(t,x)^p]}{p^{3/2}}
     \le   \limsup_{p\to \infty} \frac{\log \E[u(t,x)^p]}{p^{3/2}}
     \le C_4 \: t, & t>0, \label{E:AsymSWE-p}
  \end{align}
\end{subequations}
where $C_1,\cdots, C_4$ are some nonnegative constants that do not depend on $t$ and $p$. \bigskip

It turns out that the method that we use to resolve the above conjecture can be applied to a much
wider class of stochastic partial differential equations (SPDEs). Indeed, in this paper, we will
study the following stochastic fractional diffusion equation with both SHE \eqref{E:SHE} and SWE
\eqref{E:SWE} as two special cases:
\begin{equation}\label{E:fde}
  \begin{cases}
    \left(\partial_t^{\beta}+\dfrac{\nu}{2}\left(-\Delta\right)^{\alpha / 2}\right) u(t, x)= \: I_{t}^{\gamma}\left[\lambda u(t, x) \dot{W}(t, x)\right] , & t>0, x \in \mathbb{R}^{d},   \\
    u(0,\cdot)=u_0,                                                                                                                                      & \text { if } \beta \in(0,1], \\
    u(0,\cdot)=u_0, \quad \dfrac{\partial}{\partial t} u(0, \cdot)=u_1,                                                                                  & \text { if } \beta \in(1,2],
  \end{cases}
\end{equation}
where $\dot W$ is space-time white noise, $\left(-\Delta\right)^{\alpha / 2}$ is the fractional
Laplacian and
\begin{align*}
  \alpha>0, \quad
  \beta\in (0,2], \quad
  \gamma\ge 0, \quad
  \lambda\ne 0, \quad
  \nu >0, \quad
  u_0, u_1\in \R.
\end{align*}
The symbol $\partial_t^\beta$ denotes the {\em Caputo fractional differential} operator of order
$\beta> 0$:
\begin{align*}
  \partial_t^{\beta} f(t):=
  \begin{cases}
    \displaystyle
    \dfrac{1}{\Gamma(n-\beta)}\int_{0}^{t}\frac{f^{(n)}(\tau)}{(t-\tau)^{\beta+1-n}}\ud\tau, & \mbox{if } \beta \neq n, \\[1em]
    \dfrac{\ud^n}{\ud t^n}f(t),                                                              & \mbox{if }  \beta = n,
  \end{cases}
\end{align*}
where $n=\Ceil{\beta}$ is the smallest integer that is not smaller than $\beta$ (i.e.,
$\Ceil{\cdot}$ is the ceiling function), and $\Gamma(x)$ is the {\it gamma function}. We use
$I_t^\gamma$ to refer to the {\it Riemann-Liouville integral} in the time variable to the right of
zero $I_{0+}^\gamma$; see Definition \ref{def-R-I}.

The SPDE \eqref{E:fde} is interpreted as the following integral equation:
\begin{equation}\label{E:fde-Int}
  u(t,x) = J_0(t,x)+ \lambda\int_{0}^{t}\int_{\bR^d}p(t-s,x-y) u(s,y)W(\ud s,\ud y),
\end{equation}
where $J_0(t,x)$ is the solution to the homogeneous equation (see \eqref{E:J0(t,x)} and
\eqref{E:J0(t)} below), $p(t,x)$ is the underlying fundamental solution (see \eqref{E:p}), and the
stochastic integral refers to the {\it Walsh} or {\it Skorohod integral}. Set the following four
constants:
\begin{align}
\begin{array}{lcl} \label{E:theta}
  \theta := 2(\beta+\gamma)-2-\beta d/\alpha,                                                                 & \quad \quad & t_p:= p^{1+ 1 /(1+\theta)} t, \\ [0.5em]
  \displaystyle \Theta := (2\pi)^{-d}\int_{\R^d} E^2_{\beta, \beta+\gamma}(-2^{-1} \nu |\xi|^\alpha) \ud \xi, & \quad \quad & \widehat{t}:= \Theta\: \Gamma\left(\theta+1\right)t^{\theta+1},
\end{array}
\end{align}
where the function $E_{a,b}(z)$ is the {\em two parameter Mittag-Leffler function} of two parameters
(see, e.g., \cite[Section 1.8]{kilbas.srivastava.ea:06:theory}), i.e., for $a, b>0$,
\begin{align} \label{E:ML}
  E_{a, b}(z):=\sum_{k=0}^{\infty}\frac{z^k}{\Gamma(a k+b)}, \quad z\in \mathbb{C}.
\end{align}
We use the convention $E_a(\cdot):=E_{a,1}(\cdot)$.

We will prove in Theorem \ref{T:Exist} below that under {\it Dalang's condition}
\begin{equation}\label{E:Dalang'}
  \begin{cases}
    \displaystyle d<2\alpha+\frac{\alpha}{\beta}\min\{2\gamma-1,0\}, & \text{ if } \beta\in(0,2), \vspace{0.2cm} \\
    d<\alpha\min\{2, 1+\gamma\},                                     & \text{ if } \beta=2,
  \end{cases}
\end{equation}
there exists an unique random field solution $u(t,x)$ with finite $p$-th moment for all $p\ge
2$, $t>0$ and $x\in\R^d$. It is an easy exercise to check that Dalang's condition \eqref{E:Dalang'}
implies that $\theta>-1$ and $\Theta<\infty$, so that all constants in \eqref{E:theta} are
well-defined. The aim of this paper is to establish the following theorem, which gives the exact
formula for the second moment and the sharp moment asymptotics both in terms of $t$ and $p\ge 2$.

\begin{theorem}\label{T:fde}
  Suppose that Dalang's condition \eqref{E:Dalang'} is satisfied and let $u(t,x)$ be the solution to
  \eqref{E:fde}. Recall that the quantities $\theta$, $\Theta$, $t_p$ and $\widehat{t}$ are defined
  in \eqref{E:theta}. The $p$-th moment satisfies the following properties:

  \noindent (a) When $p=2$,
  \begin{equation}\label{E:SecMom}
    \E\left[u^2(t,x)\right] =
    \left\{
    \begin{array}{rc}
         u_0^2\:    E_{\theta+1}  \left(\lambda^2 \widehat{t}\: \right) & \qquad \text{if $\beta\in(0, 1]$,} \vspace{0.8em} \\
         u_0^2\:    E_{\theta+1}  \left(\lambda^2 \widehat{t}\: \right) &                                                   \\
      + 2u_0u_1t\:  E_{\theta+1,2}\left(\lambda^2 \widehat{t}\: \right) & \qquad \text{if $\beta\in(1,2]$,}                 \\
      + 2u_1^2t^2\: E_{\theta+1,3}\left(\lambda^2 \widehat{t}\: \right) &
    \end{array}\right.
  \end{equation}
  for all $t>0$ and $x\in\R^d$. As a consequence,
  \begin{equation}\label{E:2nd-Ly}
    \lim_{t\rightarrow\infty}\frac{1}{t}\log\E\left[u(t,x)^2\right]
    = \left(\lambda^2\Theta \Gamma(\theta+1)\right)^{1/(\theta+1)},
    \quad \text{for all $x\in\R^d$}.
  \end{equation}

  \noindent (b) For any $p\ge 2$,
  \begin{align} \label{E:p-mom}
    \Norm{u(t,x)}_p^2 \le
    \left\{
      \begin{array}{ rc }
           2 u_0^2\:    E_{\theta+1}  \left(8p \lambda^2 \widehat{t}\:\right) & \qquad \text{if $\beta\in(0,1]$,} \vspace{0.8em} \\
           2 u_0^2\:    E_{\theta+1}  \left(8p \lambda^2 \widehat{t}\:\right) &                                                  \\
         + 4 u_0u_1t\:  E_{\theta+1,2}\left(8p \lambda^2 \widehat{t}\:\right) & \qquad \text{if $\beta\in (1,2]$,}               \\
         + 4 u_1^2t^2\: E_{\theta+1,3}\left(8p \lambda^2 \widehat{t}\:\right) &                                                  \\
      \end{array}
    \right.
  \end{align}
  for all $t>0$ and $x\in\R^d$. As a consequence,
  \begin{align} \label{E:upper-tplim}
    \limsup_{t_p\rightarrow\infty} t_p^{-1} \log\E[|u(t,x)|^p]
    \le \frac{1}{2} \left(8 \lambda^2 \Theta \Gamma(\theta+1)\right)^{1/(\theta+1)}.
  \end{align}
  In particular, by freezing $p>2$ or $t>0$, we have the following two asymptotics:
  \begin{subequations} \label{E:upper-lim}
  \begin{gather}\label{E:upper-tlim}
        \limsup_{t\rightarrow\infty} t^{-1} \log\E[|u(t,x)|^p]
    \le \frac{1}{2} \left(8 \lambda^2 \Theta \Gamma(\theta+1)\right)^{1/(\theta+1)}
        p^{1+1/(\theta+1)}, \\
        \limsup_{p\rightarrow\infty} p^{-\left(1+1/(\theta+1)\right)}\log\E[|u(t,x)|^p]
    \le \frac12\left( 8 \lambda^2 \Theta \Gamma(\theta+1)\right)^{1/(\theta+1)} t.
    \label{E:upper-plim}
  \end{gather}
  \end{subequations}

  \noindent (c) If in addition we have
  \begin{enumerate}
    \item either $\beta\in(0,2)$ and the fundamental function $p(t,x)$ is nonnegative or
      $\alpha=\beta=2$ and $\gamma=0$; and
    \item the initial position $u_0$ is strictly positive and the initial velocity $u_1$ is
      nonnegative,
  \end{enumerate}
  then by treating $t_p$ defined in \eqref{E:theta} as a function from $\R_+\times 2\mathbb{N}$ to
  $\R_+$, for all $x\in\R^d$,
  \begin{align} \label{E:lower-tplim}
    C:=\liminf_{t_p\rightarrow\infty} t_p^{-1} \log\E\left[u(t,x)^p\right] >0.
  \end{align}
  In particular, by freezing an even integer $p\ge 2$ or $t>0$, with the same constant $C$ as in
  \eqref{E:lower-tplim}, we have the following two asymptotics:
  \begin{subequations} \label{E:lower-lim}
  \begin{gather}\label{E:lower-tlim}
        \liminf_{t\rightarrow\infty} t^{-1} \log\E\left[u(t,x)^p\right]
    \ge C p^{1+1/(\theta+1)}, \\
        \liminf_{p\rightarrow\infty} p^{-\left(1+1/(\theta+1)\right)}\log\E\left[u(t,x)^p\right]
    \ge C t.
    \label{E:lower-plim}
  \end{gather}
  \end{subequations}
\end{theorem}

As applications of Theorem \ref{T:fde}, in Section \ref{S:Example} we shall revisit some SPDEs which have attracted considerable attention in the literature and calculate sharp moment asymptotics for the solutions.

 To conclude the introduction, we highlight some of the contributions of this paper:
\begin{enumerate}[wide=1em]
  \item The conjecture on the moment asymptotics of SWE \eqref{E:SWE} is solved; see     \eqref{E:AsymSFWE} and Example~\ref{Ex:SWE}.

  \item For the solutions of  a large class of SPDEs, explicit representations  \eqref{E:SecMom} for the second moments and \eqref{E:2nd-Ly} for the second moment Lyapunov exponents are obtained;  the asymptotic behavior of $p$-th moments is characterized sharply by the upper bounds \eqref{E:upper-tplim} and the lower bounds \eqref{E:lower-tplim}.  Regarding the quantities obtained in Theorem~ \ref{T:fde}, as will be shown in Section~\ref{S:Example}, some of them recover known results for SPDEs with some specific parameters $(\alpha, \beta, \gamma, d)$  in the literature, while, to our best knowledge,  most of them (in particular the lower bounds for $p$-th moments) are new. Moreover, the quantities that characterize the asymptotics of the solutions to \eqref{E:fde} depend on the parameters $\left(\alpha,\beta,\gamma,d\right)$ in an interesting way (see also the figures in Section~\ref{S:Example} for an illustration), which may relate to physical phenomena and desire further investigation.

  \item For the fundamental solution of \eqref{E:fde}, we extend the results in
    \cite{chen.hu.ea:19:nonlinear} from $\alpha\in (0,2]$ to  all $\alpha>0$ (see Appendix \ref{S:Funamental}). As a consequence,  Dalang's condition \eqref{E:Dalang'} allows to consider SPDE \eqref{E:fde} in high dimension $d$,  if $\alpha$ is sufficiently big.

\end{enumerate}

%
%
%
%
%
%


\bigskip

The paper is organized as follows: We first list some examples and give some discussions in Section
\ref{S:Example}. Then in Section \ref{S:Exist}, we establish the existence and uniqueness of the
solution in a slightly more general setting. The second moment formula and the $p$-th moment upper
bounds are obtained in Section \ref{S:Upper}; while the lower $p$-th moment bounds are derived in
Section \ref{S:Lower}. Some preliminaries about the fractional calculus and Mittag-Leffler functions
are given in Appendix~\ref{S:Prelim}. In Appendix~\ref{S:Lemmas} we prove some technical lemmas used
in this paper. Finally, in Appendix \ref{S:Funamental}, we derive the fundamental solutions under
the settings of $\alpha>0$,  $\beta\in(0,2]$ and $\gamma\ge 0$. \bigskip

Throughout the paper, $\Norm{\cdot}_p$ denotes the probability $L^p(\Omega)$-norm. We use $B_r(x)$
to denote an open ball centered at $x\in\R^d$ with radius $r$, i.e., $B_r(x)=\left\{x\in\R^d:\:
|x|<r\right\}$, where $|x|=\sqrt{x_1^2+\cdots+x_d^2}$.  For $a\in\R$, $\Ceil{a}$ (resp. $\Floor{a}$)
is the smallest (resp. largest) integer that is not smaller (resp. larger) than $a$, i.e., the
ceiling (reps. floor) function. We use the convention $\mathbb{N} = \{1,2,\cdots\}$.

\section{Examples and discussions} \label{S:Example}
In this section, we give some concrete examples for the main result Theorem \ref{T:fde}. We will use $C_1,\cdots, C_4$ to denote generic
constants that do not depend on $t$ and $p$.

\begin{example}[SHE] \label{Ex:SHE}
  When $\alpha = 2$, $\beta=1$, $\gamma=0$ and $d=1$, equation \eqref{E:fde} reduces to SHE
  \eqref{E:SHE}. In this case, Dalang's condition \eqref{E:Dalang'} is satisfied and
  \begin{align*}
    \theta      = -1/2, \quad
    \Theta      = \frac{1}{2\pi} \int_\R e^{-\nu |\xi|^2}\ud \xi
                = \frac{1}{\sqrt{4\pi\nu}}, \quad
    \widehat{t} = \frac{\sqrt{t}}{\sqrt{4\nu}}, \quad \text{and} \quad
    t_p         = p^3 t.
  \end{align*}
  \begin{enumerate}[wide=1em]
    \item {\it Second moment formula:} The second moment formula \eqref{E:SecMom} reduces to
      \begin{equation} \label{E:2ndSHE}
         \E[u^2(t,x)]
         = u_0^2\: E_{1/2}\left(\frac{\lambda^2}{\sqrt{4\nu}}t^{1/2}\right)
         = 2u_0^2\:e^{\frac{\lambda^4t}{4\nu}}\Phi\left(\frac{\lambda^2t^{1/2}}{\sqrt{2\nu}}\right),
      \end{equation}
      where we have applied \eqref{E:ML-ex}. This formula recovers the one obtained in
      \cite{chen.dalang:15:moments} as a special case; see Corollary 2.5. {\it ibid.} \smallskip
    \item {\it Second moment Lyapunov exponent:} From \eqref{E:2ndSHE}, we immediately see that
      \begin{align} \label{E:LySHE}
        \lim_{t\to \infty} t^{-1} \log \E[u(t,x)^2] = \frac{\lambda^4}{4\nu}.
      \end{align}
      Results obtained by Balan and Song \cite{balan.song:19:second} also reduce to this special
      case with the exact second moment Lyapunov exponent being equal to $1/4$ (where $\lambda=1$
      and $\nu=1$); see Remark 1.6 {\it ibid.} \smallskip
    \item {\it Moment asymptotics:} Because the heat kernel is nonnegative, we can combine the
      asymptotics in \eqref{E:upper-lim} and \eqref{E:lower-lim} to conclude that
      \begin{subequations} \label{E:AsymSHE}
        \begin{align} \label{E:AsymSHE-t}
             & C_1 p^3
           \le \liminf_{t\to \infty} t^{-1} \log \E[u(t,x)^p]
           \le \limsup_{t\to \infty} t^{-1} \log \E[u(t,x)^p]
           \le C_2 p^{3}, & p\ge 2, \\
             & C_3\: t
           \le \liminf_{p\to \infty} p^{-3} \log \E[u(t,x)^p]
           \le \limsup_{p\to \infty} p^{-3} \log \E[u(t,x)^p]
           \le C_4\: t,   & t>0. \label{E:AsymSHE-p}
        \end{align}
      \end{subequations}
      These asymptotics are consistent with the exact asymptotics obtained by X. Chen; see
      \cite[Theorem 1.1, Remark 3.1]{chen:15:precise}.
  \end{enumerate}
\end{example}

\begin{example}[SWE] \label{Ex:SWE}
  When $\alpha=2$, $\beta=2$, $\gamma=0$ and $d=1$, equation \eqref{E:fde} reduces to SWE
  \eqref{E:SWE}. In this case, $J_0(t) = u_0+u_1 t$, Dalang's condition \eqref{E:Dalang'} is
  satisfied, and
  \begin{align*}
    \theta      = 1, \quad
    \Theta      = \frac{1}{\pi}\int_0^\infty \frac{\sin\left(\sqrt{\nu/2}\: \xi\right)^2}{\left(\nu/2\right)\xi^2}\ud \xi
                = \frac{1}{\sqrt{2\nu}}, \quad
    \widehat{t} = \frac{t^2}{\sqrt{2\nu}}, \quad\text{and}\quad
    t_p         = p^{3/2}t,
  \end{align*}
  where $\Theta$ is obtained via Lemma \ref{L:sin}.
  \begin{enumerate}[wide=1em]
    \item {\it Second moment formula:} The second moment formula \eqref{E:SecMom} becomes
      \begin{align} \notag
        \E\left[u^2(t,x)\right]
        &= u_0^2E_{2}\left(\frac{\lambda^2t^2}{\sqrt{2\nu}}\right)+2u_0u_1tE_{2,2}\left(\frac{\lambda^2t^2}{\sqrt{2\nu}}\right)
        + 2u_1^2t^2E_{2,3}\left(\frac{\lambda^2t^2}{\sqrt{2\nu}}\right).
      \end{align}
      Now using \eqref{E:ML_a+b} and the special cases in \eqref{E:ML-ex}, we see that
      \begin{align} \label{E:2ndSWE}
      \begin{split}
        \E\left[u^2(t,x)\right]
        = - \frac{2^{3/2} \nu^{1/2} u_1^2}{\lambda^2}
        + \left(u_0^2+\frac{2^{3/2} \nu^{1/2} u_1^2}{\lambda^2}\right)\cosh\left(\frac{|\lambda| t}{(2\nu)^{1/4}}\right)\: &  \\
        + \frac {2^{5/4}\nu^{1/4}u_0u_1}{|\lambda|}\sinh\left(\frac{|\lambda| t}{(2\nu)^{1/4}}\right),                     &
      \end{split}
      \end{align}
      which recovers \cite[Corollary 1.1]{chen.dalang:15:moment} \footnote{There is a typo in the
        paper \cite{chen.dalang:15:moment} where the fundamental solution for the wave kernel should
        be $\frac{1}{2\kappa} \one_{[-\kappa t,\kappa t]}(x)$ instead of $\frac{1}{2} \one_{[-\kappa
        t,\kappa t]}(x)$; see the equation after (1.2) {\it ibid.} If one sets $\kappa=1$ {\it
      ibid.} or equivalently sets $\nu=2$ in the current paper, the results should coincide.}.
      \smallskip
    \item {\it Second moment Lyapunov exponent:} From \eqref{E:2ndSWE}, we immediately see that
      \begin{align} \label{E:LySWE}
        \lim_{t\to \infty} t^{-1}\E[|u(t,x)|^2] = \frac{|\lambda|}{\left(2\nu\right)^{1/4}},
      \end{align}
      which has also been obtained by Balan and Song in \cite[Remark 1.6]{balan.song:19:second}.
      \smallskip
    \item {\it Moment asymptotics:} Since the fundamental solution is nonnegative, combining the
      asymptotics in \eqref{E:upper-lim} and \eqref{E:lower-lim} shows  \eqref{E:AsymSWE}. The upper
      bound in the large-time asymptotics \eqref{E:AsymSWE-t} is consistent with \cite[Theorem 2.7
      ]{chen.dalang:15:moment}.
  \end{enumerate}
\end{example}

\begin{example}[SFHE] \label{Ex:SFHE}
  When $\alpha>0$, $\beta=1$, $\gamma=0$ and $d=1$, equation \eqref{E:fde} becomes the following
  one-dimensional stochastic fractional heat equation:
\begin{equation}\label{E:SFHE}
  \text{(SFHE)} \quad
  \begin{cases}
    \left(\dfrac{\partial}{\partial t}+\dfrac{\nu}{2}\left(-\Delta\right)^{\alpha / 2}\right) u(t, x)= \: \lambda u(t, x) \dot{W}(t, x) , & t>0, x \in \mathbb{R}, \\
    u(0,\cdot)=u_0.
  \end{cases}
\end{equation}
  In this case, Dalang's condition \eqref{E:Dalang'} becomes $\alpha>d=1$. For for $\alpha>1$, we
  have
  \begin{align*}
    \theta              = & -\frac{1}{\alpha}, \\
    \Theta_{\alpha,\nu} = & \frac{1}{2\pi} \int_\R e^{-\nu|\xi|^\alpha}\ud\xi
                        = \frac{\Gamma\left(1+1/\alpha\right)}{\nu^{1/\alpha}\pi}, \\
    \widehat{t}         = & \frac{\Gamma\left(1-1/\alpha\right)\Gamma\left(1+1/\alpha\right)}{\nu^{1/\alpha}\pi} t^{1-1/\alpha}
                        = \left(\nu^{1/\alpha}\alpha \sin\left(\pi / \alpha\right)\right)^{-1} t^{1-1/\alpha},\\
    t_p                 = & p^{1 + \alpha/(\alpha-1)} t,
  \end{align*}
  where in computing $\widehat{t}$ we have use the reflection formula \eqref{E:Reflection}.
  \begin{enumerate}[wide=1em]
    \item {\it Second moment formula:} The second moment formula \eqref{E:SecMom} reduces to
      \begin{equation} \label{E:2ndSFHE}
         \E[u^2(t,x)]
         = u_0^2\: E_{1-1/\alpha}\left(\frac{\lambda^2}{\nu^{1/\alpha}\alpha \sin\left(\pi / \alpha\right)} t^{1-1/\alpha}\right).
      \end{equation}
      In \cite{chen.dalang:15:moments}, this equation with $\alpha\in (1,2]$ has been studied with a
      non-homogeneous initial conditions. \smallskip
    \item {\it Second moment Lyapunov exponent:} From \eqref{E:2ndSFHE}, we immediately see that
      \begin{align} \label{E:LySFHE}
           \lim_{t\to \infty} t^{-1}\E[|u(t,x)|^2]
           = \left(\frac{\lambda^2}{\nu^{1/\alpha}\alpha \sin\left(\pi / \alpha\right)} \right)^{\alpha/(\alpha-1)},
         \qquad \alpha>1;
      \end{align}
      see Figure \ref{F:SFHE} for a plot of this expression as a function of $\alpha$.
    \item {\it Moment asymptotics:} If $\alpha\in (1,2]$, the fundamental solution is nonnegative
      (see Remark \ref{R:Nonneg}), then the asymptotics in \eqref{E:upper-lim} and
      \eqref{E:lower-lim} reduce to
      \begin{subequations} \label{E:AsymSFHE}
        \begin{align} \label{E:AsymSFHE-t}
               C_1 p^{\frac{2\alpha-1}{\alpha-1}}
           \le & \liminf_{t\to \infty} \frac{\log \E[u(t,x)^p]}{t}
           \le   \limsup_{t\to \infty} \frac{\log \E[|u(t,x)|^p]}{t}
           \le C_2 p^{\frac{2\alpha-1}{\alpha-1}}, & p\ge 2, \\
               C_3 \: t
           \le & \liminf_{p\to \infty} \frac{\log \E[u(t,x)^p]}{p^{\frac{2\alpha-1}{\alpha-1}}}
           \le   \limsup_{p\to \infty} \frac{\log \E[|u(t,x)|^p]}{p^{\frac{2\alpha-1}{\alpha-1}}}
           \le C_4 \: t, & t>0. \label{E:AsymSFHE-p}
        \end{align}
      \end{subequations}
      The upper bound in the large-time asymptotics \eqref{E:AsymSFHE-t} is consistent with
      \cite[Theorem 3.4]{chen.dalang:15:moments}. In \cite[Theorem 1.1]{chen.hu.ea:18:temporal},
      Chen {\it et al} obtained the exact large-time asymptotics when the noise is colored in the
      sense of \eqref{E:NoiseCC}. Note also that only the lower bounds in \eqref{E:AsymSFHE-t} and
      \eqref{E:AsymSFHE-p} require the nonnegativity of the fundamental solution. The upper bounds
      still hold true for all $\alpha>1$.
  \end{enumerate}
\end{example}

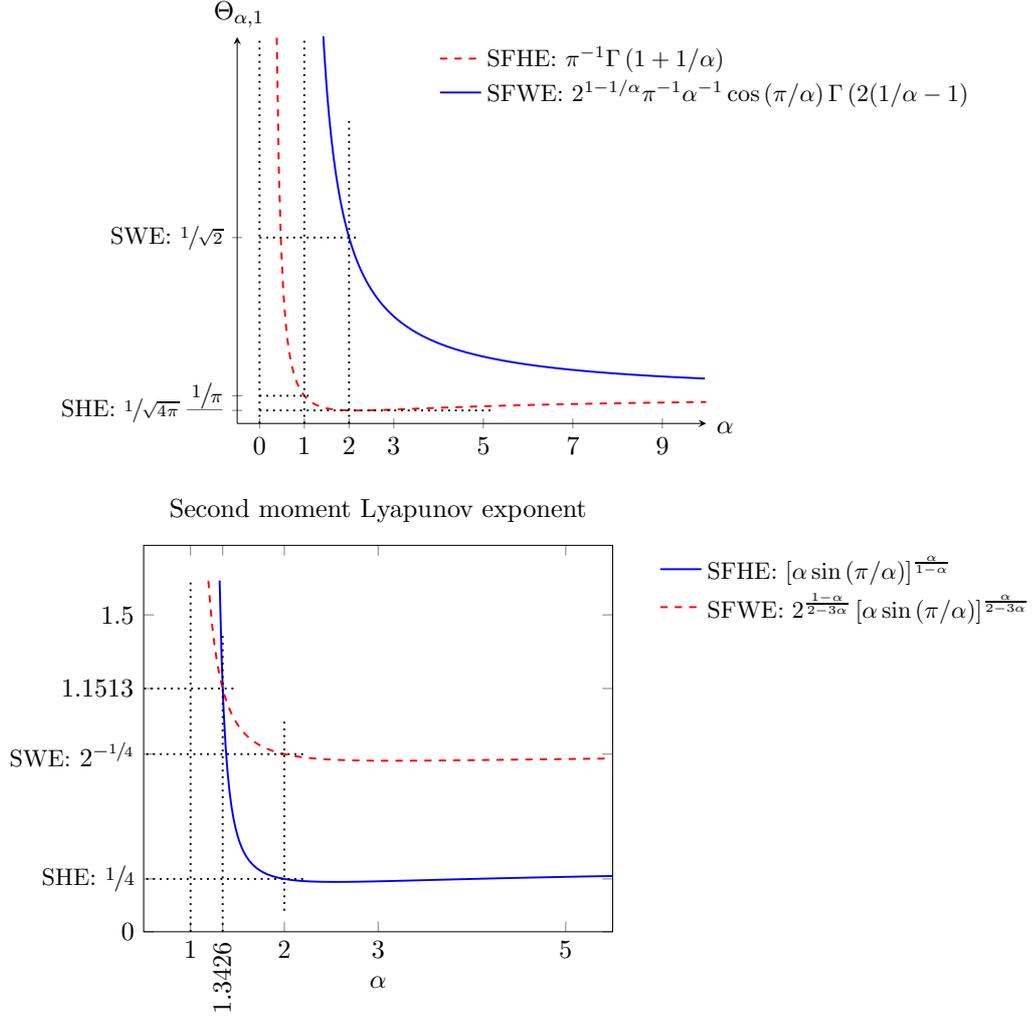
\begin{figure}[htpb]
  \centering
  \begin{center}
    \begin{tikzpicture}[scale=0.9, transform shape, x=3em, y=5em]
      \tikzset{>=latex}
      \begin{axis}[
        ymax=1.2,
        axis lines = left,
        xmin=-0.5,
        ytick={      0.31813,          0.282095,                                     0.707 },
        yticklabels={$\sfrac{1}{\pi}$, {\small SHE: $\sfrac{1}{\sqrt{4\pi}}$ -----}, {\small SWE: $\sfrac{1}{\sqrt{2}}$}},
        xtick={0,1, 2, 3, 5,7,9},
        xlabel={$\alpha$}, xlabel style={at=(current axis.right of origin), anchor=center, xshift=2.5em, yshift=1.2em},
        ylabel={$\Theta_{\alpha,1}$},ylabel style={at=(current axis.above origin), anchor=south, rotate = -90, xshift=2.5em, yshift=1.5em},
        legend style={at={(1.0,+1.00)}, anchor=north, legend cell align=left, draw=none},
        ]
        \addplot[domain=0.3:10, red, dashed, thick] coordinates {
( 0.329,2.0074879503560927 )
( 0.379,1.2363683056054173 )
( 0.42900000000000005,0.8821615055373297 )
( 0.47900000000000004,0.6913083198933705 )
( 0.529,0.5767483830081792 )
( 0.579,0.5025097545374081 )
( 0.629,0.45161032351974006 )
( 0.679,0.41518861083662834 )
( 0.7290000000000001,0.38824701366795894 )
( 0.779,0.36778728269804356 )
( 0.829,0.35191802416601115 )
( 0.879,0.3393955536565757 )
( 0.929,0.32937316869948013 )
( 0.9790000000000001,0.3212571925661031 )
( 1.0290000000000001,0.31462071478514525 )
( 1.079,0.30914996738786316 )
( 1.129,0.3046099059855785 )
( 1.179,0.30082148936659414 )
( 1.229,0.29764630531583103 )
( 1.2790000000000001,0.29497593576142755 )
( 1.329,0.29272445398034036 )
( 1.379,0.2908230369150719 )
( 1.429,0.2892160339825739 )
( 1.479,0.28785805673196757 )
( 1.5290000000000001,0.28671179561695354 )
( 1.579,0.28574636233601963 )
( 1.629,0.28493601721418027 )
( 1.679,0.28425918219215657 )
( 1.729,0.2836976681031991 )
( 1.7790000000000001,0.2832360644366572 )
( 1.829,0.2828612535238862 )
( 1.879,0.2825620208717554 )
( 1.929,0.28232874042822115 )
( 1.979,0.2821531187107501 )
( 2.0290000000000004,0.2820279855187692 )
( 2.079,0.2819471217699424 )
( 2.129,0.2819051171150199 )
( 2.1790000000000003,0.2818972515864427 )
( 2.229,0.28191939675668015 )
( 2.2790000000000004,0.2819679328205133 )
( 2.329,0.2820396787416957 )
( 2.3790000000000004,0.2821318331703325 )
( 2.4290000000000003,0.28224192428112055 )
( 2.479,0.2823677670327427 )
( 2.5290000000000004,0.28250742662656725 )
( 2.579,0.2826591871644889 )
( 2.6290000000000004,0.2828215246835532 )
( 2.6790000000000003,0.2829930838883091 )
( 2.7290000000000005,0.2831726580178887 )
( 2.7790000000000004,0.2833591713792163 )
( 2.829,0.28355166415487737 )
( 2.8790000000000004,0.2837492791574485 )
( 2.9290000000000003,0.28395125025420015 )
( 2.9790000000000005,0.2841568922291642 )
( 3.0290000000000004,0.2843655918853028 )
( 3.079,0.2845768002192756 )
( 3.1290000000000004,0.28479002552616556 )
( 3.1790000000000003,0.2850048273123579 )
( 3.2290000000000005,0.2852208109122859 )
( 3.2790000000000004,0.2854376227195236 )
( 3.329,0.2856549459551956 )
( 3.3790000000000004,0.28587249690726396 )
( 3.4290000000000003,0.286090021583253 )
( 3.4790000000000005,0.28630729272664796 )
( 3.5290000000000004,0.28652410715375926 )
( 3.579,0.28674028337346014 )
( 3.6290000000000004,0.2869556594570273 )
( 3.6790000000000003,0.28717009112946 )
( 3.7290000000000005,0.28738345005723365 )
( 3.7790000000000004,0.2875956223105306 )
( 3.829,0.28780650698066906 )
( 3.8790000000000004,0.28801601493577195 )
( 3.9290000000000003,0.2882240676997345 )
( 3.9790000000000005,0.28843059644130764 )
( 4.029,0.28863554106164424 )
( 4.079,0.2888388493699983 )
( 4.1290000000000004,0.2890404763384349 )
( 4.179,0.2892403834274403 )
( 4.229,0.289438537975222 )
( 4.279,0.2896349126442817 )
( 4.329,0.2898294849195453 )
( 4.379,0.2900222366529484 )
( 4.429,0.2902131536499221 )
( 4.479,0.29040222529370413 )
( 4.529,0.29058944420382893 )
( 4.579,0.2907748059255264 )
( 4.629,0.2909583086470974 )
( 4.679,0.2911399529426308 )
( 4.729,0.29131974153769447 )
( 4.779,0.2914976790958684 )
( 4.829,0.2916737720242002 )
( 4.879,0.2918480282958522 )
( 4.929,0.29202045728837805 )
( 4.979,0.2921910696362168 )
( 5.029,0.2923598770961312 )
( 5.079,0.29252689242443347 )
( 5.1290000000000004,0.29269212926495647 )
( 5.179,0.2928556020468208 )
( 5.229,0.29301732589114116 )
( 5.279,0.2931773165258912 )
( 5.329,0.2933355902082203 )
( 5.3790000000000004,0.2934921636535781 )
( 5.429,0.2936470539710619 )
( 5.479,0.2938002786044557 )
( 5.529,0.2939518552784741 )
( 5.579,0.2941018019497712 )
( 5.6290000000000004,0.2942501367623102 )
( 5.679,0.29439687800672804 )
( 5.729,0.2945420440833572 )
( 5.779,0.2946856534686012 )
( 5.829,0.2948277246843811 )
( 5.8790000000000004,0.29496827627039973 )
( 5.929,0.2951073267589871 )
( 5.979,0.2952448946523147 )
( 6.029,0.2953809984017814 )
( 6.079,0.29551565638939176 )
( 6.1290000000000004,0.29564888691096197 )
( 6.179,0.29578070816100227 )
( 6.229,0.2959111382191382 )
( 6.279,0.29604019503794243 )
( 6.329,0.29616789643206154 )
( 6.3790000000000004,0.2962942600685307 )
( 6.429,0.2964193034581773 )
( 6.479,0.2965430439480238 )
( 6.529,0.2966654987146064 )
( 6.579,0.29678668475813375 )
( 6.6290000000000004,0.29690661889741476 )
( 6.679,0.29702531776549196 )
( 6.729,0.2971427978059199 )
( 6.779,0.2972590752696351 )
( 6.829,0.29737416621236623 )
( 6.8790000000000004,0.29748808649253944 )
( 6.929,0.29760085176963463 )
( 6.979,0.29771247750295504 )
( 7.029,0.2978229789507731 )
( 7.079,0.2979323711698194 )
( 7.1290000000000004,0.2980406690150842 )
( 7.179,0.2981478871399034 )
( 7.229,0.29825403999630284 )
( 7.279,0.2983591418355766 )
( 7.329,0.2984632067090781 )
( 7.3790000000000004,0.2985662484692028 )
( 7.429,0.2986682807705443 )
( 7.479,0.29876931707120624 )
( 7.529,0.2988693706342547 )
( 7.579,0.29896845452929566 )
( 7.6290000000000004,0.2990665816341648 )
( 7.679,0.29916376463671673 )
( 7.729,0.2992600160367029 )
( 7.779,0.29935534814772646 )
( 7.829,0.29944977309926624 )
( 7.8790000000000004,0.29954330283875935 )
( 7.929,0.299635949133735 )
( 7.979,0.299727723573992 )
( 8.029,0.29981863757381316 )
( 8.079,0.2999087023742102 )
( 8.129000000000001,0.29999792904519285 )
( 8.179,0.3000863284880582 )
( 8.229000000000001,0.3001739114376939 )
( 8.279,0.30026068846489185 )
( 8.329,0.30034666997866777 )
( 8.379000000000001,0.3004318662285838 )
( 8.429,0.30051628730706914 )
( 8.479000000000001,0.3005999431517373 )
( 8.529000000000002,0.3006828435476967 )
( 8.579,0.3007649981298513 )
( 8.629000000000001,0.30084641638519 )
( 8.679,0.30092710765506253 )
( 8.729000000000001,0.3010070811374391 )
( 8.779000000000002,0.30108634588915356 )
( 8.829,0.301164910828127 )
( 8.879000000000001,0.30124278473557203 )
( 8.929,0.3013199762581753 )
( 8.979000000000001,0.30139649391025725 )
( 9.029000000000002,0.3014723460759096 )
( 9.079,0.3015475410111077 )
( 9.129000000000001,0.3016220868457984 )
( 9.179,0.3016959915859623 )
( 9.229000000000001,0.3017692631156503 )
( 9.279000000000002,0.30184190919899323 )
( 9.329,0.301913937482185 )
( 9.379000000000001,0.30198535549543803 )
( 9.429,0.3020561706549124 )
( 9.479000000000001,0.30212639026461613 )
( 9.529000000000002,0.30219602151827857 )
( 9.579,0.30226507150119575 )
( 9.629000000000001,0.30233354719204814 )
( 9.679,0.3024014554646904 )
( 9.729000000000001,0.3024688030899134 )
( 9.779000000000002,0.3025355967371789 )
( 9.829,0.3026018429763259 )
( 9.879000000000001,0.30266754827925085 )
( 9.929000000000002,0.3027327190215593 )
( 9.979000000000001,0.30279736148419123 )
          };
        \addplot[domain=1.0:10, blue, solid, thick] coordinates {
( 1.2009999999999998,2.0867457833800365 )
( 1.251,1.7567303524178692 )
( 1.301,1.5332994595544645 )
( 1.351,1.3713836300293887 )
( 1.4009999999999998,1.2483045970876527 )
( 1.4509999999999998,1.1513813362892384 )
( 1.501,1.0729535312410106 )
( 1.551,1.0081118659049149 )
( 1.601,0.9535594515183543 )
( 1.6509999999999998,0.9069966909263272 )
( 1.701,0.8667685302862456 )
( 1.751,0.8316520145690276 )
( 1.801,0.80072297878257 )
( 1.851,0.7732694241241155 )
( 1.9009999999999998,0.7487335116830016 )
( 1.951,0.726671689641747 )
( 2.001,0.7067266481974173 )
( 2.051,0.6886071884793044 )
( 2.101,0.6720735080296164 )
( 2.151,0.656926269316748 )
( 2.201,0.6429983589068871 )
( 2.251,0.6301485920780721 )
( 2.301,0.6182568452237637 )
( 2.351,0.6072202505022359 )
( 2.401,0.5969501906917917 )
( 2.451,0.5873699038003327 )
( 2.501,0.5784125572401452 )
( 2.551,0.5700196871594955 )
( 2.601,0.5621399243220028 )
( 2.651,0.5547279467492169 )
( 2.701,0.5477436132294449 )
( 2.751,0.541151242147321 )
( 2.801,0.5349190078788577 )
( 2.851,0.529018432912435 )
( 2.901,0.5234239583866186 )
( 2.951,0.5181125792328423 )
( 3.001,0.5130635328308252 )
( 3.051,0.50825803221487 )
( 3.101,0.503679036548841 )
( 3.151,0.49931105292038497 )
( 3.201,0.4951399645688027 )
( 3.251,0.4911528815149848 )
( 3.301,0.4873380102511351 )
( 3.351,0.48368453970718067 )
( 3.4010000000000002,0.48018254116665865 )
( 3.451,0.47682288017824054 )
( 3.501,0.47359713881624516 )
( 3.551,0.47049754689728146 )
( 3.601,0.46751692097070624 )
( 3.6510000000000002,0.4646486100759095 )
( 3.701,0.4618864474060326 )
( 3.751,0.4592247071406524 )
( 3.801,0.456658065813508 )
( 3.851,0.4541815676687369 )
( 3.9010000000000002,0.4517905935331829 )
( 3.951,0.4494808327952747 )
( 4.0009999999999994,0.4472482581346345 )
( 4.051,0.44508910269244734 )
( 4.101,0.44299983941192916 )
( 4.151,0.44097716231201944 )
( 4.2010000000000005,0.4390179694865373 )
( 4.2509999999999994,0.4371193476461794 )
( 4.301,0.4352785580425052 )
( 4.351,0.4334930236319503 )
( 4.401,0.43176031735433107 )
( 4.4510000000000005,0.4300781514146236 )
( 4.5009999999999994,0.42844436746931236 )
( 4.551,0.4268569276295397 )
( 4.601,0.4253139062028939 )
( 4.651,0.42381348210409947 )
( 4.7010000000000005,0.42235393187230086 )
( 4.7509999999999994,0.42093362323917477 )
( 4.801,0.4195510091978818 )
( 4.851,0.4182046225279954 )
( 4.901,0.4168930707360697 )
( 4.9510000000000005,0.41561503137553474 )
( 5.0009999999999994,0.4143692477131952 )
( 5.051,0.41315452471277997 )
( 5.101000000000001,0.4119697253088469 )
( 5.151,0.41081376694687455 )
( 5.2010000000000005,0.4096856183676482 )
( 5.2509999999999994,0.4085842966160781 )
( 5.301,0.407508864256411 )
( 5.351000000000001,0.40645842677743305 )
( 5.401,0.40543213017273716 )
( 5.4510000000000005,0.40442915868244184 )
( 5.5009999999999994,0.4034487326839614 )
( 5.551,0.40249010672048363 )
( 5.601000000000001,0.401552567656802 )
( 5.651,0.40063543295301834 )
( 5.7010000000000005,0.3997380490474366 )
( 5.7509999999999994,0.3988597898406862 )
( 5.801,0.39800005527376814 )
( 5.851000000000001,0.39715826999332327 )
( 5.901,0.39633388209794407 )
( 5.9510000000000005,0.3955263619598666 )
( 6.0009999999999994,0.3947352011168087 )
( 6.051,0.3939599112291451 )
( 6.101000000000001,0.393200023097972 )
( 6.151,0.3924550857399653 )
( 6.2010000000000005,0.3917246655152396 )
( 6.2509999999999994,0.39100834530470724 )
( 6.301,0.390305723733699 )
( 6.351000000000001,0.38961641443883943 )
( 6.401,0.38894004537540255 )
( 6.4510000000000005,0.38827625816256317 )
( 6.5009999999999994,0.38762470746415734 )
( 6.551,0.3869850604027237 )
( 6.601000000000001,0.386356996004766 )
( 6.651,0.3857402046753134 )
( 6.7010000000000005,0.3851343876999965 )
( 6.7509999999999994,0.3845392567729708 )
( 6.801,0.38395453354914255 )
( 6.851000000000001,0.38337994921925284 )
( 6.901,0.38281524410646794 )
( 6.9510000000000005,0.38226016728322554 )
( 7.0009999999999994,0.38171447620715726 )
( 7.051,0.38117793637499464 )
( 7.101000000000001,0.38065032099342894 )
( 7.151,0.38013141066596906 )
( 7.2010000000000005,0.3796209930948963 )
( 7.2509999999999994,0.37911886279747864 )
( 7.301,0.37862482083565413 )
( 7.351000000000001,0.3781386745584434 )
( 7.401,0.37766023735640586 )
( 7.4510000000000005,0.3771893284274778 )
( 7.5009999999999994,0.3767257725535941 )
( 7.551,0.3762693998875116 )
( 7.601000000000001,0.37582004574930133 )
( 7.651,0.37537755043199683 )
( 7.7010000000000005,0.3749417590159305 )
( 7.7509999999999994,0.37451252119130063 )
( 7.801,0.3740896910885533 )
( 7.851000000000001,0.3736731271161791 )
( 7.901,0.3732626918055502 )
( 7.9510000000000005,0.3728582516624447 )
( 8.001,0.3724596770249242 )
( 8.051,0.3720668419272526 )
( 8.101,0.37167962396955573 )
( 8.151,0.37129790419294617 )
( 8.201,0.3709215669598432 )
( 8.251,0.3705504998392429 )
( 8.301,0.37018459349669797 )
( 8.351,0.369823741588786 )
( 8.401,0.36946784066185323 )
( 8.451,0.36911679005483206 )
( 8.501,0.36877049180594884 )
( 8.551,0.3684288505631327 )
( 8.601,0.3680917734979618 )
( 8.651,0.3677591702229866 )
( 8.701,0.36743095271226905 )
( 8.751,0.36710703522500315 )
( 8.801,0.36678733423207144 )
( 8.851,0.36647176834540685 )
( 8.901,0.36616025825004284 )
( 8.951,0.3658527266387246 )
( 9.001,0.3655490981489757 )
( 9.051,0.3652492993025109 )
( 9.100999999999999,0.364953258446895 )
( 9.151,0.36466090569934984 )
( 9.201,0.36437217289262075 )
( 9.251,0.36408699352281076 )
( 9.301,0.3638053026991048 )
( 9.350999999999999,0.36352703709530143 )
( 9.401,0.36325213490307573 )
( 9.451,0.36298053578690664 )
( 9.501,0.36271218084059526 )
( 9.551,0.3624470125453113 )
( 9.600999999999999,0.36218497472910566 )
( 9.651,0.3619260125278279 )
( 9.701,0.36167007234739645 )
( 9.751,0.3614171018273613 )
( 9.801,0.3611670498057134 )
( 9.850999999999999,0.3609198662848881 )
( 9.901,0.3606755023989161 )
( 9.951,0.3604339103816797 )
          };
        \legend{
          {\small SFHE: $\pi^{-1}\Gamma\left(1+1/\alpha\right)$},
          {\small SFWE: $2^{1-1/\alpha}\pi^{-1}\alpha^{-1}\cos\left(\pi/\alpha\right)\Gamma\left(2(1/\alpha-1\right)$}
        };
        \addplot[thick, dotted] coordinates {(0,0.25)(0,2)};
        \addplot[thick, dotted] coordinates {(1,0.25)(1,2)};
        \addplot[thick, dotted] coordinates {(2,0.25)(2,1.0)};
        \addplot[thick, dotted] coordinates {(0.0,0.31813)(1.2,0.31813)};
        \addplot[thick, dotted] coordinates {(0.0,0.282095)(5.2,0.282095)};
        \addplot[thick, dotted] coordinates {(0.0,0.707)(2.2,0.707)};
      \end{axis}
    \end{tikzpicture}
    \bigskip

    \begin{tikzpicture}[scale=0.9, transform shape, x=3em, y=5em]
      \tikzset{>=latex}
      \begin{axis}[
        xmin=0.5,
        xmax=5.5,
        ymin=0,
        ytick={        0, 0.25,                         0.840896,                          1.1513,   1.5  },
        yticklabels={$0$, {\small SHE:} $\sfrac{1}{4}$, {\small SWE:} $2^{-\sfrac{1}{4}}$, $1.1513$, $1.5$},
        xtick={        1, 1.3426,   2,   3,   5  },
        xticklabels={$1$, $\rotatebox{90}{1.3426}$, $2$, $3$, $5$},
        legend style={at={(1.50,+1.00)}, anchor=north, legend cell align=left, draw=none},
        xlabel={$\alpha$},
        title={Second moment Lyapunov exponent}
        ]
        \addplot[domain=1.31:5.5, blue, solid, thick] coordinates {
            (1.3100000e+00,1.6627598e+00)
            (1.3310553e+00,1.2968472e+00)
            (1.3521106e+00,1.0521308e+00)
            (1.3731658e+00,8.8103162e-01)
            (1.3942211e+00,7.5695988e-01)
            (1.4152764e+00,6.6422272e-01)
            (1.4363317e+00,5.9312184e-01)
            (1.4573869e+00,5.3742290e-01)
            (1.4784422e+00,4.9297986e-01)
            (1.4994975e+00,4.5695276e-01)
            (1.5205528e+00,4.2734520e-01)
            (1.5416080e+00,4.0272137e-01)
            (1.5626633e+00,3.8202741e-01)
            (1.5837186e+00,3.6447564e-01)
            (1.6047739e+00,3.4946773e-01)
            (1.6258291e+00,3.3654239e-01)
            (1.6468844e+00,3.2533928e-01)
            (1.6679397e+00,3.1557348e-01)
            (1.6889950e+00,3.0701718e-01)
            (1.7100503e+00,2.9948643e-01)
            (1.7311055e+00,2.9283130e-01)
            (1.7521608e+00,2.8692852e-01)
            (1.7732161e+00,2.8167598e-01)
            (1.7942714e+00,2.7698842e-01)
            (1.8153266e+00,2.7279424e-01)
            (1.8363819e+00,2.6903287e-01)
            (1.8574372e+00,2.6565281e-01)
            (1.8784925e+00,2.6261000e-01)
            (1.8995477e+00,2.5986659e-01)
            (1.9206030e+00,2.5738987e-01)
            (1.9416583e+00,2.5515147e-01)
            (1.9627136e+00,2.5312668e-01)
            (1.9837688e+00,2.5129388e-01)
            (2.0048241e+00,2.4963409e-01)
            (2.0258794e+00,2.4813058e-01)
            (2.0469347e+00,2.4676859e-01)
            (2.0679899e+00,2.4553499e-01)
            (2.0890452e+00,2.4441813e-01)
            (2.1101005e+00,2.4340761e-01)
            (2.1311558e+00,2.4249414e-01)
            (2.1522111e+00,2.4166937e-01)
            (2.1732663e+00,2.4092581e-01)
            (2.1943216e+00,2.4025672e-01)
            (2.2153769e+00,2.3965600e-01)
            (2.2364322e+00,2.3911813e-01)
            (2.2574874e+00,2.3863813e-01)
            (2.2785427e+00,2.3821145e-01)
            (2.2995980e+00,2.3783396e-01)
            (2.3206533e+00,2.3750191e-01)
            (2.3417085e+00,2.3721185e-01)
            (2.3627638e+00,2.3696066e-01)
            (2.3838191e+00,2.3674545e-01)
            (2.4048744e+00,2.3656358e-01)
            (2.4259296e+00,2.3641264e-01)
            (2.4469849e+00,2.3629040e-01)
            (2.4680402e+00,2.3619481e-01)
            (2.4890955e+00,2.3612397e-01)
            (2.5101508e+00,2.3607616e-01)
            (2.5312060e+00,2.3604975e-01)
            (2.5522613e+00,2.3604326e-01)
            (2.5733166e+00,2.3605531e-01)
            (2.5943719e+00,2.3608461e-01)
            (2.6154271e+00,2.3612999e-01)
            (2.6364824e+00,2.3619034e-01)
            (2.6575377e+00,2.3626465e-01)
            (2.6785930e+00,2.3635195e-01)
            (2.6996482e+00,2.3645136e-01)
            (2.7207035e+00,2.3656207e-01)
            (2.7417588e+00,2.3668329e-01)
            (2.7628141e+00,2.3681432e-01)
            (2.7838693e+00,2.3695448e-01)
            (2.8049246e+00,2.3710315e-01)
            (2.8259799e+00,2.3725974e-01)
            (2.8470352e+00,2.3742371e-01)
            (2.8680905e+00,2.3759454e-01)
            (2.8891457e+00,2.3777176e-01)
            (2.9102010e+00,2.3795491e-01)
            (2.9312563e+00,2.3814358e-01)
            (2.9523116e+00,2.3833737e-01)
            (2.9733668e+00,2.3853591e-01)
            (2.9944221e+00,2.3873886e-01)
            (3.0154774e+00,2.3894588e-01)
            (3.0365327e+00,2.3915667e-01)
            (3.0575879e+00,2.3937095e-01)
            (3.0786432e+00,2.3958844e-01)
            (3.0996985e+00,2.3980889e-01)
            (3.1207538e+00,2.4003205e-01)
            (3.1418090e+00,2.4025770e-01)
            (3.1628643e+00,2.4048564e-01)
            (3.1839196e+00,2.4071566e-01)
            (3.2049749e+00,2.4094757e-01)
            (3.2260302e+00,2.4118119e-01)
            (3.2470854e+00,2.4141636e-01)
            (3.2681407e+00,2.4165292e-01)
            (3.2891960e+00,2.4189073e-01)
            (3.3102513e+00,2.4212963e-01)
            (3.3313065e+00,2.4236951e-01)
            (3.3523618e+00,2.4261023e-01)
            (3.3734171e+00,2.4285169e-01)
            (3.3944724e+00,2.4309376e-01)
            (3.4155276e+00,2.4333634e-01)
            (3.4365829e+00,2.4357935e-01)
            (3.4576382e+00,2.4382267e-01)
            (3.4786935e+00,2.4406623e-01)
            (3.4997487e+00,2.4430995e-01)
            (3.5208040e+00,2.4455374e-01)
            (3.5418593e+00,2.4479753e-01)
            (3.5629146e+00,2.4504126e-01)
            (3.5839698e+00,2.4528485e-01)
            (3.6050251e+00,2.4552825e-01)
            (3.6260804e+00,2.4577140e-01)
            (3.6471357e+00,2.4601424e-01)
            (3.6681910e+00,2.4625672e-01)
            (3.6892462e+00,2.4649880e-01)
            (3.7103015e+00,2.4674042e-01)
            (3.7313568e+00,2.4698156e-01)
            (3.7524121e+00,2.4722215e-01)
            (3.7734673e+00,2.4746217e-01)
            (3.7945226e+00,2.4770159e-01)
            (3.8155779e+00,2.4794036e-01)
            (3.8366332e+00,2.4817846e-01)
            (3.8576884e+00,2.4841585e-01)
            (3.8787437e+00,2.4865252e-01)
            (3.8997990e+00,2.4888842e-01)
            (3.9208543e+00,2.4912354e-01)
            (3.9419095e+00,2.4935786e-01)
            (3.9629648e+00,2.4959136e-01)
            (3.9840201e+00,2.4982401e-01)
            (4.0050754e+00,2.5005579e-01)
            (4.0261307e+00,2.5028669e-01)
            (4.0471859e+00,2.5051670e-01)
            (4.0682412e+00,2.5074579e-01)
            (4.0892965e+00,2.5097396e-01)
            (4.1103518e+00,2.5120119e-01)
            (4.1314070e+00,2.5142747e-01)
            (4.1524623e+00,2.5165278e-01)
            (4.1735176e+00,2.5187712e-01)
            (4.1945729e+00,2.5210049e-01)
            (4.2156281e+00,2.5232286e-01)
            (4.2366834e+00,2.5254424e-01)
            (4.2577387e+00,2.5276461e-01)
            (4.2787940e+00,2.5298397e-01)
            (4.2998492e+00,2.5320232e-01)
            (4.3209045e+00,2.5341964e-01)
            (4.3419598e+00,2.5363594e-01)
            (4.3630151e+00,2.5385122e-01)
            (4.3840704e+00,2.5406546e-01)
            (4.4051256e+00,2.5427866e-01)
            (4.4261809e+00,2.5449083e-01)
            (4.4472362e+00,2.5470196e-01)
            (4.4682915e+00,2.5491204e-01)
            (4.4893467e+00,2.5512109e-01)
            (4.5104020e+00,2.5532909e-01)
            (4.5314573e+00,2.5553605e-01)
            (4.5525126e+00,2.5574197e-01)
            (4.5735678e+00,2.5594685e-01)
            (4.5946231e+00,2.5615068e-01)
            (4.6156784e+00,2.5635347e-01)
            (4.6367337e+00,2.5655523e-01)
            (4.6577889e+00,2.5675594e-01)
            (4.6788442e+00,2.5695562e-01)
            (4.6998995e+00,2.5715427e-01)
            (4.7209548e+00,2.5735188e-01)
            (4.7420101e+00,2.5754846e-01)
            (4.7630653e+00,2.5774402e-01)
            (4.7841206e+00,2.5793855e-01)
            (4.8051759e+00,2.5813206e-01)
            (4.8262312e+00,2.5832455e-01)
            (4.8472864e+00,2.5851603e-01)
            (4.8683417e+00,2.5870650e-01)
            (4.8893970e+00,2.5889596e-01)
            (4.9104523e+00,2.5908441e-01)
            (4.9315075e+00,2.5927186e-01)
            (4.9525628e+00,2.5945832e-01)
            (4.9736181e+00,2.5964379e-01)
            (4.9946734e+00,2.5982827e-01)
            (5.0157286e+00,2.6001176e-01)
            (5.0367839e+00,2.6019428e-01)
            (5.0578392e+00,2.6037582e-01)
            (5.0788945e+00,2.6055639e-01)
            (5.0999497e+00,2.6073599e-01)
            (5.1210050e+00,2.6091464e-01)
            (5.1420603e+00,2.6109233e-01)
            (5.1631156e+00,2.6126906e-01)
            (5.1841709e+00,2.6144485e-01)
            (5.2052261e+00,2.6161970e-01)
            (5.2262814e+00,2.6179361e-01)
            (5.2473367e+00,2.6196659e-01)
            (5.2683920e+00,2.6213864e-01)
            (5.2894472e+00,2.6230977e-01)
            (5.3105025e+00,2.6247998e-01)
            (5.3315578e+00,2.6264928e-01)
            (5.3526131e+00,2.6281768e-01)
            (5.3736683e+00,2.6298517e-01)
            (5.3947236e+00,2.6315177e-01)
            (5.4157789e+00,2.6331748e-01)
            (5.4368342e+00,2.6348230e-01)
            (5.4578894e+00,2.6364624e-01)
            (5.4789447e+00,2.6380930e-01)
            (5.5000000e+00,2.6397149e-01)
          };
        \addplot[domain=1.31:5.5, red, dashed, thick] coordinates {
            (1.1900000e+00,1.6603830e+00)
            (1.2116583e+00,1.5368180e+00)
            (1.2333166e+00,1.4386716e+00)
            (1.2549749e+00,1.3591361e+00)
            (1.2766332e+00,1.2936034e+00)
            (1.2982915e+00,1.2388491e+00)
            (1.3199497e+00,1.1925530e+00)
            (1.3416080e+00,1.1530050e+00)
            (1.3632663e+00,1.1189192e+00)
            (1.3849246e+00,1.0893103e+00)
            (1.4065829e+00,1.0634120e+00)
            (1.4282412e+00,1.0406196e+00)
            (1.4498995e+00,1.0204500e+00)
            (1.4715578e+00,1.0025131e+00)
            (1.4932161e+00,9.8649024e-01)
            (1.5148744e+00,9.7211914e-01)
            (1.5365327e+00,9.5918195e-01)
            (1.5581910e+00,9.4749633e-01)
            (1.5798492e+00,9.3690861e-01)
            (1.6015075e+00,9.2728843e-01)
            (1.6231658e+00,9.1852456e-01)
            (1.6448241e+00,9.1052154e-01)
            (1.6664824e+00,9.0319705e-01)
            (1.6881407e+00,8.9647974e-01)
            (1.7097990e+00,8.9030749e-01)
            (1.7314573e+00,8.8462601e-01)
            (1.7531156e+00,8.7938765e-01)
            (1.7747739e+00,8.7455043e-01)
            (1.7964322e+00,8.7007726e-01)
            (1.8180905e+00,8.6593524e-01)
            (1.8397487e+00,8.6209509e-01)
            (1.8614070e+00,8.5853071e-01)
            (1.8830653e+00,8.5521875e-01)
            (1.9047236e+00,8.5213826e-01)
            (1.9263819e+00,8.4927039e-01)
            (1.9480402e+00,8.4659818e-01)
            (1.9696985e+00,8.4410628e-01)
            (1.9913568e+00,8.4178083e-01)
            (2.0130151e+00,8.3960922e-01)
            (2.0346734e+00,8.3758000e-01)
            (2.0563317e+00,8.3568278e-01)
            (2.0779899e+00,8.3390802e-01)
            (2.0996482e+00,8.3224706e-01)
            (2.1213065e+00,8.3069195e-01)
            (2.1429648e+00,8.2923539e-01)
            (2.1646231e+00,8.2787070e-01)
            (2.1862814e+00,8.2659174e-01)
            (2.2079397e+00,8.2539284e-01)
            (2.2295980e+00,8.2426878e-01)
            (2.2512563e+00,8.2321475e-01)
            (2.2729146e+00,8.2222630e-01)
            (2.2945729e+00,8.2129932e-01)
            (2.3162312e+00,8.2042997e-01)
            (2.3378894e+00,8.1961474e-01)
            (2.3595477e+00,8.1885034e-01)
            (2.3812060e+00,8.1813370e-01)
            (2.4028643e+00,8.1746201e-01)
            (2.4245226e+00,8.1683261e-01)
            (2.4461809e+00,8.1624304e-01)
            (2.4678392e+00,8.1569100e-01)
            (2.4894975e+00,8.1517436e-01)
            (2.5111558e+00,8.1469109e-01)
            (2.5328141e+00,8.1423934e-01)
            (2.5544724e+00,8.1381735e-01)
            (2.5761307e+00,8.1342348e-01)
            (2.5977889e+00,8.1305617e-01)
            (2.6194472e+00,8.1271400e-01)
            (2.6411055e+00,8.1239561e-01)
            (2.6627638e+00,8.1209972e-01)
            (2.6844221e+00,8.1182514e-01)
            (2.7060804e+00,8.1157073e-01)
            (2.7277387e+00,8.1133545e-01)
            (2.7493970e+00,8.1111829e-01)
            (2.7710553e+00,8.1091832e-01)
            (2.7927136e+00,8.1073465e-01)
            (2.8143719e+00,8.1056644e-01)
            (2.8360302e+00,8.1041291e-01)
            (2.8576884e+00,8.1027331e-01)
            (2.8793467e+00,8.1014694e-01)
            (2.9010050e+00,8.1003312e-01)
            (2.9226633e+00,8.0993125e-01)
            (2.9443216e+00,8.0984071e-01)
            (2.9659799e+00,8.0976095e-01)
            (2.9876382e+00,8.0969143e-01)
            (3.0092965e+00,8.0963165e-01)
            (3.0309548e+00,8.0958113e-01)
            (3.0526131e+00,8.0953942e-01)
            (3.0742714e+00,8.0950608e-01)
            (3.0959296e+00,8.0948071e-01)
            (3.1175879e+00,8.0946291e-01)
            (3.1392462e+00,8.0945233e-01)
            (3.1609045e+00,8.0944861e-01)
            (3.1825628e+00,8.0945141e-01)
            (3.2042211e+00,8.0946043e-01)
            (3.2258794e+00,8.0947536e-01)
            (3.2475377e+00,8.0949591e-01)
            (3.2691960e+00,8.0952182e-01)
            (3.2908543e+00,8.0955283e-01)
            (3.3125126e+00,8.0958868e-01)
            (3.3341709e+00,8.0962915e-01)
            (3.3558291e+00,8.0967401e-01)
            (3.3774874e+00,8.0972304e-01)
            (3.3991457e+00,8.0977605e-01)
            (3.4208040e+00,8.0983284e-01)
            (3.4424623e+00,8.0989323e-01)
            (3.4641206e+00,8.0995703e-01)
            (3.4857789e+00,8.1002409e-01)
            (3.5074372e+00,8.1009423e-01)
            (3.5290955e+00,8.1016732e-01)
            (3.5507538e+00,8.1024320e-01)
            (3.5724121e+00,8.1032172e-01)
            (3.5940704e+00,8.1040277e-01)
            (3.6157286e+00,8.1048621e-01)
            (3.6373869e+00,8.1057192e-01)
            (3.6590452e+00,8.1065978e-01)
            (3.6807035e+00,8.1074968e-01)
            (3.7023618e+00,8.1084152e-01)
            (3.7240201e+00,8.1093519e-01)
            (3.7456784e+00,8.1103060e-01)
            (3.7673367e+00,8.1112765e-01)
            (3.7889950e+00,8.1122626e-01)
            (3.8106533e+00,8.1132634e-01)
            (3.8323116e+00,8.1142780e-01)
            (3.8539698e+00,8.1153058e-01)
            (3.8756281e+00,8.1163459e-01)
            (3.8972864e+00,8.1173977e-01)
            (3.9189447e+00,8.1184604e-01)
            (3.9406030e+00,8.1195334e-01)
            (3.9622613e+00,8.1206161e-01)
            (3.9839196e+00,8.1217078e-01)
            (4.0055779e+00,8.1228081e-01)
            (4.0272362e+00,8.1239163e-01)
            (4.0488945e+00,8.1250320e-01)
            (4.0705528e+00,8.1261545e-01)
            (4.0922111e+00,8.1272835e-01)
            (4.1138693e+00,8.1284184e-01)
            (4.1355276e+00,8.1295589e-01)
            (4.1571859e+00,8.1307045e-01)
            (4.1788442e+00,8.1318547e-01)
            (4.2005025e+00,8.1330093e-01)
            (4.2221608e+00,8.1341677e-01)
            (4.2438191e+00,8.1353298e-01)
            (4.2654774e+00,8.1364950e-01)
            (4.2871357e+00,8.1376632e-01)
            (4.3087940e+00,8.1388339e-01)
            (4.3304523e+00,8.1400068e-01)
            (4.3521106e+00,8.1411818e-01)
            (4.3737688e+00,8.1423585e-01)
            (4.3954271e+00,8.1435366e-01)
            (4.4170854e+00,8.1447158e-01)
            (4.4387437e+00,8.1458960e-01)
            (4.4604020e+00,8.1470769e-01)
            (4.4820603e+00,8.1482583e-01)
            (4.5037186e+00,8.1494399e-01)
            (4.5253769e+00,8.1506216e-01)
            (4.5470352e+00,8.1518032e-01)
            (4.5686935e+00,8.1529844e-01)
            (4.5903518e+00,8.1541651e-01)
            (4.6120101e+00,8.1553451e-01)
            (4.6336683e+00,8.1565242e-01)
            (4.6553266e+00,8.1577023e-01)
            (4.6769849e+00,8.1588793e-01)
            (4.6986432e+00,8.1600549e-01)
            (4.7203015e+00,8.1612291e-01)
            (4.7419598e+00,8.1624018e-01)
            (4.7636181e+00,8.1635727e-01)
            (4.7852764e+00,8.1647417e-01)
            (4.8069347e+00,8.1659088e-01)
            (4.8285930e+00,8.1670739e-01)
            (4.8502513e+00,8.1682368e-01)
            (4.8719095e+00,8.1693974e-01)
            (4.8935678e+00,8.1705556e-01)
            (4.9152261e+00,8.1717114e-01)
            (4.9368844e+00,8.1728647e-01)
            (4.9585427e+00,8.1740153e-01)
            (4.9802010e+00,8.1751632e-01)
            (5.0018593e+00,8.1763083e-01)
            (5.0235176e+00,8.1774505e-01)
            (5.0451759e+00,8.1785898e-01)
            (5.0668342e+00,8.1797261e-01)
            (5.0884925e+00,8.1808594e-01)
            (5.1101508e+00,8.1819895e-01)
            (5.1318090e+00,8.1831165e-01)
            (5.1534673e+00,8.1842402e-01)
            (5.1751256e+00,8.1853606e-01)
            (5.1967839e+00,8.1864777e-01)
            (5.2184422e+00,8.1875915e-01)
            (5.2401005e+00,8.1887018e-01)
            (5.2617588e+00,8.1898087e-01)
            (5.2834171e+00,8.1909120e-01)
            (5.3050754e+00,8.1920118e-01)
            (5.3267337e+00,8.1931081e-01)
            (5.3483920e+00,8.1942007e-01)
            (5.3700503e+00,8.1952897e-01)
            (5.3917085e+00,8.1963751e-01)
            (5.4133668e+00,8.1974568e-01)
            (5.4350251e+00,8.1985347e-01)
            (5.4566834e+00,8.1996089e-01)
            (5.4783417e+00,8.2006794e-01)
            (5.5000000e+00,8.2017461e-01)
          };
        \legend{
          {\small SFHE: $\left[ \alpha\sin\left(\pi/\alpha\right)\right ]^{\frac{\alpha}{1-\alpha}}$},
          {\small SFWE: $2^{\frac{1-\alpha}{2-3\alpha}}\left[ \alpha\sin\left(\pi/\alpha\right)\right ]^{\frac{\alpha}{2-3\alpha}}$}
        };
        \addplot[thick, dotted] coordinates {(2,0.1)(2,1)};
        \addplot[thick, dotted] coordinates {(2.2,0.25)(-0.5,0.25)} node [left] {$1/4$};
        \addplot[thick, dotted] coordinates {(2.2,0.840896)(-0.5,0.840896)} node [left] {$1/4$};
        \addplot[thick, dotted] coordinates {(1.0,0.0) (1.0,1.66)};
        \addplot[thick, dotted] coordinates {(1.3426,0.0) (1.3426,1.4)};
        \addplot[thick, dotted] coordinates {(0,1.1513) (1.5,1.1513)};
      \end{axis}
    \end{tikzpicture}
  \end{center}

  \caption{Plots of both $\Theta_{\alpha,\nu}$ and the second moment Lyapunov exponents as functions
  of $\alpha\in (1,\infty)$ with $\nu=\lambda=1$ for both SFHE in Example \ref{Ex:SFHE} and SFWE in
Example \ref{Ex:SFWE}. For the second moment Lyapunov exponents, two curves intersect at
$\left(1.3426,1.1513\right)$ via some numerical solver.}

  \label{F:SFHE}
\end{figure}

%
%
%

\begin{example}[SFWE] \label{Ex:SFWE}
  For the stochastic fractional wave equation
  \begin{equation}\label{E:SFWE}
  \text{(SFWE)} \quad
  \begin{cases}
    \left(\dfrac{\partial^2}{\partial t^2}+\dfrac{\nu}{2}\left(-\Delta\right)^{\alpha / 2}\right) u(t, x)= \: \lambda u(t, x) \dot{W}(t, x) , & t>0, x \in \mathbb{R}, \\
    u(0,\cdot)=u_0, \quad \dfrac{\partial}{\partial t} u(0, \cdot)=u_1,
  \end{cases}
\end{equation}
   i.e., $\alpha>0$, $\beta=2$, $\gamma=0$ and $d=1$, Dalang's condition \eqref{E:Dalang'} becomes
   $\alpha>1$, and the quantities in \eqref{E:theta} reduce to
  \begin{align*}
    \theta              = & 2(1-1/\alpha), \\
    \Theta_{\alpha,\nu} = & \frac{1}{\pi}\int_0^\infty \frac{\sin^2\left(\sqrt{\nu/2}\: \xi^{\alpha/2}\right)}{\left(\nu/2\right)\xi^\alpha}\ud \xi
                        = \frac{2^{2-1/\alpha}\cos\left(\pi/\alpha\right) \Gamma\left(2\left(1/\alpha-1\right)\right)}{\nu^{1/\alpha}\pi\alpha}, \\
    \widehat{t}         = & \frac{2^{2-1/\alpha} \cos\left(\pi/\alpha\right) \Gamma\left(3-2/\alpha\right) \Gamma\left(2/\alpha-2\right)}{\nu^{1/\alpha}\pi\alpha} t^{3-2/\alpha} \\
                        = & \frac{2^{2-1/\alpha} \cos\left(\pi/\alpha\right)}{\nu^{1/\alpha}\sin\left(2\pi/\alpha\right)\alpha} t^{3-2/\alpha}
                        = \frac{2^{1-1/\alpha}}{\nu^{1/\alpha}\sin\left(\pi/\alpha\right)\alpha} t^{3-2/\alpha},\\
    t_p                 = & p^{1+\alpha/(3\alpha-2)}t,
  \end{align*}
  where we have applied Lemma \ref{L:sin} and the reflection formula \eqref{E:Reflection} in
  computing $\Theta$ and $\widehat{t}$, respectively.
  \begin{enumerate}[wide=1em]
    \item {\it Second moment formula:} By \eqref{E:SecMom}, the second moment formula is
      \begin{equation}\label{E:2ndSFWE}
        \begin{split}
           \E\left[u^2(t,x)\right] =
           & u_0^2       \:E_{3-2/\alpha}  \left(\frac{2^{1-1/\alpha}\lambda^2}{\nu^{1/\alpha}\sin\left(\pi/\alpha\right)\alpha} t^{3-2/\alpha}\: \right) \\
           & + 2u_0u_1 t \:E_{3-2/\alpha,2}\left(\frac{2^{1-1/\alpha}\lambda^2}{\nu^{1/\alpha}\sin\left(\pi/\alpha\right)\alpha} t^{3-2/\alpha}\: \right) \\
           & + 2u_1^2 t^2\:E_{3-2/\alpha,3}\left(\frac{2^{1-1/\alpha}\lambda^2}{\nu^{1/\alpha}\sin\left(\pi/\alpha\right)\alpha} t^{3-2/\alpha}\: \right).
        \end{split}
      \end{equation}
    \item {\it Second moment Lyapunov exponent:} From \eqref{E:2ndSFWE}, we immediately see that
      \begin{align} \label{E:LySFWE}
          \lim_{t\to \infty} t^{-1} \log \E\left[u(t,x)^2\right]
        = \left(\frac{2^{1-1/\alpha}\lambda^2}{\nu^{1/\alpha}\sin\left(\pi/\alpha\right)\alpha} \right)^{\alpha/(3\alpha-2)},
        \qquad \alpha>1;
      \end{align}
      see Figure \ref{F:SFHE} for a plot of this expression as a function of $\alpha$.
    \item {\it Moment asymptotics:} The asymptotics in \eqref{E:upper-lim} shows that
      \begin{subequations} \label{E:AsymSFWE}
        \begin{align} \label{E:AsymSFWE-t}
             & \limsup_{t\to \infty} t^{-1} \log \E[|u(t,x)|^p]
           \le C_2 p^{\frac{4\alpha-2}{3\alpha-2}}, & p\ge 2, \\
             &
               \limsup_{p\to \infty} p^{-\frac{4\alpha-2}{3\alpha-2}} \log \E[|u(t,x)|^p]
           \le C_4 \: t, & t>0. \label{E:AsymSFWE-p}
        \end{align}
      \end{subequations}
      The large-time asymptotics in \eqref{E:AsymSFWE-t} is consistent with Proposition 4.1 of
      \cite{song.song.ea:20:fractional}. Since we don't know if the fundamental solution is
      nonnegative, we cannot apply the lower asymptotics in \eqref{E:lower-lim}. To the best of our
      knowledge, formulas \eqref{E:2ndSFWE} and \eqref{E:LySFWE} and the limit \eqref{E:AsymSFWE}
      are new.
  \end{enumerate}
\end{example}

\begin{example} \label{Ex:TFSPDE}
  The following one-parameter family of SPDEs
  \begin{equation}\label{E:TFSPDE}
  \begin{cases}
    \left(\partial_t^{\beta}-\dfrac{\nu}{2}\dfrac{\partial^2}{\partial x^2}\right) u(t, x)= \: I_{t}^{\Ceil{\beta}-\beta}\left[\lambda u(t, x) \dot{W}(t, x)\right] , & t>0, x \in \mathbb{R},       \\
    u(0,\cdot)=u_0,                                                                                                                       & \text { if } \beta \in(0,1], \\
    u(0,\cdot)=u_0, \quad \dfrac{\partial}{\partial t} u(0, \cdot)=u_1,                                                                   & \text { if } \beta \in(1,2),
  \end{cases}
  \end{equation}
has been studied in \cite{chen:17:nonlinear}. This is the case when $d=1$,
    $\alpha=2$, $\beta\in (0,2)$ and $\gamma=\Ceil{\beta}-\beta$ and the upper bound of the
    large-time asymptotics was obtained ({\it ibid.}). It can be easily checked that Dalang's
    condition \eqref{E:Dalang'} holds true for all $\beta\in(0,2)$ in this case. The quantities in
  \eqref{E:theta} reduce to
  \begin{align*}
    \theta             & = 2\left(\Ceil{\beta}-1\right) -\beta/2,                                                       & t_p         & = p^{1+\frac{2}{4\Ceil{\beta}-2-\beta}} t, \\
    \Theta_{\beta,\nu} & = \frac{1}{\pi}\int_0^\infty E_{\beta,\Ceil{\beta}}^2 \left(-\frac{\nu}{2}\xi^2\right) \ud\xi, & \widehat{t} & = \Theta_{\beta,\nu} \Gamma\left(2\Ceil{\beta}-1-\beta/2\right) t^{2\Ceil{\beta}-1 -\beta/2},
  \end{align*}
  and hence we have the following:
  \begin{enumerate}[wide=1em]
    \item {\it Second moment formula:} By \eqref{E:SecMom}, the second moment formula is
      \begin{align} \label{E:2ndTFSPDE}
        \E\left[u^2(t,x)\right] =
        \begin{cases}
          \quad \quad u_0^2 \:E_{1-\beta/2}  \left(\lambda^2 \Theta_{\beta,\nu} \Gamma\left(1-\beta/2\right) t^{1 -\beta/2}\right) & \text{if $\beta\in (0,1]$},\\[1em]
        \begin{aligned}
           &  & u_0^2          \:E_{3 -\beta/2}   \left(\lambda^2 \Theta_{\beta,\nu} \Gamma\left(3-\beta/2\right) t^{3-\beta/2}\right) \\
           &  & + 2u_0u_1\: t  \:E_{3 -\beta/2,2} \left(\lambda^2 \Theta_{\beta,\nu} \Gamma\left(3-\beta/2\right) t^{3-\beta/2}\right) \\
           &  & + 2u_1^2 \: t^2\:E_{3 -\beta/2,3} \left(\lambda^2 \Theta_{\beta,\nu} \Gamma\left(3-\beta/2\right) t^{3-\beta/2}\right)
        \end{aligned} & \text{if $\beta\in (1,2)$}.
        \end{cases}
      \end{align}
    \item {\it Second moment Lyapunov exponent:} From \eqref{E:2ndSFWE}, we see that
      \begin{align} \label{E:LyTFSPDE}
          \lim_{t\to \infty} t^{-1} \log \E\left[u(t,x)^2\right]
          = \left(\lambda^2 \Theta_{\beta,\nu} \Gamma\left(2\Ceil{\beta}-1-\beta/2\right)\right)^{\frac{2}{4\Ceil{\beta}-2-\beta}}.
      \end{align}
    \item {\it Moment asymptotics:} Since the fundamental solution in this case is nonnegative (see
      Remark \ref{R:Nonneg} below), we can combine the asymptotics in both \eqref{E:upper-lim} and
      \eqref{E:lower-lim} to see that
      \begin{subequations} \label{E:AsymTFSPDE}
        \begin{align} \label{E:AsymTFSPDE-t}
               \hspace{-2em}
               C_1 p^{\frac{4\Ceil{\beta}-\beta}{4\Ceil{\beta}-2-\beta}}
           \le & \liminf_{t\to \infty} \frac{\log \E[u(t,x)^p]}{t}
           \le   \limsup_{t\to \infty} \frac{\log \E[u(t,x)^p]}{t}
           \le C_2 p^{\frac{4\Ceil{\beta}-\beta}{4\Ceil{\beta}-2-\beta}}, \quad  p\ge 2, \\
               C_3 t
           \le & \liminf_{p\to \infty} \frac{\log \E[u(t,x)^p]}{p^{\frac{4\Ceil{\beta}-\beta}{4\Ceil{\beta}-2-\beta}}}
           \le   \limsup_{p\to \infty} \frac{\log \E[u(t,x)^p]}{p^{\frac{4\Ceil{\beta}-\beta}{4\Ceil{\beta}-2-\beta}}}
           \le C_4 t, \quad t>0. \label{E:AsymTFSPDE-p}
        \end{align}
      \end{subequations}
      The upper bound for the large-time asymptotics in \eqref{E:AsymTFSPDE-t} recover the results
      obtained in \cite{chen:17:nonlinear}; see Theorems 3.5 and 3.6 ({\it ibid.}). In particular,
      when $\beta\in (0,1]$, Mijena and Nane \cite[Theorem 2]{mijena.nane:15:space-time} obtained
      the same upper bound as in \eqref{E:AsymTFSPDE-t}. Except the upper bound
      in \eqref{E:AsymTFSPDE-t}, all the rest results in this example are new.
  \end{enumerate}

  In Figure \ref{F:TFSPDE}, we plot the graphs of $\theta$, $\Theta_{\beta,\nu}$, $1+1/(1+\theta)$,
  and the second moment Lyapunov exponent as functions of $\beta$ with $\lambda$ and $\nu$ being set
  to $1$ and $2$, respectively.
\end{example}

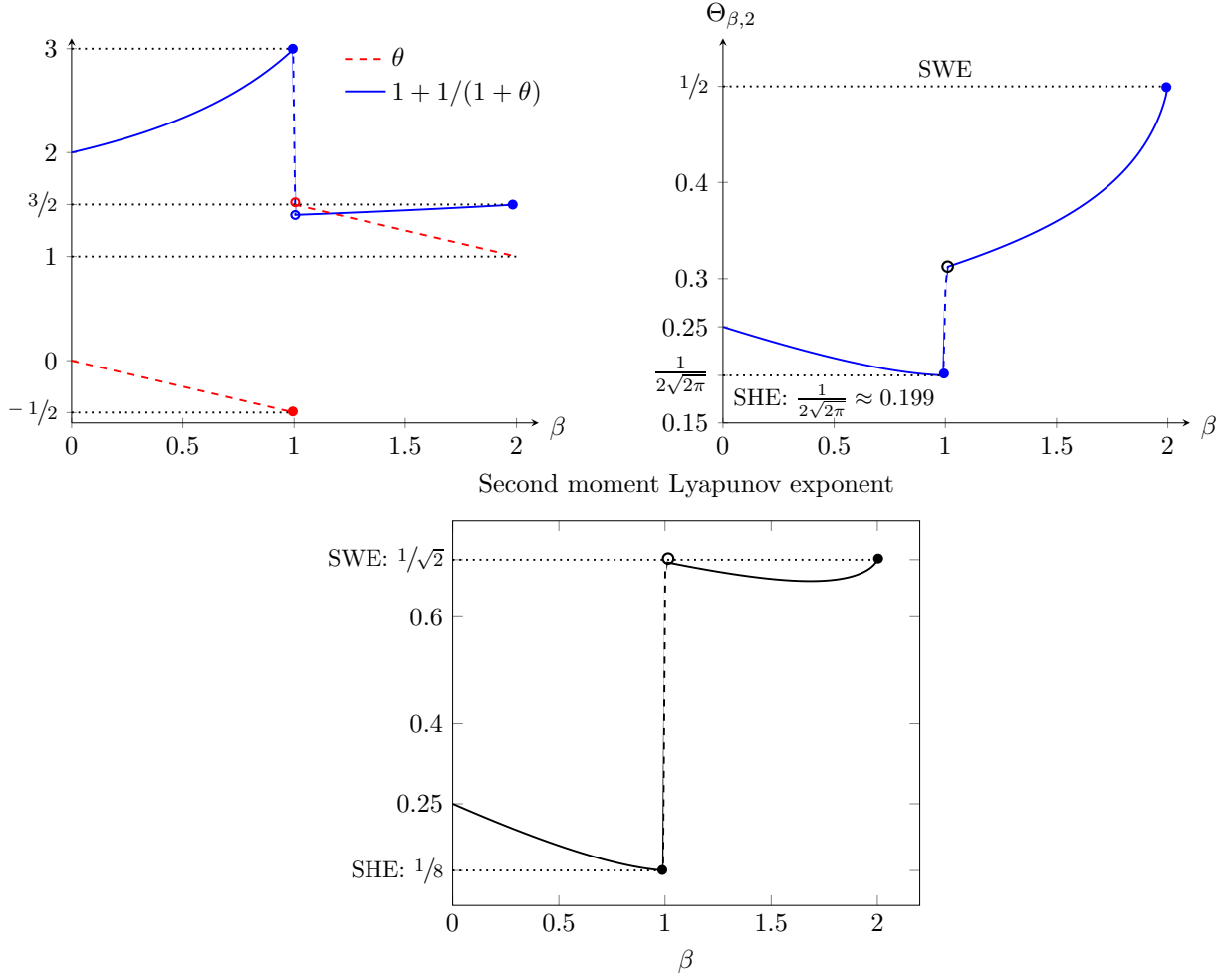
\begin{figure}[htpb]
  \begin{center}
    \hspace{-1em}
    \begin{tikzpicture}[scale=0.9, transform shape, x=1em, y=1em]
      \tikzset{>=latex}
      \begin{axis}[
        axis lines = left,
        ytick={     -0.5,             0,   1,   1.5,            2,   3},
        yticklabels={$\sfrac{-1}{2}$, $0$, $1$, $\sfrac{3}{2}$, $2$, $3$},
        legend style={at={(0.8,+1.00)}, anchor=north, legend cell align=left, draw=none},
        xmin=0,   xmax= 2.1,
        ymax=3.1, ymin=-0.6,
        xlabel={$\beta$},  xlabel style={at=(current axis.right of origin), anchor=center,xshift=2.5em, yshift = -1.0em},
        ]
        \addplot[thick, red, dashed, thick] coordinates {(0,0)(1,-0.5)};
        \addplot[domain=0.001:2, blue, solid, thick] coordinates {
            (1.0000000e-03,2.0005003e+00)
            (1.1045226e-02,2.0055533e+00)
            (2.1090452e-02,2.0106576e+00)
            (3.1135678e-02,2.0158140e+00)
            (4.1180905e-02,2.0210233e+00)
            (5.1226131e-02,2.0262863e+00)
            (6.1271357e-02,2.0316039e+00)
            (7.1316583e-02,2.0369768e+00)
            (8.1361809e-02,2.0424060e+00)
            (9.1407035e-02,2.0478924e+00)
            (1.0145226e-01,2.0534368e+00)
            (1.1149749e-01,2.0590402e+00)
            (1.2154271e-01,2.0647035e+00)
            (1.3158794e-01,2.0704277e+00)
            (1.4163317e-01,2.0762138e+00)
            (1.5167839e-01,2.0820628e+00)
            (1.6172362e-01,2.0879757e+00)
            (1.7176884e-01,2.0939536e+00)
            (1.8181407e-01,2.0999975e+00)
            (1.9185930e-01,2.1061086e+00)
            (2.0190452e-01,2.1122880e+00)
            (2.1194975e-01,2.1185368e+00)
            (2.2199497e-01,2.1248562e+00)
            (2.3204020e-01,2.1312474e+00)
            (2.4208543e-01,2.1377117e+00)
            (2.5213065e-01,2.1442503e+00)
            (2.6217588e-01,2.1508644e+00)
            (2.7222111e-01,2.1575555e+00)
            (2.8226633e-01,2.1643249e+00)
            (2.9231156e-01,2.1711738e+00)
            (3.0235678e-01,2.1781038e+00)
            (3.1240201e-01,2.1851164e+00)
            (3.2244724e-01,2.1922129e+00)
            (3.3249246e-01,2.1993949e+00)
            (3.4253769e-01,2.2066639e+00)
            (3.5258291e-01,2.2140216e+00)
            (3.6262814e-01,2.2214696e+00)
            (3.7267337e-01,2.2290096e+00)
            (3.8271859e-01,2.2366432e+00)
            (3.9276382e-01,2.2443722e+00)
            (4.0280905e-01,2.2521984e+00)
            (4.1285427e-01,2.2601237e+00)
            (4.2289950e-01,2.2681500e+00)
            (4.3294472e-01,2.2762792e+00)
            (4.4298995e-01,2.2845132e+00)
            (4.5303518e-01,2.2928542e+00)
            (4.6308040e-01,2.3013042e+00)
            (4.7312563e-01,2.3098655e+00)
            (4.8317085e-01,2.3185401e+00)
            (4.9321608e-01,2.3273303e+00)
            (5.0326131e-01,2.3362386e+00)
            (5.1330653e-01,2.3452672e+00)
            (5.2335176e-01,2.3544187e+00)
            (5.3339698e-01,2.3636955e+00)
            (5.4344221e-01,2.3731003e+00)
            (5.5348744e-01,2.3826358e+00)
            (5.6353266e-01,2.3923045e+00)
            (5.7357789e-01,2.4021095e+00)
            (5.8362312e-01,2.4120535e+00)
            (5.9366834e-01,2.4221396e+00)
            (6.0371357e-01,2.4323709e+00)
            (6.1375879e-01,2.4427504e+00)
            (6.2380402e-01,2.4532814e+00)
            (6.3384925e-01,2.4639673e+00)
            (6.4389447e-01,2.4748115e+00)
            (6.5393970e-01,2.4858175e+00)
            (6.6398492e-01,2.4969891e+00)
            (6.7403015e-01,2.5083299e+00)
            (6.8407538e-01,2.5198439e+00)
            (6.9412060e-01,2.5315350e+00)
            (7.0416583e-01,2.5434074e+00)
            (7.1421106e-01,2.5554652e+00)
            (7.2425628e-01,2.5677130e+00)
            (7.3430151e-01,2.5801552e+00)
            (7.4434673e-01,2.5927964e+00)
            (7.5439196e-01,2.6056415e+00)
            (7.6443719e-01,2.6186955e+00)
            (7.7448241e-01,2.6319635e+00)
            (7.8452764e-01,2.6454508e+00)
            (7.9457286e-01,2.6591629e+00)
            (8.0461809e-01,2.6731055e+00)
            (8.1466332e-01,2.6872843e+00)
            (8.2470854e-01,2.7017056e+00)
            (8.3475377e-01,2.7163754e+00)
            (8.4479899e-01,2.7313004e+00)
            (8.5484422e-01,2.7464873e+00)
            (8.6488945e-01,2.7619429e+00)
            (8.7493467e-01,2.7776746e+00)
            (8.8497990e-01,2.7936896e+00)
            (8.9502513e-01,2.8099959e+00)
            (9.0507035e-01,2.8266014e+00)
            (9.1511558e-01,2.8435143e+00)
            (9.2516080e-01,2.8607435e+00)
            (9.3520603e-01,2.8782976e+00)
            (9.4525126e-01,2.8961862e+00)
            (9.5529648e-01,2.9144187e+00)
            (9.6534171e-01,2.9330053e+00)
            (9.7538693e-01,2.9519564e+00)
            (9.8543216e-01,2.9712827e+00)
            (9.9547739e-01,2.9909955e+00)
            (1.0055226e+00,1.4004423e+00)
            (1.0155678e+00,1.4012493e+00)
            (1.0256131e+00,1.4020596e+00)
            (1.0356583e+00,1.4028732e+00)
            (1.0457035e+00,1.4036900e+00)
            (1.0557487e+00,1.4045102e+00)
            (1.0657940e+00,1.4053337e+00)
            (1.0758392e+00,1.4061606e+00)
            (1.0858844e+00,1.4069908e+00)
            (1.0959296e+00,1.4078245e+00)
            (1.1059749e+00,1.4086616e+00)
            (1.1160201e+00,1.4095021e+00)
            (1.1260653e+00,1.4103461e+00)
            (1.1361106e+00,1.4111936e+00)
            (1.1461558e+00,1.4120445e+00)
            (1.1562010e+00,1.4128990e+00)
            (1.1662462e+00,1.4137571e+00)
            (1.1762915e+00,1.4146187e+00)
            (1.1863367e+00,1.4154840e+00)
            (1.1963819e+00,1.4163528e+00)
            (1.2064271e+00,1.4172253e+00)
            (1.2164724e+00,1.4181015e+00)
            (1.2265176e+00,1.4189813e+00)
            (1.2365628e+00,1.4198649e+00)
            (1.2466080e+00,1.4207522e+00)
            (1.2566533e+00,1.4216432e+00)
            (1.2666985e+00,1.4225381e+00)
            (1.2767437e+00,1.4234367e+00)
            (1.2867889e+00,1.4243392e+00)
            (1.2968342e+00,1.4252455e+00)
            (1.3068794e+00,1.4261557e+00)
            (1.3169246e+00,1.4270698e+00)
            (1.3269698e+00,1.4279878e+00)
            (1.3370151e+00,1.4289098e+00)
            (1.3470603e+00,1.4298358e+00)
            (1.3571055e+00,1.4307658e+00)
            (1.3671508e+00,1.4316998e+00)
            (1.3771960e+00,1.4326379e+00)
            (1.3872412e+00,1.4335800e+00)
            (1.3972864e+00,1.4345263e+00)
            (1.4073317e+00,1.4354767e+00)
            (1.4173769e+00,1.4364313e+00)
            (1.4274221e+00,1.4373900e+00)
            (1.4374673e+00,1.4383530e+00)
            (1.4475126e+00,1.4393203e+00)
            (1.4575578e+00,1.4402918e+00)
            (1.4676030e+00,1.4412676e+00)
            (1.4776482e+00,1.4422478e+00)
            (1.4876935e+00,1.4432323e+00)
            (1.4977387e+00,1.4442212e+00)
            (1.5077839e+00,1.4452146e+00)
            (1.5178291e+00,1.4462124e+00)
            (1.5278744e+00,1.4472146e+00)
            (1.5379196e+00,1.4482214e+00)
            (1.5479648e+00,1.4492327e+00)
            (1.5580101e+00,1.4502487e+00)
            (1.5680553e+00,1.4512692e+00)
            (1.5781005e+00,1.4522943e+00)
            (1.5881457e+00,1.4533241e+00)
            (1.5981910e+00,1.4543586e+00)
            (1.6082362e+00,1.4553979e+00)
            (1.6182814e+00,1.4564419e+00)
            (1.6283266e+00,1.4574907e+00)
            (1.6383719e+00,1.4585444e+00)
            (1.6484171e+00,1.4596029e+00)
            (1.6584623e+00,1.4606663e+00)
            (1.6685075e+00,1.4617346e+00)
            (1.6785528e+00,1.4628079e+00)
            (1.6885980e+00,1.4638862e+00)
            (1.6986432e+00,1.4649696e+00)
            (1.7086884e+00,1.4660580e+00)
            (1.7187337e+00,1.4671515e+00)
            (1.7287789e+00,1.4682502e+00)
            (1.7388241e+00,1.4693540e+00)
            (1.7488693e+00,1.4704631e+00)
            (1.7589146e+00,1.4715774e+00)
            (1.7689598e+00,1.4726970e+00)
            (1.7790050e+00,1.4738219e+00)
            (1.7890503e+00,1.4749522e+00)
            (1.7990955e+00,1.4760879e+00)
            (1.8091407e+00,1.4772291e+00)
            (1.8191859e+00,1.4783757e+00)
            (1.8292312e+00,1.4795279e+00)
            (1.8392764e+00,1.4806856e+00)
            (1.8493216e+00,1.4818489e+00)
            (1.8593668e+00,1.4830179e+00)
            (1.8694121e+00,1.4841926e+00)
            (1.8794573e+00,1.4853730e+00)
            (1.8895025e+00,1.4865591e+00)
            (1.8995477e+00,1.4877511e+00)
            (1.9095930e+00,1.4889489e+00)
            (1.9196382e+00,1.4901526e+00)
            (1.9296834e+00,1.4913623e+00)
            (1.9397286e+00,1.4925779e+00)
            (1.9497739e+00,1.4937996e+00)
            (1.9598191e+00,1.4950273e+00)
            (1.9698643e+00,1.4962612e+00)
            (1.9799095e+00,1.4975012e+00)
            (1.9899548e+00,1.4987475e+00)
            (2.0000000e+00,1.5000000e+00)
          };
        \legend{
          $\theta$,
          $1+1/(1+\theta)$,
        };
        \addplot[thick, red, dashed, thick] coordinates {(1,1.5)(2,1.0)};
        \addplot[thick, dotted] coordinates {(0,+3.0)(1,+3.0)};
        \addplot[thick, dotted] coordinates {(0,+1.5)(2,+1.5)};
        \addplot[thick, dotted] coordinates {(0,+1.0)(2,+1.0)};
        \addplot[thick, dotted] coordinates {(0,-0.5)(1,-0.5)};
        \addplot[very thick, dashed, white] coordinates {(1,-0.5)(1,3)};
        \coordinate (o1) at (axis cs:0.87,1.27);
        \coordinate (o2) at (axis cs:0.87,1.16);
        \coordinate (d1) at (axis cs:0.86,-0.54);
        \coordinate (d2) at (axis cs:0.86,2.60);
        \coordinate (d3) at (axis cs:1.75,1.25);
      \end{axis}
      \draw[thick,red ] (o1) circle(0.15);
      \draw[thick,blue] (o2) circle(0.15);
      \filldraw[thick,red ] (d1) circle(0.15);
      \filldraw[thick,blue] (d2) circle(0.15);
      \filldraw[thick,blue] (d3) circle(0.15);
    \end{tikzpicture}
    \hfill
    \begin{tikzpicture}[scale=0.9, transform shape, x=1em, y=1em]
      \tikzset{>=latex}
      \begin{axis}[
        axis lines = left,
        xmin=0, xmax=2.1,
        ymin=0.15, ymax=0.55,
        legend pos=south east,
        ytick={        0.15, 0.199471,                 0.25,   0.3,   0.4,   0.5},
        yticklabels={$0.15$, $\frac{1}{2\sqrt{2\pi}}$, $0.25$, $0.3$, $0.4$, $\sfrac{1}{2}$},
        xlabel={$\beta$}, xlabel style={at=(current axis.right of origin), anchor=center,xshift=2.5em, yshift = +1.2em},
        ylabel={$\Theta_{\beta,2}$}, ylabel style={at=(current axis.above origin), anchor=south, rotate = -90, xshift = 3.5em, yshift = 1.5em},
        ]
        \addplot[domain=0:2, id=Theta, blue, thick] coordinates {
( 0.001,0.2499278584589625 )
( 0.011,0.24920759534174453 )
( 0.021,0.24848944988662128 )
( 0.031,0.24777345502817438 )
( 0.041,0.2470596445788786 )
( 0.051000000000000004,0.24634805324228332 )
( 0.061,0.2456387166267695 )
( 0.07100000000000001,0.2449316712599644 )
( 0.081,0.24422695460374888 )
( 0.091,0.24352460506996437 )
( 0.101,0.24282466203678296 )
( 0.111,0.24212716586576 )
( 0.121,0.24143215791972278 )
( 0.131,0.2407396720889704 )
( 0.14100000000000001,0.24004977727180168 )
( 0.151,0.2393624924723805 )
( 0.161,0.23867787174393348 )
( 0.171,0.23799596174939874 )
( 0.181,0.2373168102765295 )
( 0.191,0.23664046626125856 )
( 0.201,0.2359669798124457 )
( 0.211,0.2352964022374847 )
( 0.221,0.23462878606844773 )
( 0.231,0.23396418509070643 )
( 0.241,0.23330265437069986 )
( 0.251,0.23264425028635016 )
( 0.261,0.2319890305580688 )
( 0.271,0.2313370542811012 )
( 0.281,0.23068838195915942 )
( 0.291,0.2300430755393948 )
( 0.301,0.22940119844876455 )
( 0.311,0.2287628156318487 )
( 0.321,0.22812799359017633 )
( 0.331,0.22749680042312473 )
( 0.341,0.22686930587045584 )
( 0.35100000000000003,0.2262455813565596 )
( 0.361,0.22562570003629356 )
( 0.371,0.22500973684358025 )
( 0.381,0.22439772476328984 )
( 0.391,0.22378987376986592 )
( 0.401,0.22318613310563112 )
( 0.41100000000000003,0.2225866291153952 )
( 0.421,0.22199144641424576 )
( 0.431,0.22140067172670455 )
( 0.441,0.22081439395026106 )
( 0.451,0.22023270422131141 )
( 0.461,0.21965569598347232 )
( 0.47100000000000003,0.2190834650596854 )
( 0.481,0.21851610972657215 )
( 0.491,0.21795373079202557 )
( 0.501,0.21739643167621298 )
( 0.511,0.21684431849643448 )
( 0.521,0.21629750015457194 )
( 0.531,0.2157560884294979 )
( 0.541,0.21522019807256026 )
( 0.551,0.21468994690760496 )
( 0.561,0.2141654559354208 )
( 0.5710000000000001,0.2136468494424803 )
( 0.581,0.2131342551148585 )
( 0.591,0.21262780415687582 )
( 0.601,0.21212763141531546 )
( 0.611,0.21163387550897583 )
( 0.621,0.2111466789643175 )
( 0.631,0.2106661883570975 )
( 0.641,0.21019255446062313 )
( 0.651,0.2097259324007935 )
( 0.661,0.20926648181824842 )
( 0.671,0.2088143670384022 )
( 0.681,0.208369757251612 )
( 0.6910000000000001,0.20793282668909147 )
( 0.7010000000000001,0.20750375483376599 )
( 0.711,0.20708272661417626 )
( 0.721,0.20666993262148692 )
( 0.731,0.20626556933245627 )
( 0.741,0.20586983935306702 )
( 0.751,0.20548295165451258 )
( 0.761,0.20510512184042293 )
( 0.771,0.20473657241880694 )
( 0.781,0.20437753308725998 )
( 0.791,0.20402824103407957 )
( 0.801,0.20368894125794695 )
( 0.811,0.20335988689863316 )
( 0.8210000000000001,0.2030413395880518 )
( 0.8310000000000001,0.20273356981962193 )
( 0.841,0.20243685733665776 )
( 0.851,0.20215149154205392 )
( 0.861,0.2018777719295763 )
( 0.871,0.20161600853903375 )
( 0.881,0.2013665224358182 )
( 0.891,0.20112964621793192 )
( 0.901,0.20090572455063543 )
( 0.911,0.20069511473112567 )
( 0.921,0.20049818728553542 )
( 0.931,0.2003153266008827 )
( 0.9410000000000001,0.20014693159153174 )
( 0.9510000000000001,0.19999341640756396 )
( 0.961,0.19985515053839592 )
( 0.971,0.19973276282196262 )
( 0.981,0.19962653584664747 )
( 0.991,0.199537013312883 )
( 1.001,0.29296538201340894 )
( 1.011,0.31243580312047625 )
( 1.021,0.31320715987950265 )
( 1.031,0.31398382831011545 )
( 1.041,0.3147659217515195 )
( 1.051,0.31555356357750486 )
( 1.061,0.3163468811674267 )
( 1.071,0.3171460041599756 )
( 1.081,0.31795106471865453 )
( 1.091,0.31876220178596126 )
( 1.101,0.3195795579369042 )
( 1.111,0.32040327924085843 )
( 1.121,0.32123349727691675 )
( 1.131,0.32207042412871517 )
( 1.141,0.3229141643051395 )
( 1.151,0.32376490194217156 )
( 1.1609999999999998,0.3246228078103392 )
( 1.1709999999999998,0.3254880587129584 )
( 1.1809999999999998,0.3263608383516048 )
( 1.1909999999999998,0.32724132980921217 )
( 1.2009999999999998,0.328129732662407 )
( 1.2109999999999999,0.329026246109545 )
( 1.2209999999999999,0.32993107789151277 )
( 1.2309999999999999,0.3308444428442973 )
( 1.2409999999999999,0.33176682407167113 )
( 1.251,0.3326976694666959 )
( 1.261,0.33363799806912603 )
( 1.271,0.3345877971780971 )
( 1.281,0.3355473217026117 )
( 1.291,0.3365168362162073 )
( 1.301,0.3374966151223352 )
( 1.311,0.3384869429847124 )
( 1.321,0.3394881150476506 )
( 1.331,0.3405004377846164 )
( 1.341,0.3415242299621769 )
( 1.351,0.3425598205916055 )
( 1.361,0.34360755493917466 )
( 1.371,0.3446677898174333 )
( 1.381,0.34574089701751487 )
( 1.391,0.34682726354758164 )
( 1.401,0.3479272924698711 )
( 1.4109999999999998,0.3490414037899192 )
( 1.4209999999999998,0.350170035417086 )
( 1.4309999999999998,0.3513136441893233 )
( 1.4409999999999998,0.3524727069743191 )
( 1.4509999999999998,0.35364772182228293 )
( 1.4609999999999999,0.35483973017416004 )
( 1.4709999999999999,0.356047713616795 )
( 1.4809999999999999,0.3572738045540336 )
( 1.4909999999999999,0.35851807873376096 )
( 1.501,0.35978116151501727 )
( 1.511,0.3610637090313297 )
( 1.521,0.3623664103641902 )
( 1.531,0.36369000406195723 )
( 1.541,0.365035227604249 )
( 1.551,0.36640289220325184 )
( 1.561,0.36779384208142635 )
( 1.571,0.369208967591851 )
( 1.581,0.37064920872794016 )
( 1.591,0.37211555895306425 )
( 1.601,0.37360906946061617 )
( 1.611,0.37513098864583516 )
( 1.621,0.3766820933868883 )
( 1.631,0.3782640426286273 )
( 1.641,0.3798780359686803 )
( 1.651,0.38152549471721314 )
( 1.661,0.38320793374897755 )
( 1.6709999999999998,0.38492697579497726 )
( 1.6809999999999998,0.3866843549830006 )
( 1.6909999999999998,0.3884819312184187 )
( 1.7009999999999998,0.3903217018605808 )
( 1.7109999999999999,0.3922058156215615 )
( 1.7209999999999999,0.3941365888703986 )
( 1.7309999999999999,0.3961165348391685 )
( 1.7409999999999999,0.39814839297168575 )
( 1.751,0.4002349942512119 )
( 1.761,0.4023797261084938 )
( 1.771,0.4045861020043651 )
( 1.781,0.4068581166588208 )
( 1.791,0.40919977390492257 )
( 1.801,0.41161584290503483 )
( 1.811,0.4141112414629301 )
( 1.821,0.41669244796055044 )
( 1.831,0.419366144486123 )
( 1.841,0.4221395440113091 )
( 1.851,0.42501690796881086 )
( 1.861,0.4280195384436768 )
( 1.871,0.4311466975340487 )
( 1.881,0.43441589358591304 )
( 1.891,0.43784222813261486 )
( 1.901,0.4414439489014064 )
( 1.911,0.44524326133719283 )
( 1.9209999999999998,0.4492678348596698 )
( 1.9309999999999998,0.45355234085036583 )
( 1.9409999999999998,0.45814226501890215 )
( 1.9509999999999998,0.4630988690117351 )
( 1.9609999999999999,0.468510305945018 )
( 1.9709999999999999,0.47450825969685473 )
( 1.9809999999999999,0.4813304839161856 )
( 1.9909999999999999,0.4894966882184091 )
( 2.,0.5 )
        };
        \addplot[thick, dotted] coordinates {(0,+0.199471) (1,+0.199471)} node [pos=0.5,below] {\small SHE: $\frac{1}{2\sqrt{2\pi}}\approx 0.199$ };
        \addplot[thick, dotted] coordinates {(0,+0.5) (2,+0.5)} node [pos=0.5,above] {\small SWE};
        \addplot[very thick, dashed, white] coordinates {(1,+0.2) (1,+0.4)};
        \coordinate (o) at (axis cs:0.875,0.292);
        \coordinate (d1) at (axis cs:0.86,0.192);
        \coordinate (d2) at (axis cs:1.76,0.460);
      \end{axis}
      \draw[thick] (o) circle(0.2);
      \filldraw[thick,blue] (d1) circle(0.15);
      \filldraw[thick,blue] (d2) circle(0.15);
    \end{tikzpicture}

    \begin{tikzpicture}[scale=0.9, transform shape, x=1em, y=1em]
      \tikzset{>=latex}
      \begin{axis}[
        xmin=0,
        title = {Second moment Lyapunov exponent},
        xlabel={$\beta$},
        ymax = 0.78,
        ytick={0, 0.125, 0.25, 0.4, 0.6, 0.7071},
        yticklabels={$0$, {\small SHE:} $\sfrac{1}{8}$, $0.25$, $0.4$, $0.6$, {\small SWE:} $\sfrac{1}{\sqrt{2}}$},
        ],
        \addplot[domain=0:2, id=Theta, black, thick] coordinates {
( 0.001,0.24982673803399943 )
( 0.011,0.24809686383434862 )
( 0.021,0.24637203137151414 )
( 0.031,0.2446523131466168 )
( 0.041,0.2429377830823012 )
( 0.051000000000000004,0.24122851654851749 )
( 0.061,0.23952459038936758 )
( 0.07100000000000001,0.23782608295113444 )
( 0.081,0.2361330741114653 )
( 0.091,0.2344456453098552 )
( 0.101,0.23276387957943967 )
( 0.111,0.2310878615801582 )
( 0.121,0.22941767763349524 )
( 0.131,0.22775340716129786 )
( 0.14100000000000001,0.2260951657099051 )
( 0.151,0.2244430190171898 )
( 0.161,0.22279706902634616 )
( 0.171,0.22115741094283917 )
( 0.181,0.2195241418770592 )
( 0.191,0.2178973608914046 )
( 0.201,0.21627716905003422 )
( 0.211,0.21466366947084997 )
( 0.221,0.21305696737947755 )
( 0.231,0.2114571701673282 )
( 0.241,0.2098643874504579 )
( 0.251,0.20827873113291617 )
( 0.261,0.2067003154726155 )
( 0.271,0.20512925715062508 )
( 0.281,0.20356567534398795 )
( 0.291,0.20200969180227515 )
( 0.301,0.20046143092810864 )
( 0.311,0.19892101986189556 )
( 0.321,0.19738858857103841 )
( 0.331,0.1958642699439018 )
( 0.341,0.1943481998888379 )
( 0.35100000000000003,0.19284051743859362 )
( 0.361,0.19134136486025677 )
( 0.371,0.18985088777223721 )
( 0.381,0.18836918986972964 )
( 0.391,0.18689656003779975 )
( 0.401,0.18543301851947128 )
( 0.41100000000000003,0.18397877103257493 )
( 0.421,0.18253398193568962 )
( 0.431,0.1810988197881722 )
( 0.441,0.179673457522312 )
( 0.451,0.17825807262553547 )
( 0.461,0.1768528473332171 )
( 0.47100000000000003,0.1754579688342199 )
( 0.481,0.17407362948825159 )
( 0.491,0.1727000270568215 )
( 0.501,0.17133736494880109 )
( 0.511,0.16998585248194512 )
( 0.521,0.16864570516001182 )
( 0.531,0.1673171449690229 )
( 0.541,0.16600040069182675 )
( 0.551,0.16469570824374838 )
( 0.561,0.16340331103056172 )
( 0.5710000000000001,0.1621234603301292 )
( 0.581,0.16085641570025436 )
( 0.591,0.15960244541402582 )
( 0.601,0.15836182692547932 )
( 0.611,0.157134847367421 )
( 0.621,0.15592180408452008 )
( 0.631,0.1547230052040716 )
( 0.641,0.1535387702478572 )
( 0.651,0.15236943078830584 )
( 0.661,0.15121533115260744 )
( 0.671,0.1500768291792534 )
( 0.681,0.14895429703344112 )
( 0.6910000000000001,0.14784812207013862 )
( 0.7010000000000001,0.1467587077921053 )
( 0.711,0.1456864748520093 )
( 0.721,0.14463186214787802 )
( 0.731,0.14359532799979402 )
( 0.741,0.142577351431083 )
( 0.751,0.14157843353089025 )
( 0.761,0.14059909894952044 )
( 0.771,0.13963989750723413 )
( 0.781,0.13870140593660937 )
( 0.791,0.13778422977337312 )
( 0.801,0.1368890054120122 )
( 0.811,0.13601640233257284 )
( 0.8210000000000001,0.13516712552552076 )
( 0.8310000000000001,0.1343419181308387 )
( 0.841,0.13354156431272277 )
( 0.851,0.13276689239543715 )
( 0.861,0.13201877828642966 )
( 0.871,0.13129814921816416 )
( 0.881,0.13060598784163402 )
( 0.891,0.12994333671156968 )
( 0.901,0.1293113032045164 )
( 0.911,0.12871106491878948 )
( 0.921,0.12814387561117438 )
( 0.931,0.12761107173248662 )
( 0.9410000000000001,0.12711407962828697 )
( 0.9510000000000001,0.1266544234893273 )
( 0.961,0.12623366040166104 )
( 0.971,0.1258537587283821 )
( 0.981,0.1255163715542514 )
( 0.991,0.1252235860111366 )
( 1.001,0.6856258048218072 )
( 1.011,0.7020212019650203 )
( 1.021,0.7012322681101243 )
( 1.031,0.700447583913828 )
( 1.041,0.6996672405153186 )
( 1.051,0.6988913375577798 )
( 1.061,0.6981199777325718 )
( 1.071,0.6973532651961056 )
( 1.081,0.6965913057936498 )
( 1.091,0.6958342108282703 )
( 1.101,0.6950820942221368 )
( 1.111,0.6943350723716419 )
( 1.121,0.6935932482080965 )
( 1.131,0.6928567952472221 )
( 1.141,0.6921257910895491 )
( 1.151,0.6914003830788911 )
( 1.1609999999999998,0.6906807059624737 )
( 1.1709999999999998,0.6899668990203066 )
( 1.1809999999999998,0.6892591067840634 )
( 1.1909999999999998,0.688557472413208 )
( 1.2009999999999998,0.6878621526528927 )
( 1.2109999999999999,0.6871733030030994 )
( 1.2209999999999999,0.6864910854925299 )
( 1.2309999999999999,0.6858156673620669 )
( 1.2409999999999999,0.6851474476863361 )
( 1.251,0.6844859260440515 )
( 1.261,0.6838319645035664 )
( 1.271,0.6831855285334911 )
( 1.281,0.682546814644468 )
( 1.291,0.6819160264527947 )
( 1.301,0.6812933747543547 )
( 1.311,0.6806790777381999 )
( 1.321,0.6800733613594743 )
( 1.331,0.6794764597314078 )
( 1.341,0.6788886159517231 )
( 1.351,0.6783100802677795 )
( 1.361,0.6777411151120581 )
( 1.371,0.6771819909843034 )
( 1.381,0.6766329892733518 )
( 1.391,0.6760944023457158 )
( 1.401,0.6755665341364041 )
( 1.4109999999999998,0.6750497007750031 )
( 1.4209999999999998,0.6745442312623957 )
( 1.4309999999999998,0.6740504681912524 )
( 1.4409999999999998,0.6735687685195768 )
( 1.4509999999999998,0.6730995043755634 )
( 1.4609999999999999,0.6726434990676005 )
( 1.4709999999999999,0.6721998525079811 )
( 1.4809999999999999,0.6717702932878425 )
( 1.4909999999999999,0.6713548288549898 )
( 1.501,0.6709539221925384 )
( 1.511,0.670568058194739 )
( 1.521,0.670197745198167 )
( 1.531,0.6698435283119718 )
( 1.541,0.6695059474215617 )
( 1.551,0.669185598580588 )
( 1.561,0.6688831000114369 )
( 1.571,0.6685991025244441 )
( 1.581,0.6683342919574413 )
( 1.591,0.6680893918318028 )
( 1.601,0.6678651663095941 )
( 1.611,0.6676625327571223 )
( 1.621,0.6674820187771088 )
( 1.631,0.6673248593244688 )
( 1.641,0.6671919080095156 )
( 1.651,0.6670841886018082 )
( 1.661,0.6670027901458015 )
( 1.6709999999999998,0.6669488774218748 )
( 1.6809999999999998,0.666923692635587 )
( 1.6909999999999998,0.6669285656703988 )
( 1.7009999999999998,0.6669649219908811 )
( 1.7109999999999999,0.6670342921299266 )
( 1.7209999999999999,0.6671383228659504 )
( 1.7309999999999999,0.6672787983298861 )
( 1.7409999999999999,0.6674576608056402 )
( 1.751,0.667676902293143 )
( 1.761,0.6679389271898741 )
( 1.771,0.6682462097908259 )
( 1.781,0.6686015689542741 )
( 1.791,0.6690077953465687 )
( 1.801,0.6694682361639173 )
( 1.811,0.6699863092059388 )
( 1.821,0.6705665872174685 )
( 1.831,0.671213736348228 )
( 1.841,0.6719327617984916 )
( 1.851,0.6727263417046416 )
( 1.861,0.6736100647948531 )
( 1.871,0.6745825903853518 )
( 1.881,0.6756563532314029 )
( 1.891,0.6768417749397418 )
( 1.901,0.6781514484182145 )
( 1.911,0.6796006860767431 )
( 1.9209999999999998,0.6812085632083406 )
( 1.9309999999999998,0.6829989463948455 )
( 1.9409999999999998,0.6850031513795889 )
( 1.9509999999999998,0.6872633378686754 )
( 1.9609999999999999,0.6898402340228004 )
( 1.9709999999999999,0.6928245067190337 )
( 1.9809999999999999,0.6963805466604955 )
( 1.9909999999999999,0.7008688236199396 )
( 2,0.7071 )
        };
        \addplot[thick, dotted] coordinates {(0,+0.125) (1,+0.125)};
        \addplot[thick, dotted] coordinates {(0,+0.707) (2,+0.707)};
        \addplot[very thick, dashed,white] coordinates {(1,+0.125) (1,+0.707)};
        \coordinate (o) at (axis cs:0.85,0.63);
        \coordinate (d1) at (axis cs:1.74,0.63);
        \coordinate (d2) at (axis cs:0.825,0.105);
      \end{axis}
      \draw[thick] (o) circle(0.2);
      \filldraw[thick] (d1) circle(0.15);
      \filldraw[thick] (d2) circle(0.15);
    \end{tikzpicture}
  \end{center}

  \caption{Plots of the quantities in Example \ref{Ex:TFSPDE} with $\lambda=1$ and $\nu=2$. For all
  these graphs, at the jump point $\beta=1$, one needs to take the left limit.}

  \label{F:TFSPDE}
\end{figure}

%
%
%
%

\begin{example} \label{Ex:Nane}
  Mijena and Nane \cite{mijena.nane:15:space-time} studied the case when $\beta\in(0,1]$, $\alpha\in
  (0,2]$, $\gamma=1-\beta$, namely,
  \begin{equation}\label{E:Nane}
  \begin{cases}
    \left(\partial_t^{\beta}+\dfrac{\nu}{2}\left(-\Delta\right)^{\alpha / 2}\right) u(t, x)= \: I_{t}^{1-\beta}\left[\lambda u(t, x) \dot{W}(t, x)\right] , & t>0, x \in \mathbb{R}^{d},   \\
    u(0,\cdot)=u_0,                                                                                                                                       &
  \end{cases}
  \end{equation}
  under the condition
  \begin{align} \label{E:CondNane}
    d < \alpha\:\min\left(2,\beta^{-1}\right).
  \end{align}
  Note that
  condition \eqref{E:CondNane} is the same as \eqref{E:Dalang'} under this specific setting. In \cite{mijena.nane:15:space-time}, the upper bound of the large-time exponent \eqref{E:upper-tlim} was obtained; see Theorem 2
  {\it ibid.} Since the fundamental solution in this case is nonnegative (see Remark
  \ref{R:Nonneg}), we can apply Theorem \ref{T:fde} to have exact formulas for both the second
  moment and the second moment Lyapunov exponent, and to have matching lower bounds for the moment
  asymptotics. To be more precise, in this case we have
  \begin{align*}
    \theta                      & = -\beta d /\alpha,                                                                         & t_p         & = p^{\frac{2\alpha-\beta d}{\alpha - \beta d}} t, \\
    \Theta_{\alpha,\beta,d,\nu} & = \frac{1}{(2\pi)^d}\int_{\R^d} E_{\beta}^2 \left(-\frac{\nu}{2}|\xi|^\alpha\right) \ud\xi, & \widehat{t} & = \Theta_{\alpha,\beta,d,\nu} \Gamma\left(1-\beta d/\alpha\right) t^{1-\beta d/\alpha}.
  \end{align*}
  and hence we have the following results:
  \begin{enumerate}[wide=1em]
    \item {\it Second moment formula:}
      \begin{align} \label{E:2ndNane}
         \E\left[u^2(t,x)\right] =
         u_0^2 \:E_{1-\beta d/\alpha}  \left(\lambda^2 \Theta_{\alpha,\beta,d,\nu} \: \Gamma\left(1-\beta d/\alpha\right) t^{1 -\beta d/\alpha}\right).
      \end{align}
    \item {\it Second moment Lyapunov exponent:}
      \begin{align} \label{E:LyNane}
          \lim_{t\to \infty} t^{-1} \log \E\left[u(t,x)^2\right]
          = \left(\lambda^2 \Theta_{\alpha,\beta,d,\nu}\: \Gamma\left(1-\beta d/\alpha\right)\right)^{\frac{\alpha}{\alpha-\beta d}}.
      \end{align}
    \item {\it Moment asymptotics:}
      \begin{subequations} \label{E:AsymNane}
        \begin{align} \label{E:AsymNane-t}
               C_1 p^{\frac{2\alpha-\beta d}{\alpha-\beta d}}
           \le & \liminf_{t\to \infty} \frac{\log \E[u(t,x)^p]}{t}
           \le   \limsup_{t\to \infty} \frac{\log \E[u(t,x)^p]}{t}
           \le C_2 p^{\frac{2\alpha-\beta d}{\alpha-\beta d}}, & p\ge 2, \\
               C_3 \: t
           \le & \liminf_{p\to \infty} \frac{\log \E[u(t,x)^p]}{p^{\frac{2\alpha-\beta d}{\alpha-\beta d}}}
           \le   \limsup_{p\to \infty} \frac{\log \E[u(t,x)^p]}{p^{\frac{2\alpha-\beta d}{\alpha-\beta d}}}
           \le C_4 \: t, & t>0. \label{E:AsymNane-p}
        \end{align}
      \end{subequations}
  \end{enumerate}
  Note that except the two lower bounds in \eqref{E:AsymNane-t} and \eqref{E:AsymNane-p} require
  $\alpha\in (0,2]$, all the rest formulas/upper bounds in this example hold true for all
  $\alpha>0$. In particular, this would allow higher dimensions for large $\alpha$; see
  \eqref{E:CondNane}.
\end{example}

\begin{example} \label{Ex:SHESWE}
  In this example, we study the following one-parameter family of SPDEs with SHE \eqref{E:SHE}
  (resp. SWE \eqref{E:SWE}) being a special (resp. limiting) case:
  \begin{equation}\label{E:SHESWE}
    \left(\partial_t^{\beta}-\dfrac{\nu}{2}\dfrac{\partial^2}{\partial x^2}\right) u(t, x)= \: \lambda u(t, x) \dot{W}(t, x) ,
    \quad  t>0,\: x \in \mathbb{R},\: \beta\in \left(0,2\right),
  \end{equation}
  with the same initial condition as SHE \eqref{E:SHE} (resp. SWE \eqref{E:SWE}) when $\beta\in
  (0,1]$ (resp. $\beta\in(1,2)$). This is the case when $\alpha=2$, $\beta\in (0,2)$, $\gamma=0$ and
  $d=1$. Dalang's condition \eqref{E:Dalang'} reduces to
  \begin{align*}
    \beta > 2/3,
  \end{align*}
  and quantities in \eqref{E:theta} become
  \begin{align*}
  \begin{array}{lcl}
    \theta := -2+3 \beta /2,                                                                                 & \quad \quad & t_p:= p^{3\beta / (3\beta-2)} t, \\ [0.5em]
    \displaystyle \Theta_{\beta,\nu} := \pi^{-1}\int_0^\infty E^2_{\beta, \beta}(-2^{-1} \nu \xi^2) \ud \xi, & \quad \quad & \widehat{t}:= \Theta_{\beta,\nu}\: \Gamma\left(-1+3\beta /2\right)t^{-1+3\beta /2}.
  \end{array}
  \end{align*}
  Note that the fundamental solution is nonnegative (see Remark \ref{R:Nonneg}). Here we summarize
  the properties of the solution to \eqref{E:SHESWE} as follows:
  \begin{enumerate}[wide=1em]
    \item {\it Second moment formula:}
      \begin{align} \label{E:2ndSHESWE}
        \hspace{-2em}
        \E\left[u^2(t,x)\right] =
        \begin{cases}
          \quad \quad u_0^2 \:E_{-1+3\beta /2}  \left(\lambda^2 \Theta_{\beta,\nu} \: \Gamma\left(-1+3\beta /2\right)t^{-1+3\beta /2}\right) & \text{if $\beta\in (0,1]$},\\[1em]
        \begin{aligned}
           &  & u_0^2          \:E_{-1+3\beta /2}    \left(\lambda^2 \Theta_{\beta,\nu} \: \Gamma\left(-1+3\beta /2\right)t^{-1+3\beta /2}\right) \\
           &  & + 2u_0u_1\: t  \:E_{-1+3\beta /2,2}  \left(\lambda^2 \Theta_{\beta,\nu} \: \Gamma\left(-1+3\beta /2\right)t^{-1+3\beta /2}\right) \\
           &  & + 2u_1^2 \: t^2\:E_{-1+3\beta /2,3}  \left(\lambda^2 \Theta_{\beta,\nu} \: \Gamma\left(-1+3\beta /2\right)t^{-1+3\beta /2}\right)
        \end{aligned} & \text{if $\beta\in (1,2)$}.
        \end{cases}
      \end{align}
    \item {\it Second moment Lyapunov exponent:}
      \begin{align} \label{E:LySHESWE}
          \lim_{t\to \infty} t^{-1} \log \E\left[u(t,x)^2\right]
          = \left(\lambda^2 \Theta_{\beta,\nu} \: \Gamma\left(-1+3\beta /2\right)\right)^{2/(3\beta-2)}.
      \end{align}
    \item {\it Moment asymptotics:}
      \begin{subequations} \label{E:AsymSHESWE}
        \begin{align} \label{E:AsymSHESWE-t}
               C_1 p^{\frac{3\beta}{3\beta-2}}
           \le & \liminf_{t\to \infty} \frac{\log \E[u(t,x)^p]}{t}
           \le   \limsup_{t\to \infty} \frac{\log \E[u(t,x)^p]}{t}
           \le C_2 p^{\frac{3\beta}{3\beta-2}}, & p\ge 2, \\
               C_3 \: t
           \le & \liminf_{p\to \infty} \frac{\log \E[u(t,x)^p]}{p^{3\beta / (3\beta-2)}}
           \le   \limsup_{p\to \infty} \frac{\log \E[u(t,x)^p]}{p^{3\beta / (3\beta-2)}}
           \le C_4 \: t, & t>0. \label{E:AsymSHESWE-p}
        \end{align}
      \end{subequations}
  \end{enumerate}
  Thanks to \eqref{E:ML-ex}, all the above quantities when $\beta\to 2$ converge to the
  corresponding ones in Example \ref{Ex:SWE} for SWE \eqref{E:SWE}; see Figure \ref{F:SHESWE} for
  some numerical computations.
\end{example}

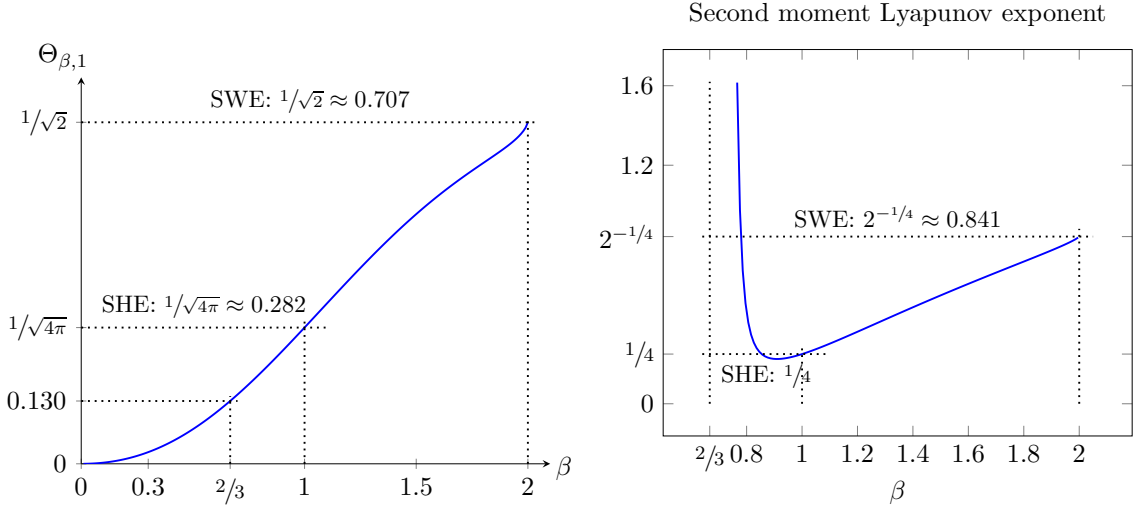
\begin{figure}[htpb]
  \begin{center}
    \begin{tikzpicture}[scale=0.9, transform shape, x=1em, y=1em]
      \tikzset{>=latex}
      \begin{axis}[
        ymax = 0.80,
        xmax = 2.10,
        axis lines = left,
        ytick={        0, 0.130,   0.282,                    0.707},
        yticklabels={$0$, $0.130$, $\sfrac{1}{\sqrt{4\pi}}$, $\sfrac{1}{\sqrt{2}}$},
        xtick={      0, 0.3, 0.667,          1, 1.5, 2},
        xticklabels={0, 0.3, $\sfrac{2}{3}$, 1, 1.5, 2},
        xlabel={$\beta$},
        ylabel={$\Theta_{\beta,1}$},
        xlabel style={at=(current axis.right of origin), anchor=center, xshift=2.3em, yshift=1.2em},
        ylabel style={at=(current axis.above origin), anchor=south, rotate = -90, xshift=2.5em,yshift=1.4em},
        ]
        \addplot[domain=0:2, id=Theta, blue, thick] coordinates {
            (0.01,0.000022287501866342994)
            (0.02,0.00008990427773282904)
            (0.03,0.0002039646581879306)
            (0.04,0.0003655594279086583)
            (0.05,0.0005757548303229392)
            (0.060000000000000005,0.0008355918967806942)
            (0.06999999999999999,0.001146085081277085)
            (0.08,0.0015082218996128206)
            (0.09,0.0019229618745259872)
            (0.09999999999999999,0.0023913696112699057)
            (0.11,0.00291394502813121)
            (0.12,0.0034919606983066158)
            (0.13,0.004126123123569873)
            (0.14,0.004817241119437536)
            (0.15000000000000002,0.005566091405619505)
            (0.16,0.006373418026664055)
            (0.17,0.007239931830828584)
            (0.18000000000000002,0.008166309951674465)
            (0.19,0.009153195347816695)
            (0.2,0.010200836506971385)
            (0.21000000000000002,0.011310886348145166)
            (0.22,0.012482803251630701)
            (0.23,0.013717449330928319)
            (0.24000000000000002,0.01501529083207566)
            (0.25,0.016376757728646205)
            (0.26,0.017802243489804796)
            (0.27,0.019292104880301156)
            (0.28,0.020846661793637193)
            (0.29000000000000004,0.022466197117414526)
            (0.3,0.02415095662966464)
            (0.31,0.025901148927363647)
            (0.32,0.027716945386041674)
            (0.33,0.02959848014912575)
            (0.34,0.031545850148432586)
            (0.35000000000000003,0.033559115154061465)
            (0.36000000000000004,0.035638297863984994)
            (0.37,0.037783383958356634)
            (0.38,0.03999432234417212)
            (0.39,0.04227102521027182)
            (0.4,0.044613371900119214)
            (0.41000000000000003,0.04702119098826523)
            (0.42000000000000004,0.04949429678483229)
            (0.43,0.05203245334607863)
            (0.44,0.05463539289911257)
            (0.45,0.0573028125356594)
            (0.46,0.06003437455571523)
            (0.47000000000000003,0.06282970683409707)
            (0.48000000000000004,0.06568840321037178)
            (0.49,0.06861002390041672)
            (0.5,0.07159409592967741)
            (0.51,0.07464011358684042)
            (0.52,0.07774753889823663)
            (0.53,0.08091580212153185)
            (0.54,0.08414430225848119)
            (0.55,0.08743240758558582)
            (0.56,0.09077945620279909)
            (0.5700000000000001,0.09418475659929203)
            (0.5800000000000001,0.09764758823499876)
            (0.59,0.1011672021381523)
            (0.6,0.10474282145902922)
            (0.61,0.10837364239158424)
            (0.62,0.11205883422425225)
            (0.63,0.11579754058323088)
            (0.64,0.11958887980374364)
            (0.65,0.12343194566706152)
            (0.66,0.1273258080891263)
            (0.67,0.1312695138200322)
            (0.68,0.13526208715329022)
            (0.6900000000000001,0.13930253064420448)
            (0.7000000000000001,0.14338982583729987)
            (0.7100000000000001,0.1475229340009468)
            (0.72,0.1517007968697286)
            (0.73,0.15592233739414926)
            (0.74,0.16018646049562466)
            (0.75,0.16449205382753862)
            (0.76,0.1688379885422067)
            (0.77,0.17322312005642487)
            (0.78,0.17764628883566608)
            (0.79,0.18210632116148395)
            (0.8,0.1866020299178761)
            (0.81,0.19113221537204814)
            (0.8200000000000001,0.19569566595947796)
            (0.8300000000000001,0.20029115906961847)
            (0.8400000000000001,0.20491746183241039)
            (0.85,0.2095733319054185)
            (0.86,0.21425751825939837)
            (0.87,0.21896876196698895)
            (0.88,0.22370579698454812)
            (0.89,0.22846735093747048)
            (0.9,0.23325214590371873)
            (0.91,0.238058899189315)
            (0.92,0.24288632410781258)
            (0.93,0.2477331307542892)
            (0.9400000000000001,0.25259802677260745)
            (0.9500000000000001,0.2574797181270653)
            (0.9600000000000001,0.26237690986263995)
            (0.97,0.26728830686209326)
            (0.98,0.27221261460743684)
            (0.99,0.2771485399140803)
            (1.,0.2820947916914828)
            (1.01,0.2870500816760682)
            (1.02,0.29201312534577956)
            (1.03,0.29698264199430086)
            (1.04,0.3019573563492025)
            (1.05,0.30693599874402266)
            (1.06,0.311917305949434)
            (1.07,0.31690002188110517)
            (1.08,0.32188289830083067)
            (1.09,0.3268646955095352)
            (1.1,0.33184418303229024)
            (1.11,0.33682014030305346)
            (1.12,0.34179135733777777)
            (1.1300000000000001,0.3467566354059613)
            (1.1400000000000001,0.35171478769456205)
            (1.1500000000000001,0.3566646399671569)
            (1.1600000000000001,0.36160503121971793)
            (1.17,0.36653481432379154)
            (1.18,0.37145285669034145)
            (1.19,0.3763580408846543)
            (1.2,0.38124926527388403)
            (1.21,0.38612544467170856)
            (1.22,0.39098551095792866)
            (1.23,0.3958284137142135)
            (1.24,0.4006531208386941)
            (1.25,0.40545861920878523)
            (1.26,0.4102439152644355)
            (1.27,0.4150080356656113)
            (1.28,0.4197500279156479)
            (1.29,0.4244689609973127)
            (1.3,0.42916392601204123)
            (1.31,0.4338340368252802)
            (1.32,0.43847843072008796)
            (1.33,0.44309626905972277)
            (1.34,0.44768673796303937)
            (1.35,0.45224904897078877)
            (1.36,0.45678243986227424)
            (1.37,0.46128617512013187)
            (1.3800000000000001,0.46575954697572564)
            (1.3900000000000001,0.47020187595281476)
            (1.4000000000000001,0.47461251170766283)
            (1.4100000000000001,0.4789908337960792)
            (1.42,0.48333625246569323)
            (1.43,0.4876482094746418)
            (1.44,0.49192617896534885)
            (1.45,0.49616966844030674)
            (1.46,0.5003782199182157)
            (1.47,0.5045514113703444)
            (1.48,0.5086888584971611)
            (1.49,0.5127902169790997)
            (1.5,0.5168552674302042)
            (1.51,0.5208835978525346)
            (1.52,0.5248750487849312)
            (1.53,0.5288294418516986)
            (1.54,0.532746644752503)
            (1.55,0.5366265732946522)
            (1.56,0.5404691935357299)
            (1.57,0.5442745240150554)
            (1.58,0.5480426390949856)
            (1.59,0.5517736703615145)
            (1.6,0.5554678113301184)
            (1.61,0.5591253203722448)
            (1.62,0.562746525011969)
            (1.6300000000000001,0.5663318264527321)
            (1.6400000000000001,0.5698817041324908)
            (1.6500000000000001,0.5733967198235204)
            (1.6600000000000001,0.5768775212119859)
            (1.6700000000000002,0.580324846335832)
            (1.68,0.5837395325580479)
            (1.69,0.5871225351442532)
            (1.7,0.5904751966758666)
            (1.71,0.5937984278989347)
            (1.72,0.597093414538878)
            (1.73,0.6003623547082003)
            (1.74,0.6036064702062912)
            (1.75,0.6068278932928822)
            (1.76,0.6100286568284596)
            (1.77,0.6132112197812423)
            (1.78,0.6163787897722169)
            (1.79,0.6195330588189909)
            (1.8,0.6226787813614127)
            (1.81,0.6258194182732266)
            (1.82,0.6289593877064407)
            (1.83,0.6321037349404727)
            (1.84,0.6352582395406207)
            (1.85,0.6384295684530447)
            (1.86,0.6416256160363187)
            (1.87,0.6448554443864041)
            (1.8800000000000001,0.6481296013368633)
            (1.8900000000000001,0.651461135360213)
            (1.9000000000000001,0.65486597807825)
            (1.9100000000000001,0.6583635440023177)
            (1.9200000000000002,0.6619782104862897)
            (1.93,0.6657416488173091)
            (1.94,0.6696957842872769)
            (1.95,0.673899490297447)
            (1.96,0.6784371878871327)
            (1.97,0.6834480370055966)
            (1.98,0.6891582244988922)
            (1.99,0.6961007879826626)
            (2.00,0.706991)
        };
        \addplot[thick, dotted] coordinates {(1,+0.0) (1,+0.3)};
        \addplot[thick, dotted] coordinates {(2,+0.0) (2,+0.71)};
        \addplot[thick, dotted] coordinates {(0.667,+0.0) (0.667,+0.14)};
        \addplot[thick, dotted] coordinates {(0,+0.282) (1.10,+0.282)} node [pos=0.5,above] (SHE) {\small SHE: $\sfrac{1}{\sqrt{4\pi}}\approx 0.282$};
        \addplot[thick, dotted] coordinates {(0,+0.707) (2.05,+0.707)} node [pos=0.5,above] (SWE) {\small SWE: $\sfrac{1}{\sqrt{2}}\approx 0.707$};
        \addplot[thick, dotted] coordinates {(0,+0.130) (0.70,+0.130)};
      \end{axis}
    \end{tikzpicture}
    \begin{tikzpicture}[scale=0.9, transform shape, x=1em, y=1em]
      \tikzset{>=latex}
      \begin{axis}[
        ytick={        0, 0.25,           0.841,                            1.2, 1.6},
        yticklabels={$0$, $\sfrac{1}{4}$, $2^{-\sfrac{1}{4}}$, 1.2, 1.6},
        xtick={      0.667,          0.8, 1, 1.2, 1.4, 1.6, 1.8, 2},
        xticklabels={$\sfrac{2}{3}$, 0.8, 1, 1.2, 1.4, 1.6, 1.8, 2},
        title={Second moment Lyapunov exponent},
        xlabel={$\beta$}]
        \addplot[domain=0:2, id=Theta, blue, thick] coordinates {
            (0.766,1.6156229403847688)
            (0.776,0.9744621486813654)
            (0.786,0.6699959938241241)
            (0.796,0.5058772871434194)
            (0.806,0.40902708606595717)
            (0.816,0.348027170316181)
            (0.826,0.30778373308768814)
            (0.836,0.2803803463337963)
            (0.846,0.26136009618816647)
            (0.856,0.24806318826220006)
            (0.866,0.23882310249070376)
            (0.876,0.23255167117234268)
            (0.886,0.22851275215759745)
            (0.896,0.22619280422833052)
            (0.906,0.2252237921622652)
            (0.9159999999999999,0.22533559924445354)
            (0.9259999999999999,0.22632571896246442)
            (0.9359999999999999,0.2280394103012715)
            (0.946,0.23035637825569444)
            (0.956,0.23318163420297652)
            (0.966,0.23643909764414445)
            (0.976,0.24006703518456265)
            (0.986,0.24401475504632092)
            (0.996,0.2482401749904485)
            (1.006,0.25270811796358894)
            (1.016,0.25738839031223787)
            (1.026,0.2622558337795727)
            (1.036,0.26728841372631956)
            (1.046,0.2724671364898107)
            (1.056,0.2777754383333581)
            (1.066,0.2831987878900507)
            (1.076,0.2887243662910936)
            (1.086,0.294340808262165)
            (1.096,0.30003799084231353)
            (1.1059999999999999,0.30580685970991467)
            (1.116,0.3116392854163787)
            (1.126,0.31752794358116015)
            (1.1360000000000001,0.3234662144186044)
            (1.146,0.3294480979714343)
            (1.156,0.33546814217608273)
            (1.166,0.3415213814605174)
            (1.176,0.347603284127905)
            (1.186,0.35370970690534664)
            (1.196,0.35983685565406764)
            (1.206,0.3659812510993498)
            (1.216,0.37213969897842314)
            (1.226,0.3783092637241888)
            (1.236,0.3844872453456427)
            (1.246,0.39067115898377586)
            (1.256,0.3968587167168396)
            (1.266,0.403047811409113)
            (1.276,0.4092365023709076)
            (1.286,0.4154230018581997)
            (1.296,0.4216061056176497)
            (1.306,0.42778297627687756)
            (1.3159999999999998,0.43395354382905293)
            (1.326,0.44011608872645314)
            (1.3359999999999999,0.44626943714255113)
            (1.346,0.4524125140247259)
            (1.3559999999999999,0.4585443364661177)
            (1.366,0.46466400798018626)
            (1.376,0.47077071303667806)
            (1.3860000000000001,0.4768637124245801)
            (1.396,0.4829423387304994)
            (1.4060000000000001,0.489005992338311)
            (1.416,0.49505413773342294)
            (1.4260000000000002,0.5010863001274728)
            (1.436,0.5071020623911683)
            (1.446,0.5131010623025533)
            (1.456,0.5190829901441101)
            (1.466,0.525047586704184)
            (1.476,0.5309946417031706)
            (1.486,0.536923992759618)
            (1.496,0.5428355884600605)
            (1.506,0.5487292450226164)
            (1.516,0.5546049778997572)
            (1.526,0.5604627978561775)
            (1.536,0.5663027565407439)
            (1.546,0.5721249469104657)
            (1.556,0.5779295012685965)
            (1.566,0.583716592802197)
            (1.576,0.5894864348891291)
            (1.5859999999999999,0.5952392816133533)
            (1.596,0.6009754283876554)
            (1.6059999999999999,0.6066952128974474)
            (1.616,0.6123990164049925)
            (1.626,0.6180872653732944)
            (1.6360000000000001,0.6237604332063318)
            (1.646,0.6294190417511296)
            (1.6560000000000001,0.6350636623270312)
            (1.666,0.6406949167964907)
            (1.6760000000000002,0.6463134806723047)
            (1.686,0.6519200916593022)
            (1.696,0.6575155663740458)
            (1.706,0.6631010469754263)
            (1.716,0.6686768983226172)
            (1.726,0.6742449477143329)
            (1.736,0.6798057314789991)
            (1.746,0.6853608022215995)
            (1.756,0.6909114674561346)
            (1.766,0.6964592593692145)
            (1.776,0.7020058643439582)
            (1.786,0.7075531933597226)
            (1.796,0.7131032439160969)
            (1.806,0.7186583482437636)
            (1.816,0.7242213340652314)
            (1.826,0.7297943864369231)
            (1.836,0.7353815447102073)
            (1.846,0.7409863263164824)
            (1.856,0.7466131264020158)
            (1.866,0.7522670728294867)
            (1.8760000000000001,0.7579537518077222)
            (1.8860000000000001,0.7636799831432994)
            (1.896,0.769455944650156)
            (1.906,0.775290249066542)
            (1.916,0.7811963892025144)
            (1.926,0.7871908319447614)
            (1.936,0.7932955930952799)
            (1.946,0.7995396007438964)
            (1.956,0.8059658770841603)
            (1.966,0.812637664737253)
            (1.976,0.8196593488021726)
            (1.986,0.827270917650239)
            (1.996,0.8361307009192475)
        };
        \addplot[thick, dotted] coordinates {(0.667,+0.0) (0.667,+1.62)};
        \addplot[thick, dotted] coordinates {(1,+0.0) (1,+0.285)};
        \addplot[thick, dotted] coordinates {(2,+0.0) (2,+0.88)};
        \addplot[thick, dotted] coordinates {(0.64,+0.25) (1.1,+0.25)} node [pos=0.5,below]    (SHE) {\small SHE: $\sfrac{1}{4}$};
        \addplot[thick, dotted] coordinates {(0.64,+0.841) (2.05,+0.841)} node [pos=0.5,above] (SWE) {\small SWE: $2^{-\sfrac{1}{4}}\approx 0.841$};
      \end{axis}
    \end{tikzpicture}
  \end{center}

  \caption{Plots of $\Theta_{\beta,\nu}$ and the second moment Lyapunov exponent in Example
  \ref{Ex:SHESWE} with $\lambda=1$ and $\nu=1$.}

  \label{F:SHESWE}
\end{figure}

%
%

\section{Existence and uniqueness of the solution for the nonlinear equation} \label{S:Exist}

In this section, we shall establish the well-posedness of \eqref{E:fde} by working under slightly
more general settings as follows. For $\alpha>0$, $\beta\in(0,2]$, $\gamma\ge 0$, $\nu>0$, and
$\dot{W}$ as in \eqref{E:fde}, consider
\begin{equation}\label{E:fde-Gen}
  \begin{cases}
    \left(\partial_t^{\beta}+\dfrac{\nu}{2}\left(-\Delta\right)^{\alpha / 2}\right) u(t, x)= \: I_{t}^{\gamma}\left[\rho\left(u(t, x)\right) \dot{W}(t, x)\right] , & t>0, x \in \mathbb{R}^{d},               \\
    u(0,x)=u_0(x),                                                                                                                                                & \text {$x\in\R^d$, if $\beta \in(0,1]$}, \\
    u(0,x)=u_0(x), \quad \dfrac{\partial}{\partial t} u(0,x)=u_1(x),                                                                                              & \text {$x\in\R^d$, if $\beta \in(1,2]$,}
  \end{cases}
\end{equation}
where $\rho(\cdot)$ is Lipschitz continuous and $u_0, u_1\in L^\infty\left(\R^d\right)$. The
fundamental solutions for \eqref{E:fde-Gen}, as well as \eqref{E:fde}, consist of three components:
$Z_{\alpha,\beta,d}(t,x)$, $Z_{\alpha,\beta,d}^*(t,x)$ and $Y_{\alpha,\beta, \gamma, d}(t,x)$, which
have been studied in \cite[Theorem 4.1]{chen.hu.ea:19:nonlinear} for the case when $\beta\in(0,2)$
and $\alpha\in (0,2]$. The more general setting, namely, the case when $\alpha>0$ and
$\beta\in(0,2]$, is proved in Theorem \ref{T:PDE}. Throughout the rest of the article, we will write
\begin{align} \label{E:p}
  p(t,x):=Y_{\alpha,\beta,\gamma,d}(t,x),
\end{align}
whose Fourier transform is given in \eqref{E:FY}.

%
The solution to the homogeneous equation of \eqref{E:fde-Gen} is given by
\begin{align} \label{E:J0(t,x)}
  J_0(t,x) =
  \begin{cases}
   \displaystyle \int_{\R^d} Z_{\alpha,\beta,d}(t,x-y)u_0(y)\ud y,                                                    & \text{if $\beta\in(0,1]$}, \vspace{0.2cm} \\
   \displaystyle \int_{\R^d} Z_{\alpha,\beta,d}^*(t,x-y)u_0(y)\ud y + \int_{\R^d} Z_{\alpha,\beta,d}(t,x)u_1(y)\ud y, & \text{if $\beta\in(1,2]$}.
  \end{cases}
\end{align}
When the initial conditions $ u_0$ and $u_1$ are two constants, then by \eqref{E:FZ} and
\eqref{E:FZ*}, $J_0(t,x)$ does not depend on $x$ and hence is denoted by $J_0(t)$ later on:
\begin{equation} \label{E:J0(t)}
  J_0(t) =
  \begin{cases}
    u_0\mathcal{F} Z_{\alpha,\beta,d}(t,\cdot) (0)  = u_0,                                                       & \text{if $\beta\in(0,1]$} , \vspace{0.2cm} \\
    u_0\mathcal{F} Z_{\alpha,\beta,d}^*(t,\cdot)(0) + u_1\mathcal{F} Z_{\alpha,\beta,d}(t,\cdot)(0) = u_0+u_1 t, & \text{if $\beta\in(1,2]$} ,
  \end{cases}
\end{equation}
where $\mathcal{F} g=\widehat{g}$ is the Fourier transform of $g$ in spatial variable, i.e., if
$g(t,\cdot) \in L^1(\R^d)$,
\begin{align*}
  \mathcal{F} g(t,\xi)=\widehat{g}(t, \xi) :=\int_{\R^d} g(t,x)  e^{-i x\cdot \xi} \ud x.
\end{align*}

Let $W=\left\{W_t(A):A\in\cB_b(\R^d),t\geq0\right\}$ be a space–time white noise defined on a
complete probability space $\left(\Omega,\cF,\mathbb{P}\right)$, where $\cB_b(\R^d)$ is the
collection of Borel sets with finite Lebesgue measure. Let
\begin{equation*}
  \cF_t=\sigma(W_s(A):0\leq s\leq t, A\in\cB_b(\R^d))\vee\cN,\quad t\geq0,
\end{equation*}
be the natural filtration augmented by the $\sigma$-field $\cN$ generated by all $\mathbb{P}$-null
sets in $\cF$.

\begin{definition} \label{D:Sol}
  A process $u=\left\{u(t,x): t>0,x\in\R^d\right\}$ is called a {\it random field solution} to
  \eqref{E:fde-Gen} if it is adapted to the filtration $\{\mathcal{F}_t\}_{t\ge 0}$, jointly
  measurable with respect to $\mathcal{B}\left((0,\infty)\times\R^d\right)\times \mathcal{F}$,
  square integrable in the sense that
  \begin{align*}
    \int_0^t\ud s\int_{\R^d}\ud y\: p(t-s,x-y)^2 \E\left[\rho\left(u(s,y)\right)^2\right] <\infty,
    \quad \text{for all $t>0$ and $x\in\R^d$,}
  \end{align*}
  and satisfies the following integral equation a.s.
 \begin{equation}\label{E:fde-Int}
    u(t,x) = J_0(t,x)+ \int_{0}^{t}\int_{\bR^d}p(t-s,x-y) \rho\left(u(s,y)\right)W(\ud s,\ud y),
  \end{equation}
  for all $t>0$ and  $x\in\R^d$, where $J_0(t,x)$ is given by \eqref{E:J0(t,x)} and the stochastic
  integral on the right-hand side is the {\it Walsh integral} \cite{walsh:86:introduction}.
\end{definition}

The existence and uniqueness of the mild solution to \eqref{E:fde-Gen} with the bounded initial
conditions are well covered by the classical Dalang-Walsh theory; see
\cite{walsh:86:introduction,dalang:99:extending,dalang.khoshnevisan.ea:09:minicourse}. In that
theory, {\it Dalang's condition} usually refers to some simplified, but still equivalent, conditions
to
\begin{equation}\label{E:Dalang}
  \int_{0}^{t}\int_{\bR^d} p(s,y)^2\ud s\ud y<\infty, \quad \text{for all $t>0$}.
\end{equation}
Note that condition \eqref{E:Dalang} is the necessary and sufficient condition for the existence and
uniqueness of a global solution for the corresponding linear equation, i.e., the case when
$\rho(u)\equiv 1$ in \eqref{E:fde-Gen}. The following lemma finds out the explicit form of Dalang's
condition for \eqref{E:fde} and \eqref{E:fde-Gen}, which extends Lemma 5.3 of
\cite{chen.hu.ea:19:nonlinear} from the case $\alpha\in (0,2]$ and $\beta\in(0,2)$ to the case
$\alpha>0$ and $\beta\in(0,2]$.

\begin{lemma}[Dalang's condition] \label{L:Dalang}
  For the SPDE \eqref{E:fde-Gen}, Dalang's condition \eqref{E:Dalang} is equivalent to
  \eqref{E:Dalang'}.
\end{lemma}
\begin{proof}
  By \eqref{E:FY} and the Parseval–Plancherel identity, we have
  \begin{align}\label{E:L2-p}
    \int_{\R^d} |p(s,x)|^2 \ud x
    = & \frac{1}{(2\pi)^d} \int_{\R^d}|\widehat{p}(s,\xi)|^2d\xi
    = \frac{1}{(2\pi)^d}\int_{\R^d}s^{2(\beta+\gamma)-2}E_{\beta,\beta+\gamma}^2(-2^{-1}\nu|\xi|^\alpha s^\beta)\ud\xi \notag \\
    = & s^{2(\beta+\gamma)-2- \beta d/\alpha} \frac{1}{(2\pi)^d}  (2^{-1}\nu )^{-d/\alpha} \int_{\R^d} E_{\beta,\beta+\gamma}^2(-|\eta|^\alpha)\ud\eta,
  \end{align}
  where in the last step we have used the change of variable $\xi=(2^{-1}\nu s^\beta)^{-1/\alpha}
  \eta$. Then clearly, Dalang's condition \eqref{E:Dalang} is equivalent to
  \begin{equation} \label{E:power}
  \begin{cases}
    & 2(\beta+\gamma)-2-\beta d/\alpha>-1, \vspace{0.2cm} \\
    & \displaystyle \int_{\R^d} E_{\beta,\beta+\gamma}^2(-|\xi|^\alpha)\ud\xi<\infty.
  \end{cases}
  \end{equation}
  To characterize the second condition in \eqref{E:power}, noting that
  $E_{\beta,\beta+\nu}^2\left(-|\cdot|\right)$ is locally integrable, it suffices to know the
  asymptotic behavior of $E_{\beta,\beta+\gamma}^2(-|\xi|^\alpha)$ as $|\xi|\to \infty$. By Lemma
  \ref{L:ML-Asymp}, as $|\xi|\to \infty$,
  \begin{equation}\label{E:asym-E}
    E_{\beta,\beta+\gamma}\left(-|\xi|^\alpha\right)=
  \begin{cases}
   \displaystyle -\frac{1}{\Gamma(\gamma) |\xi|^\alpha}+O(|\xi|^{-2\alpha}),                                                                                                      & \beta\in(0,2),\vspace{0.3cm} \\
   \displaystyle \frac{\cos\left(\sqrt{|\xi|^\alpha}-\pi(\gamma+1)/2\right)}{|\xi|^{\alpha(1+\gamma)/2}} +  \frac{1}{\Gamma(\gamma)|\xi|^\alpha}+O\left(|\xi|^{-2\alpha}\right) , & \beta=2,
   \end{cases}
  \end{equation}
  Then, for $\beta\in(0,2)$, clearly \eqref{E:power} is equivalent to
  \begin{align*}
    2(\beta+\gamma)-2-\frac{\beta d}{\alpha}>-1 ~\text{ and }~ 2\alpha>d,
  \end{align*}
  which can be also expressed as $d<2\alpha+\frac{\alpha}{\beta}\min\{2\gamma-1,0\}$. For the case
  $\beta=2$, by \eqref{E:asym-E}, we have as $|\xi|\to \infty$,
  \begin{align*}
    E^2_{\beta,\beta+\gamma}\left(-|\xi|^\alpha\right)
    = &   \frac{\cos^2\left(\sqrt{|\xi|^\alpha}-\pi(\gamma+1)/2\right)}{|\xi|^{\alpha(1+\gamma)}}
        + \frac{1}{\Gamma^2(\gamma)|\xi|^{2\alpha}} \\
      & + 2\frac{\cos\left(\sqrt{|\xi|^\alpha}-\pi(\gamma+1)/2\right)}{\Gamma(\gamma)|\xi|^{\alpha(3+\gamma)/2}}
        + O\left(|\xi|^{-5\alpha/2}\right).
  \end{align*}
  Thus, the second condition in \eqref{E:power} is equivalent to, for any $\varepsilon>0$,
  \begin{equation*}
    \int_{|\xi|>\varepsilon} \frac{\cos^2\left(\sqrt{|\xi|^\alpha}-\pi(\gamma+1)/2\right)}{|\xi|^{\alpha(1+\gamma)}} \ud \xi <\infty
    \quad \text{and} \quad
    \int_{|\xi|>\varepsilon} \frac{1}{|\xi|^{2\alpha}} \ud \xi <\infty,
  \end{equation*}
  where the first condition is equivalent to $\alpha(1+\gamma)>d$ by Lemma \ref{L:integrability} and
  the second one is $2\alpha>d$. Therefore, when $\beta=2$, we have that \eqref{E:power} is
  equivalent to
   \begin{equation*}
   \begin{cases}
      & d<\alpha(\gamma+\frac32),\vspace{0.2cm} \\
      & d<\alpha\min\{2, 1+\gamma \},
    \end{cases}
    \quad \Longleftrightarrow \quad
      d<\alpha\min\left\{2, 1+\gamma \right\}.
  \end{equation*}
  This completes the proof of Lemma \ref{L:Dalang}.
\end{proof}

Under Dalang's condition, it is routine (see, e.g., Theorem 13 of \cite{dalang:99:extending} or the
proof of Theorem 2.4 of \cite{chen.dalang:15:moments*1}) to establish the following theorem regarding
the existence and uniqueness of the solution to \eqref{E:fde-Gen}, the proof of which will be left
for the interested readers.

\begin{theorem} \label{T:Exist}
  Under Dalang's condition \eqref{E:Dalang'}, if the initial conditions are bounded, namely, $u_0$
  and $u_1\in L^\infty(\R^d)$, then there exists a unique (in the sense of versions) random field
  solution $u(t,x)$ to \eqref{E:fde-Gen}, which is $L^2(\Omega)$-continuous with all bounded $p$-th
  moments:
  \begin{equation}\label{E:bd-con}
    \sup_{0\le t \le T}\sup_{x\in\R^d} \E\left[\left|u(t,x)\right|^p\right]<\infty, \quad
    \text{for all $p\ge 2$ and $T>0$.}
  \end{equation}
\end{theorem}

Before the end of this section, we make some remarks:

\begin{remark}[Rough initial data]
  The main focus of this paper is the exact moment formula with constant initial condition. Theorem
  \ref{T:Exist} presents the existence and uniqueness of the solution in a slightly more general
  setting, which still falls in the classical Dalang-Walsh theory. For the measure-valued initial
  conditions, such as the Dirac delta initial condition, more efforts are needed and property
  \eqref{E:bd-con} no longer holds; see
  \cite{chen.dalang:15:moments*1,chen.dalang:15:moments,chen.dalang:15:moment,chen.kim:19:nonlinear,chen.hu.ea:19:nonlinear}.
\end{remark}
\begin{remark}[H\"oder regularity]
  In \cite{chen.hu.ea:19:nonlinear} and \cite{chen.hu:22:holder}, the space-time H\"older regularity
  of the solution to \eqref{E:fde} has been obtained (in the case of $\alpha\in(0,2]$ and
  $\beta\in(0,2)$). It is an interesting open problem to extend the H\"older regularity results in
  \cite{chen.hu.ea:19:nonlinear} and \cite{chen.hu:21:holder} to the more general setting, namely,
  $\alpha>0$ and $\beta\in(0,2]$.
\end{remark}
\begin{remark}[Second moment comparison for nonlinear SPDEs] \label{R:MomComp}
  Let $u(t,x)$ be the solution to \eqref{E:fde-Gen} as stated in Theorem \ref{T:Exist}. Suppose that
  $\sigma$ is Lipschitz continuous and satisfies the following cone condition with some constants
  $0\le\underline{\lambda}\le \overline{\lambda}$:
  \begin{align*}
    \underline{\lambda}|x| \le \left|\sigma(x)\right| \le \overline{\lambda} |x|, \quad \text{for all $x\in\R$}.
  \end{align*}
  Then by denoting the right-hand side of \eqref{E:SecMom} by $f_\lambda(t)$, the moment formula in
  Theorem \ref{T:fde} can be extended directly to this case by the following
 moment comparison principle for the second moment:
  \begin{align} \label{E:MomComp}
    f_{\underline{\lambda}}(t)\le \E\left[u(t,x)^2\right] \le f_{\overline{\lambda}}(t).
  \end{align}

 When the noise is white in time but colored in space (see \eqref{E:NoiseWC}),
    the moment comparison principle (for $p\ge 2$) or more generally the stochastic comparison
    principle becomes much more involved and the parabolic nature of the equation will play an
    important role. Hence, one can in principle only handle the case when $\beta=1$. One may check
    the work along this line in    \cite{chen.kim:19:nonlinear,chen.huang:19:comparison,chen.kim:20:stochastic}. However, for the
    space-time white noise case, the second moment comparison as in \eqref{E:MomComp} comes for
    free. Note that when the noise is colored in time (see \eqref{E:NoiseCC}), to the best our
  knowledge, one can only handle the linear case, namely, $\sigma(u)=\lambda u$. In this case, the
moment comparison principle can be easily established by comparing the movements chaos by chaos.
\end{remark}
\begin{remark}[Wiener chaos expansion] \label{R:W-C}
  When $\rho(u)=\lambda u$, instead of using Dalang-Walsh theory, one can equivalently establish the
  solution to \eqref{E:fde} using the Wiener chaos expansion specified as follows: Set $ u_0(t,x) =
  J_0(t)$ and for $n\ge 1$,
  \begin{align*}
    u_{n}(t, x) = J_0(t) +\sum_{k=1}^{n} \lambda^k  \int_{[0,t]^k}\int_{\R^{kd}} g_k(s_1, \dots, s_k, x_1, \dots, x_k; t,x) W(\ud s_1, \ud x_1)\cdots W(\ud s_k,\ud x_k),
  \end{align*}
  where
  \begin{align} \label{E:gk}
      & g_k(s_1, \dots, s_k, x_1, \dots, x_k; t,x)\notag                                                                        \\
    = & p(t-s_k, x-x_k)p(s_k-s_{k-1}, x_k-x_{k-1})\cdots p(s_2-s_1, x_2-x_1) J_0(s_1) \one_{\{ 0<s_1<\dots<s_k<t \}}\notag \\
    = & \prod_{r=1}^k p(s_{r+1}-s_r, x_{r+1}-x_r ) J_0(s_1)  \one_{\{ 0<s_1<\dots<s_k<t \}},
  \end{align}
  where we use the convention $s_{k+1}=t$ and $x_{k+1}=x$.  Then, the mild solution has the
  following so-called Wiener chaos representation:
  \begin{equation}\label{E:Chaos}
    u(t,x) = J_0(t) +\sum_{k=1}^\infty \lambda^k I_k (f_k(\cdot; t,x)),
  \end{equation}
  where $f_k(\cdot; t,x)$ is the symmetrization of $g_k(\cdot; t,x)$ given by, denoting by $\mathcal
  P_k$ the set of all permutations of $\{1, \dots, k\}$,
  \begin{equation}\label{E:fk}
    f_k(s_1, \dots, s_k, x_1, \dots, x_k; t,x)=\frac1{k!}\sum_{\sigma \in \mathcal{P}_k} g_k(s_{\sigma(1)}, \dots, s_{\sigma(k)}, x_{\sigma(1)}, \dots, x_{\sigma(k)}; t,x)
  \end{equation}
  and $I_k (f_k(\cdot; t,x))$ denotes the $k$-th multiple Wiener-It\^o integral. We refer the
  interested readers to \cite{hu:17:analysis} for more details.
\end{remark}

\section{Second moment formula and upper bounds for the $p$-th moments} \label{S:Upper}

In this section, we shall prove parts (a) and (b) of  Theorem \ref{T:fde}.

\begin{proof}[Proof of part (a) of Theorem \ref{T:fde}]
  By the It\^o-Walsh isometry, we have
  \begin{equation*}
    \E\left[u^2(t,x)\right]= J_0^2(t)+\lambda^2 \int_0^{t}\int_{\R^d} p^2(t-s,x-y) \E\left[u^2(s,y)\right]\ud s\ud y,
  \end{equation*}
  where $J_0(t)$ is given by \eqref{E:J0(t)}. Note that due to the choice of the constant initial
  conditions, the solution to the homogeneous equation does not depend on $x$, i.e.,
  $J_0(t,x)=J_0(t)$. Hence, through a standard Picard iteration, one can show that the second moment
  $\E\left(u(t,x)^2\right)$ does not depend on $x$. Let $\eta(t)=\E\left(u(t,x)^2\right)$. Invoking
  \eqref{E:L2-p}, the above equation can be written as
  \begin{equation} \label{E:eta}
     \eta(t) = J_0^2(t)+\lambda^2 \Theta \int_{0}^{t}(t-s)^\theta\eta(s)\ud s,
  \end{equation}
  where $\theta$ and $\Theta$ are given in \eqref{E:theta}. Now we solve the fractional integral
  equation \eqref{E:eta} for $\beta \in (0,1]$ and for $\beta\in(1, 2]$ separately. \medskip

  \noindent {\bf Case 1.} When $\beta\in(0, 1]$, we have $J_0(t) = u_0$ by \eqref{E:J0} and thus
  \eqref{E:eta} is equivalent to
  \begin{equation*}
    \begin{cases}
      \left(D_{0+}^{\theta+1}\eta\right)(t) =\lambda^2 \Theta \Gamma(\theta+1)\eta(t) + (D_{0+}^{\theta+1} u_0^2)(t),\vspace{0.2cm} \\
      \eta(0)=u_0^2 \quad \text{and} \quad \eta^{(k)}(0)=0, \text{ for } k=1,2,..., \Ceil{\theta},
    \end{cases}
  \end{equation*}
  where $D^{\theta+1}$ is Riemann-Liouville derivative given in Definition \ref{def:D}. Using the
  Caputo fractional derivative given in Definition \ref{def:CD}, it can written as
  \begin{equation*}
    \begin{cases}
      \left({}^CD_{0+}^{\theta+1}\eta\right)(t)=\lambda^2 \Theta \Gamma(\theta+1)\eta(t). \vspace{0.2cm} \\
      \eta(0)=u_0^2 \quad \text{and} \quad \eta^{(k)}(0)=0, \text{ for } k=1,2,..., \Ceil{\theta},
    \end{cases}
  \end{equation*}
  of which the solution is directly given by \eqref{y1}:
  \begin{equation*}
    \eta(t) =u_0^2 E_{\theta+1}\left(\lambda^2 \widehat{t}\:\right).
  \end{equation*}
  This proves the first part of \eqref{E:SecMom}. \medskip

  \noindent {\bf Case 2.} When $\beta\in(1,2]$, we have $J_0(t) = u_0 + u_1 t$ by \eqref{E:J0} and
  \eqref{E:eta} now is
  \begin{equation} \label{E:eta'}
     \eta(t) = u_0^2+2u_0u_1t+u_1^2t^2+\lambda^2 \Theta \int_{0}^{t}(t-s)^\theta\eta(s)\ud s.
  \end{equation}
  Let $f(t):=2u_0u_1t+u^2_1t^2$ then \eqref{E:eta} can be written as
  \begin{equation}\label{E:jieFDE}
    \begin{cases}
      (D_{0+}^{\theta+1}\eta)(t) = \lambda^2 \Theta \Gamma(\theta+1)\eta(t) + (D_{0+}^{\theta+1}[u_0^2+f(\cdot)])(t),\vspace{0.2cm}\\
      \eta(0)       = u_0^2, ~
      \eta^{(1)}(0) = 2u_0u_1, ~
      \eta^{(2)}(0) = 2u_1^2,\vspace{0.2cm}\\
      \eta^{(k)}(0) = 0, \text{ for $k = 3, \dots, \Ceil{\theta}$.}
    \end{cases}
  \end{equation}
  In order to apply the formula \eqref{y1}, we will transform \eqref{E:jieFDE} into a Caputo
  fractional differential equation. When $\theta+1\in(0,1)$, by \eqref{E:CD}, we can write
  \eqref{E:jieFDE} as
  \begin{equation*}
    \begin{cases}
      {}^CD_{0+}^{\theta+1}\eta(t)=\lambda^2 \Theta \Gamma(\theta+1)\eta(t) +(D_{0+}^{\theta+1}f)(t). \vspace{0.2cm} \\
      \eta(0)=u_0^2.
    \end{cases}
  \end{equation*}
  The solution now follows directly from \eqref{y1}:
  \begin{align}\label{E:eta''}
     \eta(t) = & ~ u_0^2 E_{\theta+1}\left(\lambda^2 \widehat{t}\:\right) \notag \\
               & \qquad + \int_{0}^{t} (t-s)^\theta E_{\theta+1,\theta+1}\left(\lambda^2 \Theta \Gamma(\theta+1)(t-s)^{\theta+1}\right)(D_{0+}^{\theta+1}f)(s)\ud s.
   \end{align}
  For the integral on the right-hand side, by \eqref{E:ML} and Lemma \ref{L:I-D} we have
  \begin{align*}
      & \int_{0}^{t} (t-s)^\theta E_{\theta+1,\theta+1}\left(\lambda^2 \Theta \Gamma(\theta+1)(t-s)^{\theta+1}\right)(D_{0+}^{\theta+1}f)(s)\ud s                       \\
    = & \int_{0}^{t}(t-s)^\theta \sum_{k=0}^\infty \frac{(\lambda^2 \Theta \Gamma(\theta+1))^k}{\Gamma((k+1)(\theta+1))}(t-s)^{k(\theta+1)}(D_{0+}^{\theta+1}f)(s)\ud s \\
    = & \sum_{k=0}^{\infty}(\lambda^2 \Theta \Gamma(\theta+1))^k \left(I_{0+}^{(\theta+1)(k+1)}D_{0+}^{\theta+1}f\right)(t)                                                  \\
    = & \sum_{k=0}^{\infty}(\lambda^2 \Theta \Gamma(\theta+1))^k \left(I_{0+}^{k(\theta+1)}f\right)(t).
  \end{align*}
  The term $\left(I_{0+}^{k(\theta+1)}f\right)(t)=\left(I_{0+}^{k(\theta+1)} (2u_0u_1 s + u_1^2
  s^2)\right)(t)$ can be computed explicitly noting that Lemma \ref{L:ItDt} yields
 \begin{equation}\label{E:I-s}
   \left(I_{0+}^{k(\theta+1)}s\right)(t)   = \frac{t^{k(\theta+1)+1}}{\Gamma(k(\theta+1)+2)}  \quad \text{and} \quad
   \left(I_{0+}^{k(\theta+1)}s^2\right)(t) = \frac{2t^{k(\theta+1)+2}}{\Gamma(k(\theta+1)+3)}.
 \end{equation}
  Combining \eqref{E:eta''}--\eqref{E:I-s} and applying \eqref{E:ML}, we arrive at
  \begin{equation*}
    \eta(t)
    = u_0^2 E_{\theta+1}\left(\lambda^2 \widehat{t}\:\right)
    + 2u_0u_1t E_{\theta+1,2}\left(\lambda^2 \widehat{t}\:\right)
    + 2u_1^2t^2 E_{\theta+1,3}\left(\lambda^2 \widehat{t}\:\right).
  \end{equation*}
  This proves the second part of \eqref{E:SecMom} for $\theta+1\in(0,1)$. For the other two cases
  $\theta+1\in[1, 2)$ and $\theta+1\ge2$, one can calculate in a similar way and prove the desired
  result. Finally, \eqref{E:2nd-Ly} is a direct consequence of Lemma \ref{L:ML-Asymp}. This
  completes the proof of part (a) of Theorem \ref{T:fde}.
\end{proof}
\begin{remark}[Alternative approach]\label{R:method2}
  Alternatively, one can also solve \eqref{E:eta} directly by an application of Lemma \ref{L:f(t)}
  as follows:
  \begin{align*} 
    \eta(t) = J_0^2(t) + \int_0^t J_0^2(s)K(t-s)\ud s,
  \end{align*}
  where $J_0(t)$ is given in \eqref{E:J0} and the resolvent kernel function $K\left(\cdot\right)$ is
  given by
  \begin{align*}
    K(t) = \lambda^2 \Theta \Gamma(\theta+1)t^{\theta} E_{\theta+1,\theta+1} \left(\lambda^2 \widehat{t}\:\right).
  \end{align*}
  Thus we have, denoting $A=\lambda^2 \Theta \Gamma(\theta+1)$,
 \begin{align*}
   \eta(t) & = J_0^2(t)+ A\int_{0}^{t}J_0^2(s) (t-s)^\theta E_{\theta+1,\theta+1}\left(A(t-s)^{\theta+1}\right) \ud s .
 \end{align*}
  When $J_0(t) = u_0$, we have by the definition \eqref{E:ML} of $E_{a,b}$,
 \begin{align*}
    \eta(t)
    & = u_0^2+u_0^2 \sum_{k=0}^{\infty} \frac{A^{k+1}}{\Gamma((\theta+1)(k+1))} \int_{0}^{t}(t-s)^{(\theta+1)k+\theta} \ud s                                                                        \\
    & = u_0^2+u_0^2 \sum_{k=0}^{\infty} \frac{A^{k+1}t^{(\theta+1)(k+1)}}{\Gamma((\theta+1)(k+1)+1)}= u_0^2+u_0^2 \sum_{k=1}^{\infty} \frac{A^kt^{(\theta+1)k}}{\Gamma\left((\theta+1)k + 1\right)} \\
    & = u_0^2E_{\theta+1}\left(\lambda^2 \widehat{t}\:\right).
  \end{align*}
  This proves the equality of \eqref{E:SecMom} for $\beta\in(0,1]$.  Applying similar computations
  to the case $J_0(t)=u_0+u_1t$, we can justify the second part of \eqref{E:SecMom} for $\beta\in(1,
  2]$. Indeed, the $p$-th moment upper bounds will be obtained using this approach in the next proof.
\end{remark}
\begin{proof}[Proof of part (b) of Theorem \ref{T:fde}]
  Fix an arbitrary $p\ge 2$. By \eqref{E:fde-Int} we have
  \begin{align*}
    \|u(t,x)\|_p\le |J_0(t)|
    + \left(\E\left[\left|\int_{0}^{t}\int_{\R^d}p(t-s,x-y)\lambda u(s,y)W(\ud s,\ud y)\right|^p\right]\right)^{1/p}.
  \end{align*}
  Applying the Burkholder–Davis–Gundy inequality , we have
  \begin{equation*}
    \|u(t,x)\|_p
    \leq  |J_0(t)|+C_p \left(\E\left[ \left(\int_{0}^{t}\int_{\R^d}p^2(t-s,x-y)\lambda^2u^2(s,y)\ud s\ud y \right)^{p/2}\right]\right)^{1/p},
  \end{equation*}
  where $C_p$ is the universal constant in the Burkholder–Davis–Gundy inequality satisfying $C_p\in
  (0,2\sqrt{p})$ and $C_p=(2+o(1))\sqrt{p}$ as $p\to \infty$ (see
  \cite{carlen.kree:91:lp,conus.khoshnevisan:12:on}). By Minkowski's inequality, we get
  \begin{equation*}
    \Vert u(t,x)\Vert_p \leq |J_0(t)|+2\sqrt{p}\left(\int_{0}^{t}\int_{\R^d}\lambda^2p^2(t-s,x-y)\Vert u(s,y)\Vert_p^2\ud s\ud y\right)^{1/2}.
  \end{equation*}
  Denote $\psi(t)=\sup\limits_{x\in\R^d}\Vert u(t,x)\Vert_p^2$ and recall the definition
  \eqref{E:theta} of $\theta$ and $\Theta$. We have
  \begin{equation*}
    \psi(t) \le 2\left(J_0^2(t) + 4p \lambda^2\Theta\int_{0}^{t}(t-s)^{\theta}\psi(s)\ud s\right).
  \end{equation*}
  Applying Lemma \ref{L:f(t)}, we have
  \begin{equation*}
    \psi(t) \leq 2 J_0^2(t) + 2 \int_{0}^{t}J_0^2(s) K(t-s)\ud s,
  \end{equation*}
  where
  \begin{equation*}
    K(t)= 8p\lambda^2 \Theta \Gamma(\theta+1)t^{\theta}
    E_{\theta+1,\theta+1} \left(8p\lambda^2 \widehat{t}\:\right).
  \end{equation*}
  Then,  one can apply the same computations as those in Remark \ref{R:method2} to simplify the
  above $\ud s$ integral in order to obtain \eqref{E:p-mom}. Finally, \eqref{E:upper-tlim} and
  \eqref{E:upper-plim} follow from Lemma \ref{L:ML-Asymp} directly. This proves part (b) of Theorem
  \ref{T:fde}.
\end{proof}

\section{Lower bounds for the $p$-th moments} \label{S:Lower}
Compared with the upper bound for the $p$-th moment, the computation for the lower bound is more
involved.  The methodology used in this section is inspired by the recent work of
Hu and Wang~\cite{hu.wang:21:intermittency}. Some ideas are originated from Dalang and Mueller
\cite{dalang.mueller:09:intermittency}.


\subsection{Nondegeneracy and positivity of the fundamental functions} \label{SS:Nondeg}

In the next proposition, we prove a nondegeneracy property of the fundamental solutions, which is
tailored specially for the spatial white noise. Conditions for the fundamental solutions to be
nonnegative are given in Remark \ref{R:Nonneg} below.

\begin{proposition}\label{P:nondeg}
  For all $\e>0$ and $c>2$, if either
  \begin{enumerate}
    \item the fundamental solution $p(\cdot,\circ)$ is nonnegative and $\beta\in(0,2)$, or
    \item $\alpha=\beta=2$, $\gamma=0$, and $d=1,2$,
  \end{enumerate}
  then there exists some constant $C>0$ independent of $\e$ such that
  \begin{equation}\label{E:nondeg}
    \iint_{B_\e^2(x)}p(t,a-y)p(s,b-y')\delta_0(y-y')\ud y\ud y'\ge C \e^{-d}(ts)^{\beta+\gamma-1}
  \end{equation}
  for all $x\in\R^d$, $s,t\in[2\e^\frac{\alpha}{\beta},c\e^\frac{\alpha}{\beta}]$, and $a,b\in
  B_\e(x)$.
\end{proposition}
\begin{proof}
  Denote the double integral in \eqref{E:nondeg} by $I$. We first work under condition (1). In
  this case, from \eqref{E:Yab}, we see that
  \begin{align*}
    I = \int_{B_\e(x)} p(t,a-x')p(s,b-x')\ud x'
      = & \int_{B_\e(x)} \pi^{-d/2}|x'-a|^{-d}t^{\beta+\gamma-1}
          h\left(\frac{ |x'-a|^\alpha}{2^{\alpha-1}\nu t^\beta}\right)  \\
        & \times \pi^{-d/2}|x'-b|^{-d}s^{\beta+\gamma-1}
          h\left(\frac{ |x'-b|^\alpha}{2^{\alpha-1}\nu s^\beta}\right)\ud x',
  \end{align*}
  where
  \begin{equation*}
    h(x):=\FoxH{2,1}{2,3}{x}{(1,1),\:(\beta+\gamma,\beta)}{(d/2,\alpha/2),\:(1,1),\:(1,\alpha/2)}.
  \end{equation*}
  Set $I':=\pi^{-d}\left(ts\right)^{\beta+\gamma-1} I$. Notice when $\beta\in(0,2)$, the fundamental
  solution $p(t,x)$ is a smooth function for $t>0$ and $x\in\R^d$ and $x\in\R^d$ and the support of
  $p(t,x)$ is the whole space. Moreover, under the nonnegative assumption, for
  $s,t\in[2\e^\frac{\alpha}{\beta},c\e^\frac{\alpha}{\beta}]$, we have
  \begin{align*}
     I'\ge & \inf_{c_1,c_2\in[2,c]}\int_{B_\e(x)}  |x'-a|^{-d}h\left(\frac{ |x'-a|^\alpha}{2^{\alpha-1}\nu c_1^\beta\e^\alpha}\right)
             \times |x'-b|^{-d}h\left(\frac{|x'-b|^\alpha}{2^{\alpha-1}\nu c_2^\beta\e^\alpha}\right)\ud x' \\
        =  & \e^{-d} \inf_{c_1,c_2\in[2,c]}\int_{B_1(x)} |x'-a/\e|^{-d}h\left(\frac{ |x'-a/\e|^\alpha}{2^{\alpha-1}\nu c_1^\beta}\right)
             \times |x'-b/\e|^{-d}h\left(\frac{ |x'-b/\e|^\alpha}{2^{\alpha-1}\nu c_2^\beta}\right) \ud x' \\
       \ge & C\e^{-d},
  \end{align*}
 where
 \begin{equation*}
   C = \mathop{\inf_{a',b'\in B_1(x)}}_{c_1,c_2\in[2,c]}\int_{B_1(x)} |x'-a'|^{-d}h\left(\frac{ |x'-a'|^\alpha}{2^{\alpha-1}\nu c_1^\beta}\right)
       \times  |x'-b'|^{-d}h\left(\frac{ |x'-b'|^\alpha}{2^{\alpha-1}\nu c_2^\beta}\right) \ud x'
     > 0.
 \end{equation*}
 This proves \eqref{E:nondeg} under condition (1). \bigskip

 Now we assume condition (2). It suffices to show the case when $\nu=2$. It is well known that when
 $\alpha=\beta=\nu=2$ and $\gamma=0$,
  \begin{equation*}
    p(t,x)=
    \begin{cases}
      \dfrac{1}{2}\one_{\{|x|<t\}},                               & \text{if $d=1$}, \\[1em]
      \dfrac{1}{2\pi}\dfrac{1}{\sqrt{t^2-|x|^2}}\one_{\{|x|<t\}}, & \text{if $d=2$}.
    \end{cases}
  \end{equation*}
  For all $a,b,x'\in B_\e(x)$ and $s,t\ge 2\e$, we have
  $\one_{\{|x'-a|<t\}}\one_{\{|x'-b|<s\}}\equiv1$. Hence, when $d=1$,
   \begin{equation*}
   I  = \int_{B_\e(x)} \frac{1}{4}\one_{\{|x'-a|<t\}}\one_{\{|x'-b|<s\}} \ud x'=\frac{1}{2}\e \ge \frac{\e^{-1}}{2c^2}ts,
   \end{equation*}
   where the inequality is due to the fact that $t,s\le c\: \e$. Similarly, when $d=2$,
   \begin{align*}
     I = \int_{B_\e(x)}\frac{1}{4\pi^2}\frac{1}{\sqrt{t^2-|x'-a|^2}}\frac{1}{\sqrt{s^2-|x'-b|^2}}\ud x'
       \ge  \frac{1}{4\pi^2 t s} \int_{B_\e(x)}\ud x'
       \ge  \frac{\e^2}{4 \pi t s}
       \ge  \frac{\e^{-2}}{4\pi c^4} ts.
   \end{align*}
   This completes the proof of Proposition \ref{P:nondeg}.
\end{proof}

\begin{remark}[Nonnegativity of fundamental solutions] \label{R:Nonneg}
  The nonnegativity of Green's functions associated with \eqref{E:fde} was first proved in
  \cite{chen.hu.ea:17:space-time} for the case $\gamma=0$, and was later extended in \cite[Theorem
  4.6]{chen.hu.ea:19:nonlinear} to allow $\gamma\ge 0$; see also Remark 1.2 of
  \cite{chen.eisenberg:22:interpolating}. It is known that the Green's function is nonnegative in
  the following three cases:
  \begin{equation}\label{E:Pos}
    \begin{cases}
      (1)\quad \alpha\in(0,2], \; \beta\in(0,1], \: \gamma\ge 0,                 \: d\geq1;       \\[1em]
      (2)\quad 1<\beta<\alpha\leq2,              \: \gamma>0,                    \: 1\leq d\leq3; \\[1em]
      (3)\quad 1<\beta=\alpha<2,                 \: \gamma>\dfrac{d+3}{2}-\beta, \: 1\leq d\leq3. \\
    \end{cases}
  \end{equation}
\end{remark}

\subsection{Feynman Diagram Formula} \label{SS:Feynman}
In this part, we recall the Feynman Diagram formula, which is useful to compute the expectation of
products of multiple Wiener-It\^o integrals. We refer interested readers to Section 5.3 of
\cite{hu:17:analysis} for more details about the multiple Wiener-It\^o integrals.

On the lattice $\mathbb{Z}^2$, we use $(k,\ell)$ to denote a vertex, and an ordered pair $[(k_1,
\ell_1), (k_2, \ell_2)]$ to denote a directed edge pointing from $(k_1, \ell_1)$ to $(k_2, \ell_2)$.

\begin{definition}\label{D:Admissible}
  Let $p\ge 1$ and $\vec{n}=(n_1,\cdots,n_p) \in \mathbb{N}^p$ with $|\vec{n}|=n_1+\cdots+n_p$  be
  given. A \emph{Feynman diagram} is a directed graph $\cD = \left(V, E\right)$ consisting of the
  set of all vertices
  \begin{align*}
    V=\Big\{\left(k, \ell\right) :\: 1\le k\le p,\: 1\le \ell \le n_k\Big\}
  \end{align*}
  and a set $E$ of directed edges satisfying $k_1<k_2$ if $\left[\left(k_1, \ell_1\right),
  \left(k_2, \ell_2\right)\right]\in E$. A Feynman diagram $\cD=\left(V,E\right)$ is called
  \emph{admissible} if each vertex is associated with one and only one edge.  The set of all
  admissible diagrams is denoted by $\bD=\bD_{\vec{n}}$.
\end{definition}

\begin{figure}[htpb]
  \begin{center}
    \begin{tikzpicture}[scale=1, x=4em, y=3em]
      \tikzset{>=latex}
      \draw[->,thick] (-1,0) -- (5,0) node [right] {$k$};
      \draw[->,thick] (0,-0.2) -- (0,4.5) node [above] {$\ell$};
      \foreach \x in {1,...,4}{
          \draw (\x,0.1)--++(0,-0.2) node [below] {$\x$};
      }
      \draw (0.1,1)--++(-0.2,0) node [left] {$n_1=1$};
      \draw (0.1,2)--++(-0.2,0) node [left] {$n_2=n_3=2$};
      \draw (0.1,3)--++(-0.2,0) node [left] {$n_4=3$};
      \draw (0.1,4)--++(-0.2,0) node [left] {$4$};
      \node[] (1p1) at (1,1) {$(1,1)$};
      \node[] (2p2) at (2,2) {$(2,2)$};
      \node[] (2p1) at (2,1) {$(2,1)$};
      \node[] (3p1) at (3,1) {$(3,1)$};
      \node[] (3p2) at (3,2) {$(3,2)$};
      \node[] (4p1) at (4,1) {$(4,1)$};
      \node[] (4p2) at (4,2) {$(4,2)$};
      \node[] (4p3) at (4,3) {$(4,3)$};

      \node[gray!50!white] (1p2) at (1,2) {$(1,2)$};
      \node[gray!50!white] (2p3) at (2,3) {$(2,3)$};
      \node[gray!50!white] (3p3) at (3,3) {$(3,3)$};
      \node[gray!50!white] (4p4) at (4,4) {$(4,4)$};
      \draw[gray!50!white] (1p2) -- (2p3) -- (3p3) -- (4p4);

      \draw[red, thick,->] (1p1) .. controls (1.5, 1.5) .. (2p1);
      \node[red] at (1.5,1.6) {$\mathcal{D}_1$};
      \draw[red, thick,->] (3p1) -- (4p2);
      \draw[red, thick,->] (3p2) -- (4p1);
      \draw[red, thick,->] (2p2)  .. controls (3,2.7) .. (4p3);

      \draw[blue, thick,->] (1p1) .. controls (2, 0.3) .. (3p1);
      \node[blue] at (2,0.2) {$\mathcal{D}_2$};
      \draw[blue, thick,->] (2p1) .. controls (3.1,1.5) .. (4p1);
      \draw[blue, thick,->] (2p2) .. controls (3.1,1.5) .. (4p2);
      \draw[blue, thick,->] (3p2) .. controls (3.6,2.0) .. (4p3);

      \draw[dashed] (1,0) -- (1p1);                   \draw[dashed,gray!50!white] (1p1) -- (1p2); \node[] at (1,-0.8) {$I_1(h_1)$};
      \draw[dashed] (2,0) -- (2p1) -- (2p2);          \draw[dashed,gray!50!white] (2p2) -- (2p3); \node[] at (2,-0.8) {$I_2(h_2)$};
      \draw[dashed] (3,0) -- (3p1) -- (3p2);          \draw[dashed,gray!50!white] (3p2) -- (3p3); \node[] at (3,-0.8) {$I_2(h_3)$};
      \draw[dashed] (4,0) -- (4p1) -- (4p2) -- (4p3); \draw[dashed,gray!50!white] (4p3) -- (4p4); \node[] at (4,-0.8) {$I_3(h_4)$};

      \node[] at (4.35,-0.37) {$=p$};
    \end{tikzpicture}
  \end{center}
  \caption{Two admissible (red $\mathcal{D}_1$ and blue $\mathcal{D}_2$) Feynman diagrams for the
  case when $p=4$, $\vec{n}=\left(1,2,2,3\right)$ and $|\vec{n}|=8$; see Example \ref{Ex:Diagrams}.
Convention \eqref{E:Convent} applies at the gray vertices in the settings of Lemma \ref{L:product}.}

  \label{F:FeynmanD}
\end{figure}
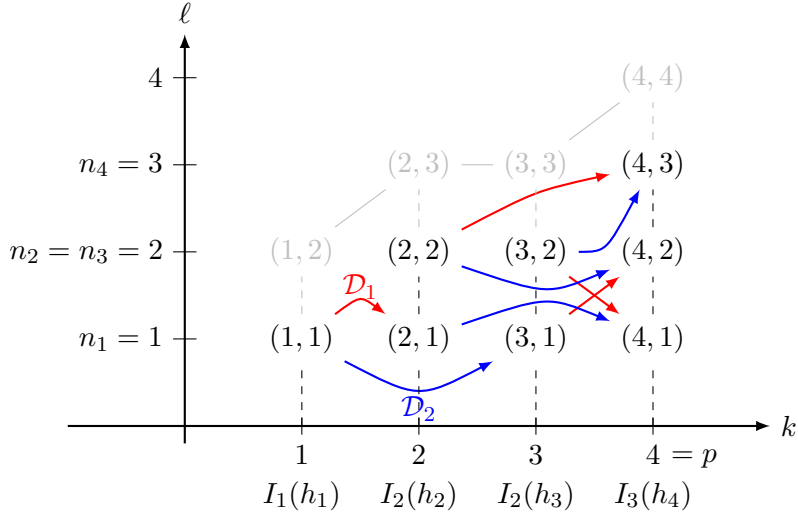


We shall provide a formula for $\E\big[I_{n_1}(h_{1})\cdots I_{n_p}(h_{p})\big]$ for square
integrable functions
\begin{align} \label{E:hi}
  h_{i}:\left(\R_+\times\R^d\right)^{n_i}\rightarrow \R, \quad i=1, \cdots,\: p,
\end{align}
where $I_{n_i}(h_{i})$ refers to the $n_i$-th multiple Wiener-It\^o integral. In particular, Given
an admissible Feynman diagram $\cD\in\bD_{\vec{n}}$, for $h_i$ given in \eqref{E:hi}, denote
\begin{equation}\label{E:FD}
  \begin{split}
    F_\cD(h_{1},\dots,h_{p}) =
    & \int_{\R_+^{|\vec{n}|}} \int_{\R^{d |\vec{n}|}} \ud \mathbf{t} \ud \mathbf{x} \: \prod_{i=1}^p  h_i \left(t_{(i,1)},x_{(i,1)}; \dots ; t_{(i, n_i)}, x_{(i, n_i)}\right) \\
    & \times \prod_{[(k_1, \ell_1),(k_2, \ell_2)]\in E(\cD)}\delta(t_{(k_1, \ell_1)}-t_{(k_2, \ell_2)}) \delta(x_{(k_1, \ell_1)}-x_{(k_2, \ell_2)}),
  \end{split}
\end{equation}
where we use the notations and $\ud \mathbf{t} \ud \mathbf{x} =
\prod_{i=1}^p \prod_{r_i=1}^{n_i}\ud t_{(i,r_i)}\ud x_{(i,r_i)}$. Then we have (see \cite[Theorem
5.3]{hu.wang:21:intermittency} and \cite[Theorems 5.7 and 5.8]{hu:17:analysis}),
\begin{equation}\label{E:FDF}
  \E\big[I_{n_1}(h_{1})\cdots I_{n_p}(h_{p})\big] = \sum_{\cD\in\bD_{\vec{n}}} F_\cD(h_{1},\dots,h_{p}).
\end{equation}

In particular, for any $t>0$ and $x\in\R^d$, considering the multiple Wiener-It\^o integrals
$I_k\left(f_k\left(\cdot; t,x\right)\right)$ in the chaos expansion \eqref{E:Chaos} for the solution
$u(t,x)$ with $f_k$ given in \eqref{E:fk} which is a symmetrization of $g_k$ in \eqref{E:gk}, we
have the following result (see \cite[Theorem 5.4]{hu.wang:21:intermittency}):

\begin{lemma} \label{L:product}
  Let $p\ge 1$ and $\vec{n}=(n_1,\cdots,n_p) \in \mathbb{N}^p$ be given. Fix arbitrary $t>0$ and
  $x_1,\cdots, x_p\in\R^d$. Recall that $f_n(\cdot; t,x)$ and $g_n(\cdot;t,x)$ be given in
  \eqref{E:fk} and \eqref{E:gk}, respectively. Then
  \begin{align} \label{E:product}
      & \E\left[\prod_{\ell=1}^p I_{n_\ell}\left(f_{n_\ell}\left(\cdot;t,x_\ell\right)\right)\right] \notag
    =   \sum_{\cD\in\bD_{\vec{n}}} F_\cD\left(g_{n_1}\left(\cdot;t,x_1\right),\dots,g_{n_p}\left(\cdot;t,x_p\right)\right)\notag \\
    = & \sum_{\cD\in\bD_{\vec{n}}} \int_{[0,t]^{|\vec{n}|}} \int_{\R^{d |\vec{n}|}} \ud \mathbf{t} \ud \mathbf{x}\:
        \left(\prod_{[(k_1, l_1),(k_2, l_2)]\in E(\mathcal{D})}\delta(t_{(k_1, l_1)}-t_{(k_2, l_2)}) \delta(x_{(k_1, l_1)}-x_{(k_2, l_2)})\right) \notag \\
      & \times \left( \prod_{i=1}^{p}J_0\left(t_{(i,1)}\right) \one_{\{0<t_{(i,1)}<\dots<t_{(i,n_i)}< t\}}
                      \prod_{r_i=1}^{n_i}p\big(t_{(i,r_i+1)}-t_{(i,r_i)}, x_{(i,r_i+1)}-x_{(i,r_i)}\big) \right),
  \end{align}
  where we have used the convention that
\begin{align} \label{E:Convent}
  \left(t_{(i,n_i+1)},\: x_{(i,n_i+1)}\right)=(t,x_i) \quad  \text{for all $i=1,\cdots p$.}
\end{align}
\end{lemma}

\begin{example} \label{Ex:Diagrams}
 Let $\mathcal{D}_1$ (resp. $\mathcal{D}_2$) refer to the red (resp. blue) admissible Feynman
 diagram in Figure \ref{F:FeynmanD}. Under the setting of Lemma \ref{L:product}, let
 \begin{align*}
   F_i = F_{\mathcal{D}_i}\left(g_1\left(\cdot;t,x\right),g_2\left(\cdot;t,x\right),g_2\left(\cdot;t,x\right),g_3\left(\cdot;t,x\right)\right)
   \quad  i=1,2.
 \end{align*}
 Then we claim that $F_1=0$ because its integrand contains the following factor:
 \begin{align*}
   \one_{\{0<t_{(3,1)}<t_{(3,2)}<t\}} \one_{\{0<t_{(3,2)}<t_{(3,1)}<t_{(2,2)}<t\}}
 \end{align*}
 which is identically equal to zero. Hence, due to the delta potentials and the simplex conditions
 in \eqref{E:product}, edges starting from one column should not cross with each other. This is the
 case for $F_2$:
 \begin{align*}
 F_2  = & \int_{[0,t]^4} \ud t_{(1,1)} \ud t_{(2,1)} \ud t_{(2,2)} \ud t_{(3,2)} \int_{\R^{4d}} \ud x_{(1,1)} \ud x_{(2,1)} \ud x_{(2,2)} \ud x_{(3,2)}                               \\
        & \times J_0\left(t_{(1,1)}\right) \one_{\{0<t_{(1,1)}<t\}} p_{t-t_{(1,1)}}\left(x-x_{(1,1)}\right)                                                                     \\
        & \times J_0\left(t_{(2,1)}\right) \one_{\{0<t_{(2,1)}<t_{(2,2)}<t\}} p_{t-t_{(2,2)}}\left(x-x_{(2,2)}\right) p_{t_{(2,2)}-t_{(2,1)}}\left(x_{(2,2)}-x_{(2,1)}\right)   \\
        & \times J_0\left(t_{(1,1)}\right) \one_{\{0<t_{(1,1)}<t_{(3,2)}<t\}} p_{t-t_{(3,2)}}\left(x-x_{(3,2)}\right) p_{t_{(3,2)}-t_{(1,1)}}\left(x_{(3,2)}-x_{(1,1)}\right)   \\
        & \times J_0\left(t_{(2,1)}\right) \one_{\{0<t_{(2,1)}<t_{(2,2)}<t_{(3,2)}<t\}}                                                                                         \\
        & \times p_{t-t_{(3,2)}}\left(x-x_{(3,2)}\right) p_{t_{(3,2)}-t_{(2,2)}}\left(x_{(3,2)}-x_{(2,2)}\right) p_{t_{(2,2)}-t_{(2,1)}}\left(x_{(2,2)}-x_{(2,1)}\right)              \\
      = & \int_{[0,t]^4} \ud t_{(1,1)} \ud t_{(2,1)} \ud t_{(2,2)} \ud t_{(3,2)} \int_{\R^{4d}} \ud x_{(1,1)} \ud x_{(2,1)} \ud x_{(2,2)} \ud x_{(3,2)}                               \\
        & \times J_0^2\left(t_{(1,1)}\right) \one_{\{0<t_{(1,1)}<t_{(3,2)}<t\}} p_{t-t_{(1,1)}}\left(x-x_{(1,1)}\right) p_{t_{(3,2)}-t_{(1,1)}}\left(x_{(3,2)}-x_{(1,1)}\right) \\
        & \times J_0^2\left(t_{(2,1)}\right) \one_{\{0<t_{(2,1)}<t_{(2,2)}<t_{(3,2)}<t\}} p_{t-t_{(2,2)}}\left(x-x_{(2,2)}\right)                                               \\
        & \times p^2_{t-t_{(3,2)}}\left(x-x_{(3,2)}\right) p_{t_{(3,2)}-t_{(2,2)}}\left(x_{(3,2)}-x_{(2,2)}\right)  p^2_{t_{(2,2)}-t_{(2,1)}}\left(x_{(2,2)}-x_{(2,1)}\right).
 \end{align*}
 Note that the original $2\times 8$-multiple integral has been collapsed to the above $2\times
 4$-multiple integral. The remaining variables are the roots of all edges in $E(\mathcal{D})$.
\end{example}

\begin{definition} \label{D:Balanced}
  For any $m\in \mathbb{N}$ and $p\in 2\mathbb{N}$, we say that $\vec{n}=(n_1,\cdots,n_p)$ is a {\it
  balanced partition} of $2m$ if
  \begin{enumerate}
    \item $|\vec{n}|=2m$;
    \item $n_i\in \{m_p, m_p+1\}$ for all $i=1,\cdots, p$, where $m_p:=\Floor{2m/p}$;
    \item $ n_1+\cdots + n_{p/2}=m$;
    \item $r_p\in[0,p)$ is the remainder of $2m/p$.
  \end{enumerate}
  Moreover, under this setting, an admissible Feynman diagram $\mathcal{D}=\left(V,E\right)$ is
  called a {\it balanced diagram} provided
  \begin{align*}
   \left[(k_1, \ell_1), (k_2,\ell_2)\right] \in E(\cD) \quad \Longrightarrow \quad \ell_1=\ell_2 \quad \text{and} \quad k_1\le p/2<k_2.
  \end{align*}
  The set of all balanced diagrams is denoted by $\mathbb{D}_{\vec{n}}^=$. It is clear that
  $\mathbb{D}_{\vec{n}}^= \subset \mathbb{D}_{\vec{n}}$.
\end{definition}

It is straightforward to show the existence  of a balanced partition, which is however not unique in
general. Let us check a few examples:

\begin{example}
  \begin{itemize}[wide=0pt]
    \item[(1)] In Figure \ref{F:FeynmanD}, we have $p=m=4$. The partition $\vec{n}=(1,2,2,3)$ is not
      a balanced partition. Indeed, for this example, the only balanced partition is
      $\vec{n}=(2,2,2,2)$.
    \item[(2)] If $p=4$ and $m=3$, the following partitions are all balanced:
      \begin{align*}
         \left(1,2,2,1\right), \quad \left(2,1,2,1\right), \quad \left(1,2,1,2\right).
      \end{align*}
      However, $\left(1,1,2,2\right)$ is not balanced.
    \item[(3)] If $p=6$ and $m=7$, it is easy to check that $\vec{n}=\left(3,2,2,2,3,2\right)$ is a
      balanced partition, upon which a balanced diagram is given; see Figure \ref{F:Balanced}.
  \end{itemize}
\end{example}

\begin{figure}[htpb]
  \begin{center}
    \begin{tikzpicture}[scale=1, x=4em, y=3em]
      \tikzset{>=latex}
      \draw[->,thick] (-1,0) -- (7,0) node [right] {$k$};
      \draw[->,thick] (0,-0.2) -- (0,4.5) node [above] {$\ell$};
      \foreach \x in {1,...,6}{
          \draw (\x,0.1)--++(0,-0.2) node [below] {$\x$};
      }
      \draw (0.1,1)--++(-0.2,0) node [left] {$1$};
      \draw (0.1,2)--++(-0.2,0) node [left] {$m_p=2$};
      \draw (0.1,3)--++(-0.2,0) node [left] {$m_p+1=3$};
      \draw (0.1,4)--++(-0.2,0) node [left] {$4$};
      \node[red ] (1p1) at (1,1) {$(1,1)$};
      \node[red ] (1p2) at (1,2) {$(1,2)$};
      \node[red ] (1p3) at (1,3) {$(1,3)$};
      \node[red ] (2p1) at (2,1) {$(2,1)$};
      \node[red ] (2p2) at (2,2) {$(2,2)$};
      \node[red ] (3p1) at (3,1) {$(3,1)$};
      \node[red ] (3p2) at (3,2) {$(3,2)$};
      \node[blue] (4p1) at (4,1) {$(4,1)$};
      \node[blue] (4p2) at (4,2) {$(4,2)$};
      \node[blue] (5p1) at (5,1) {$(5,1)$};
      \node[blue] (5p2) at (5,2) {$(5,2)$};
      \node[blue] (5p3) at (5,3) {$(5,3)$};
      \node[blue] (6p1) at (6,1) {$(6,1)$};
      \node[blue] (6p2) at (6,2) {$(6,2)$};
      \node[] (t) at (3.5,4.6) {$\vec{n}=\left(3,2,2,2,3,2\right)$};
      \node[] (7p3) at (7.5,3) {$r_p=2$};
      \draw[dotted] (8,2.5) -- (7,2.5) --++(0,1.0) --++ (0,-3.5);
      \node[text width=5em] (7p15) at (7.5,1.5) {
          \begin{gather*}
            2m_pp\\ || \\ 12
          \end{gather*}
        };

      \node[gray!50!white] (1p4) at (1,4) {$(1,4)$};
      \node[gray!50!white] (2p3) at (2,3) {$(2,3)$};
      \node[gray!50!white] (3p3) at (3,3) {$(3,3)$};
      \node[gray!50!white] (4p3) at (4,3) {$(4,3)$};
      \node[gray!50!white] (5p4) at (5,4) {$(5,4)$};
      \node[gray!50!white] (6p3) at (6,3) {$(6,3)$};
      \draw[gray!50!white] (1p4) -- (2p3) -- (3p3) -- (4p3) -- (5p4) -- (6p3);

      \draw[->] (1p3) .. controls (3, 3.5) .. (5p3);
      \draw[->] (1p2) .. controls (2, 3) and (5,3) .. (6p2);
      \draw[->] (2p2) .. controls (3, 2.5) .. (4p2);
      \draw[->] (3p2) .. controls (4, 1.5) .. (5p2);
      \draw[->] (1p1) .. controls (2, 0.2) and (4, 0.2) .. (5p1);
      \draw[->] (2p1) .. controls (3, 1.5) and (5, 1.5) .. (6p1);
      \draw[->] (3p1) .. controls (3.5, 0.4) .. (4p1);

      \draw[dashed] (1,0) -- (1p1) -- (1p2) -- (1p3); \draw[dashed,gray!50!white] (1p3) -- (1p4);
      \draw[dashed] (2,0) -- (2p1) -- (2p2);          \draw[dashed,gray!50!white] (2p2) -- (2p3);
      \draw[dashed] (3,0) -- (3p1) -- (3p2);          \draw[dashed,gray!50!white] (3p2) -- (3p3);
      \draw[dashed] (4,0) -- (4p1) -- (4p2) -- (4p3); \draw[dashed,gray!50!white] (4p3) -- (4p4);
      \draw[dashed] (5,0) -- (5p1) -- (5p2) -- (5p3); \draw[dashed,gray!50!white] (5p3) -- (5p4);
      \draw[dashed] (6,0) -- (6p1) -- (6p2);          \draw[dashed,gray!50!white] (6p2) -- (6p3);

      \node[] at (6.35,-0.37) {$=p$};
      \draw[dotted, thick] (3.5,4.3) -- ++ (0,-6.5);
      \node[] () at (2,-1) {$p/2=3$ columns};
      \node[] () at (5,-1) {$p/2=3$ columns};
      \node[red ] () at (2,-1.7) {$m=7$ starting vertices};
      \node[blue] () at (5,-1.7) {$m=7$ ending vertices};
    \end{tikzpicture}
  \end{center}

  \caption{One example of the balanced partition in case of $m=7$ and $p=6$ with a balanced diagram
  (all edges are horizontal starting from the left half of the vertices pointing to the right half).
The grayed-out vertices correspond to the convention \eqref{E:Convent}.}

  \label{F:Balanced}
\end{figure}
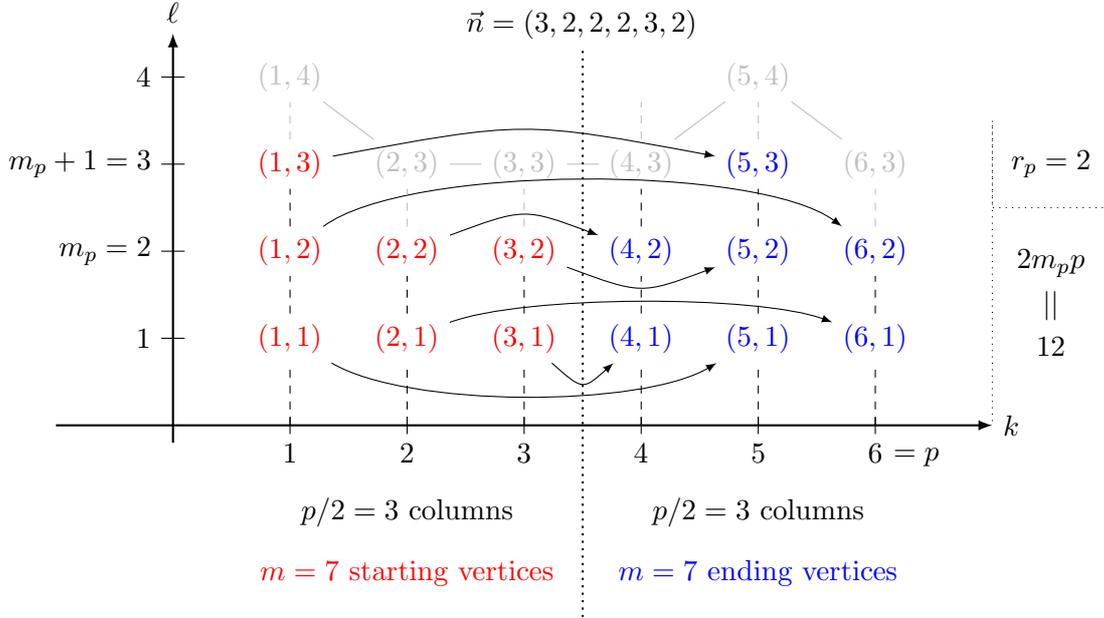

\subsection{Proof of the lower bounds} \label{SS:Lower}

In this subsection, we derive a lower bound for $\E\left[u(t,x)^p\right]$ which is consistent with
the upper bound obtained in Theorem \ref{T:fde}; see also \eqref{E:upper-tplim}.

\begin{theorem}\label{T:lower-bd}
  Assume that
  \begin{enumerate}[wide=0pt]
    \item either $\beta\in(0,2)$ and the fundamental function $p(t,x)$ is nonnegative or
      $\alpha=\beta=2$ and $\gamma=0$;
    \item the initial position $u_0$ is strictly positive and the initial velocity $u_1$ is
      nonnegative;
    \item Dalang's condition \eqref{E:Dalang'} is satisfied.
  \end{enumerate}
  Then we have for all $t>0$, $x_1,\cdots,x_p\in\R^d$, and $p\in 2\mathbb{N}$ such that $t_p = t\:
  p^{1+1/(\theta+1)}$ (see \eqref{E:theta}) is sufficiently large (recall that $\theta$ is given in
  \eqref{E:theta}), there exist constants $c_1$ and $c_2$ that do not depend on $(t,x_1,\cdots,
  x_p,p)$ such that
  \begin{equation}\label{E:lb}
    \E\left[\prod_{j=1}^{p} u(t,x_j)\right]\geq  c_1\exp\left( c_2\: t\: p^{1+1/(\theta+1)} \right).
  \end{equation}
\end{theorem}

\begin{proof}
  The proof is based on the Feynman diagram formula for the $p$-th moments and the non-degeneracy
  property of the Green's function -- Proposition \ref{P:nondeg}, which is inspired by \cite[Theorem
  3.6]{hu.wang:21:intermittency}. Choose an arbitrary even integer $p$ and let $t>0$ and
  $x_1,\cdots, x_p\in\R^d$ be fixed. By \eqref{E:Chaos}, we have
  \begin{align} \label{E:uFC}
    \E\left[\prod_{j=1}^{p}u(t,x_j)\right]
    & = \E\left[\prod_{j=1}^{p}\sum_{n_j=0}^{\infty}I_{n_j}(f_{n_j}(\cdot,t,x_j))\right] \notag                                                \\
    & = \sum_{n_1=0}^{\infty}\cdots \sum_{n_p=0}^{\infty}\E\Big[I_{n_1}(f_{n_1}(\cdot,t,x_1)) \cdots I_{n_p}(f_{n_p}(\cdot,t,x_p))\Big] \notag \\
    & = \sum_{m=0}^{\infty} \sum_{\substack{\vec{n}\in \mathbb{N}^p \cr |\vec{n}|=2m }}\sum_{\cD\in\bD_{\vec{n}}} F_\cD\left(g_{n_1}\left(\cdot;t,x_1\right),\dots,g_{n_p}\left(\cdot;t,x_p\right)\right),
  \end{align}
  where we have used the convention that $I_0(f_0(\cdot,t,x)) = J_0(t)$. We will find the lower
  bound in three steps: \bigskip

  {\noindent\bf Step 1.~} We first take care of the three summations in \eqref{E:uFC}. Applying the
  Feynman diagram formula in Lemma \ref{L:product} and noting that $\inf_{s\in[0,t]}J_0(s) \ge
  u_0$ (see \eqref{E:J0(t)}), we have
  \begin{align} \label{E:SumUlow}
     \E\left[\prod_{j=1}^{p}u(t,x_j)\right]
     =  & \sum_{m=0}^{\infty} \sum_{\substack{\vec{n}\in \mathbb{N}^p \\ |\vec{n}|=2m }}\sum_{\cD\in\bD_{\vec{n}}} F_\cD\left(g_{n_1}\left(\cdot;t,x_1\right),\dots,g_{n_p}\left(\cdot;t,x_p\right)\right)\notag\\
    \ge & c_0^{p}\: \sum_{m=p/2}^\infty \sum_{\substack{\vec{n}\in \mathbb{N}^p \\ |\vec{n}|=2m\\ \text{$\vec{n}$ is balanced}}}\sum_{\cD\in\bD_{\vec{n}}^=} I_0, \quad \text{with}
  \end{align}
  \begin{align} \label{E:I_0}
    \begin{aligned}
    I_0 := & \int_{[0,t]^{2m}}  \int_{\R^{2md}} \left(\prod_{[(k_1, l_1),(k_2, l_2)]\in E(\mathcal{D})}\delta\left(t_{(k_1, l_1)}-t_{(k_2, l_2)}\right) \delta\left(x_{(k_1, l_1)}-x_{(k_2, l_2)}\right)\right) \\
           & \times \prod_{i=1}^{p}\one_{\left\{0<t_{(i,1)}<\dots<t_{(i,n_i)}< t\right\}} \prod_{r_i=1}^{n_i}p\big(t_{(i,r_i+1)}-t_{(i,r_i)}, x_{(i,r_i+1)}-x_{(i,r_i)}\big) \ud\mathbf{t} \ud \mathbf{x},
    \end{aligned}
  \end{align}
  where we have used the assumption that $p(t,x)$ is nonnegative and the convention
  \eqref{E:Convent}. Here we emphasize that:
  \begin{enumerate}
    \item in the first summation of \eqref{E:SumUlow}, we require $m\ge p/2$;
    \item in the second summation of \eqref{E:SumUlow}, we only consider the balanced partitions of
      $2m$;
    \item in the third summation of \eqref{E:SumUlow}, we restrict us to the balanced diagrams
      $\mathbb{D}_{\vec{n}}^=$.
  \end{enumerate}

  One may check Figure \ref{F:Balanced} for some examples of the selected Feynman diagrams. Recall
  that $m_p=\Floor{2m/p}$ and $r_p$ is the remainder of $2m/p$ (see Definition \ref{D:Balanced}), namely,
  \begin{align*}
    2m = m_p \times p  + r_p, \quad \text{with $0\le r_p<p$}.
  \end{align*}
  It is easy to see that $r_p$ has to be an even integer. With these restrictions, for each fixed
  $m\ge p/2$, one can check that the total number of diagrams satisfying (2) and (3) is at least
  $\left( \left(p/2\right)!\right)^{m_p} \times \left(r_p/2\right)!$. \bigskip

  {\noindent\bf Step 2.~} Now we proceed to shrink the integral region for the $\ud
  \mathbf{t}$-integral of $I_0$ in \eqref{E:I_0} as follows: Denote $L=\frac{t}{m_p+1}$,
  $t_i=\frac{(2i-1)t}{2(m_p+1)}$, $a_i=t_i-L/4$ and $b_i=t_i+L/4$ for $i=1,\dots,m_p+1$. Let
  $I_i=[a_i,b_i]$. Then these intervals $I_i$ are disjoint with length $L/2$; See Figure
  \ref{F:Time} for an illustration.

  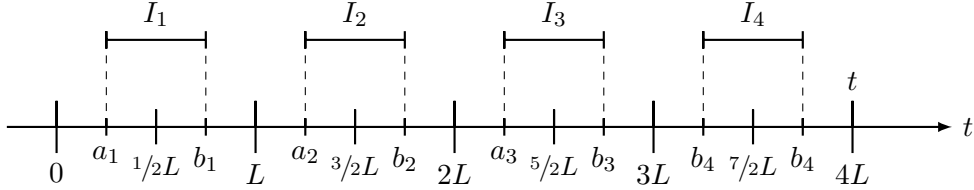
\begin{figure}[htpb]
    \begin{center}
      \begin{tikzpicture}[scale=1, x=1.7em, y=3em]
        \tikzset{>=latex}
        \draw[->,thick] (-1,0) -- (18,0) node [right] {$t$};

        \foreach \x in {1,5,9,13}{
            \draw[thick] (\x,1) --++ (0,-0.1) --++ (0,0.2) -- ++(0,-0.1) --++ (2,0)++ (0,-0.1) --++ (0,0.2);
            \draw[dashed] (\x,0) -- ++ (0,1) ++ (2,0) --++(0,-1);
        }
        \def\inc{0.3}
        \draw[thick] (0,0)--++(0,\inc) --++(0,-2*\inc) node [below] {$0$};
        \draw[thick] (4,0)--++(0,\inc) --++(0,-2*\inc) node [below] {$L$};
        \draw[thick] (8,0)--++(0,\inc) --++(0,-2*\inc) node [below] {$2L$};
        \draw[thick] (12,0)--++(0,\inc) --++(0,-2*\inc) node [below] {$3L$};
        \draw[thick] (16,0)--++(0,\inc) node [above] {$t$} --++(0,-2*\inc) node [below] {$4L$};

        \def\inc{0.2}
        \draw[thick] (2,0) --++(0,\inc) node [above,yshift=2.5em] {$I_1$} --++(0,-2*\inc) node [below] {\small $\sfrac{1}{2}L$};
        \draw[thick] (6,0) --++(0,\inc) node [above,yshift=2.5em] {$I_2$} --++(0,-2*\inc) node [below] {\small $\sfrac{3}{2}L$};
        \draw[thick] (10,0)--++(0,\inc) node [above,yshift=2.5em] {$I_3$} --++(0,-2*\inc) node [below] {\small $\sfrac{5}{2}L$};
        \draw[thick] (14,0)--++(0,\inc) node [above,yshift=2.5em] {$I_4$} --++(0,-2*\inc) node [below] {\small $\sfrac{7}{2}L$};

        \def\inc{0.1}
        \draw[thick] (1,0) --++(0,\inc) --++(0,-2*\inc) node [below] {$a_1$};
        \draw[thick] (3,0) --++(0,\inc) --++(0,-2*\inc) node [below] {$b_1$};
        \draw[thick] (5,0) --++(0,\inc) --++(0,-2*\inc) node [below] {$a_2$};
        \draw[thick] (7,0) --++(0,\inc) --++(0,-2*\inc) node [below] {$b_2$};
        \draw[thick] (9,0) --++(0,\inc) --++(0,-2*\inc) node [below] {$a_3$};
        \draw[thick] (11,0)--++(0,\inc) --++(0,-2*\inc) node [below] {$b_3$};
        \draw[thick] (13,0)--++(0,\inc) --++(0,-2*\inc) node [below] {$b_4$};
        \draw[thick] (15,0)--++(0,\inc) --++(0,-2*\inc) node [below] {$b_4$};

      \end{tikzpicture}
    \end{center}
    \caption{Some illustrations for Step 2 in the proof of Theorem \ref{T:lower-bd} with $m_p=4$.}
    \label{F:Time}
  \end{figure}

  For the integral with respect to time variables in \eqref{E:p-mom}, we only integrate on the
  region where for each $i\in\{1, \dots, p\}$, $t_{(i, r_i)}$ is in $I_i$ for $1\leq r_i\leq n_i$,
  and hence
  \begin{equation}\label{split-t}
    \frac{t}{2(m_p+1)}=\frac12 L\leq t_{(i, r_i+1)}-t_{(i, r_i)}\leq \frac32 L=\frac{3t}{2(m_p+1)}.
  \end{equation}
  Then, for each integer $m\ge p/2$, by choosing
  \begin{equation}\label{E:em}
    \e:= \left( p\: t/16m\right)^{\beta/\alpha},
  \end{equation}
  we have that
  \begin{equation}\label{split-t'}
    t_{(i, r_i+1)}-t_{(i, r_i)}\in\left[2\: \e^{\alpha/\beta},12\: \e^{\alpha/\beta}\right],
    \quad i=1,\cdots,p.
  \end{equation}
  \bigskip

  {\bf\noindent Step 3.~} Now we study the spatial integral portion of $I_0$ in \eqref{E:I_0}, which
  is equal to
  \begin{align*}
    \int_{\R^{2md}}\ud \mathbf{x} \:
    \prod_{i=1}^{p}\prod_{r_i=1}^{n_i}p\big(t_{(i,r_i+1)}-t_{(i,r_i)}, x_{(i,r_i+1)}-x_{(i,r_i)}\big)
    \left(\prod_{[(k_1, l_1),(k_2, l_2)]\in E(\cD)}\delta(x_{(k_1, l_1)}-x_{(k_2, l_2)})\right).
  \end{align*}
  It is bounded from below if one replaces the integral region $\R^{2md}$ by
  $\left(B_{\varepsilon}^2(x)\right)^{m}$ for any $\e>0$. In particular, by Step 2, we see that
  $\mathbf{t} $ satisfies $t_{(i, r_i)}\in I_i$ for $1\le r_i \le n_i, 1\le i \le p$, i.e.,
  \eqref{split-t'} holds true. Hence, we can apply Proposition \ref{P:nondeg} with $\e$ given in
  \eqref{E:em} and $c=12$ to bound the above integral from be as follows:
  \begin{align} \label{E:xLowBd}
    \ge C^m \left(pt /m\right)^{-\beta md/\alpha}  \prod_{i=1}^{p}\prod_{r_i=1}^{n_i} |t_{(i,r_i+1)}-t_{(i,r_i)}|^{\beta+\gamma-1}.
  \end{align}
  Therefore, we can find a lower bound of $I_0$ in \eqref{E:I_0} with only time integral:
  \begin{align} \label{E:Step3}
    \begin{aligned}
    I_0 \ge C^m \left(pt /m\right)^{-\beta md/\alpha} \int_{[0,t]^{2m}}
    & \ud \mathbf{t}\: \prod_{i=1}^{p} \prod_{r_i=1}^{n_i} |t_{(i,r_i+1)}-t_{(i,r_i)}|^{\beta+\gamma-1} \one_{\{t_{(i,r_{i})}\in I_i\}}\\
    & \times \left(\prod_{[(k_1, l_1),(k_2, l_2)]\in E(\cD)}\delta(t_{(k_1, l_1)}-t_{(k_2, l_2)})\right).
    \end{aligned}
  \end{align}
  \bigskip

  {\bf\noindent Step 4.~} Finally, we will carry out the remaining $\ud \mathbf{t}$-integral in
  \eqref{E:Step3} and complete the proof. We will use $C$ to denote a generic constant that does not
  depend on $\left(t,p,m\right)$ and whose value may change at each appearance. Now denote the
  integral in \eqref{E:Step3} by $I$, which can be bound from below as follows:
  \begin{align*} 
    I & \ge C^{m} L^{2m(\beta+\gamma-1)} \int_{[0,t]^{2m}} \ud \mathbf{t} \: \prod_{i=1}^{p} \prod_{r_i=1}^{n_i}  \one_{\{t_{(i,r_{i})}\in I_i\}} \left(\prod_{[(k_1, l_1),(k_2, l_2)]\in E(\cD)}\delta(t_{(k_1, l_1)}-t_{(k_2, l_2)})\right) \notag \\
      & = C^{m} L^{2m(\beta+\gamma-1)}\left(\frac{L}{2}\right)^m = C^{m} \left(\frac t{m_p}\right)^{m(2\beta+2\gamma-1)}.
  \end{align*}
  Replace the space-time integral in \eqref{E:SumUlow} by the above lower bound, together with the
  factor in front of the integral in \eqref{E:Step3}, to see that
  \begin{align*} 
    \E\left[\prod_{j=1}^{p}u(t,x_j)\right]
      & \ge c_0^{p}\: \sum_{m\ge p/2} \sum_{\substack{\vec{n}\in \mathbb{N}^p \\ |\vec{n}|=2m\\ \text{$\vec{n}$ is balanced}}}\sum_{\cD\in\bD_{\vec{n}}^=} C^m \left(\frac{pt}{m}\right)^{-\beta d m/\alpha}  \left(\frac t{m_p}\right)^{m(2\beta+2\gamma-1)}  \\
      & \ge c_0^p \sum_{m\ge p/2}  C^m  \left(\frac{pt}{m}\right)^{-\beta d m/\alpha}   \left(\frac t{m_p}\right)^{m(2\beta+2\gamma-1)} \left( \left(p/2\right)!\right)^{m_p} \times \left(r_p/2\right)!,
  \end{align*}
  where we have used the fact that there are at least $\left( \left(p/2\right)!\right)^{m_p} \times \left(r_p/2\right)!$
  terms in the double summations.

  Thanks to the following well known bounds to the Gamma function, which is related to the Stirling
  formula (see, e.g., 5.1.10 on p. 141 of \cite{olver.lozier.ea:10:nist})
  \begin{align} \label{E:Stirling}
    \sqrt{2\pi n} \left(\frac{n}{e}\right)^n < n! < 2 \sqrt{2\pi n} \left(\frac{n}{e}\right)^n, \quad
    \text{for all $n\ge 1$},
  \end{align}
  we see that up to a constant, one can replace $n!$ by $\sqrt{2\pi n} \left(n/e\right)^n$. Hence, by \eqref{E:Stirling} and the fact $\left(r_p/p\right)^{r_p/2}\ge e^{-\frac{p}{2e}}\ge C^m$, we have
  \begin{equation*}
    \begin{split}
       \left(\left(p/2\right)!\right)^{m_p}\cdot \left(r_p/2\right)! & \ge C^m \left(p\right)^{(p\:m_p)/2}\cdot (p)^{r_p/2} \left(r_p/p\right)^{r_p/2} \\
         &\ge C^m p^{\frac{p\times m_p+r_p}{2}} = C^m p^m.
    \end{split}
  \end{equation*}

  Then bound $t/m_p$ in the above lower bound from below by $pt/(2m)$ and put $c_0^p$ into $C^m$, to see that
  \begin{align*}
    \E\left[\prod_{j=1}^{p}u(t,x_j)\right]
      & \ge \sum_{m\ge p/2} C^m  \left(\frac{pt}{m}\right)^{-\beta d m/\alpha} \left(\frac{pt}{2m}\right)^{m(2\beta+2\gamma-1)} p^m\\
      & \ge \sum_{m\ge p/2}  \left(\frac{\left(Cp^{1+\frac{1}{\theta+1}}t\right)^m}{m!}\right)^{\left(\theta+1\right)} .
  \end{align*}
  Let $n:=Cp^{1+\frac{1}{\theta+1}}t$, if $n$ is sufficiently large, we have $p/2\le n$, then for sufficiently large $n$,
  \begin{align*}
    \E\left[\prod_{j=1}^{p}u(t,x_j)\right]
      & \ge \sum_{m\ge p/2}  \left(\frac{n^m}{m!}\right)^{\left(\theta+1\right)}
       \ge \sum_{m\ge n} \left(\frac{n^m}{m!}\right)^{\left(\theta+1\right)}
       \ge c_1 \exp\left(c_2 p^{1+\frac{1}{\theta+1}}t\right)
  \end{align*}
  where the third inequality follows from the Lemma \ref{L:exp-1}. This completes the proof of Theorem \ref{T:lower-bd}.
 \end{proof}

Finally, let us explain in the following remark why the space-time white noise  case requires a separate treatment.
\begin{remark} \label{R:Hu-Wang}
  The lower bounds for equations with space-time colored noise whose covariance function is given by
  \eqref{E:NoiseCC} were obtained in Hu-Wang \cite{hu.wang:21:intermittency}, by which our
  methodology is inspired.  Here it is important to make a distinction in the treatment between the
  colored noise case and the white noise case: Firstly, in the white noise case, the balanced
  Feynman diagrams (see Definition \ref{D:Balanced}) make the right contribution to the desired
  lower bound, and this is different from the colored noise case (see Step 1 in the proof of Theorem
  3.6 {\it ibid.}). Secondly, Hu-Wang's proof relies heavily on the assumption $\gamma(t) \ge C
  |t|^{-\theta}$ and $\Lambda(x)\ge C |x|^{-\lambda}$ for small values of $t$ and $x$ (see Step 2 in
  the proof Theorem 3.5 {\it ibid.}), which does not hold for the white noise case. As a
  consequence, the small ball nondegeneracy property for Green's function (see Section 3.1 {\it
  ibid.}) which plays a key role in Hu-Wang's argument does not apply to the white noise case. To
  resolve this issue, we develop a similar nondegeneracy property for the product of Green's
  functions (see Proposition \ref{P:nondeg}).
\end{remark}

\appendix
\section{Preliminaries on fractional integrals and derivatives} \label{S:Prelim}

In this section, we provide some preliminaries on fractional integrals and derivatives in the sense
of Riemann-Liouville and we also recall Caputo fractional derivatives. We refer to
\cite{kilbas.srivastava.ea:06:theory,podlubny:99:fractional} for details.

Let $\alpha\ge 0$ be a constant and $[a,b]$ be a finite interval on $\bR$. Let $f(x)$ be a
complex-valued function defined on $[a,b]$. We only recall the left-sided integrals/derivatives
which will be used in this article, and the right-sided case is similar and thus omitted.
\begin{definition} \label{def:R-L-I}
  The {\em Riemann-Liouville integral} $I_{a+}^{\alpha}f$ of order $\alpha\ge 0$ is defined by
  \begin{equation}\label{def-R-I}
    (I_{a+}^{\alpha}f)(x) := \frac{1}{\Gamma(\alpha)}\int_{a}^{x}\frac{f(t)}{(x-t)^{1-\alpha}}\ud t, \quad x\in[a,b].
  \end{equation}
\end{definition}

\begin{definition}\label{def:D}
  The {\em Riemann-Liouville derivative} $D_{a+}^{\alpha}f$ of order $\alpha\in \R_+\setminus
  \mathbb{N}$ is defined by
  \begin{align*}
    \left(D_{a+}^{\alpha}f\right)(x)
    : = \frac{\ud^n}{\ud x^n}\left(I_{a+}^{n-\alpha}f\right)(x)
      = \frac{1}{\Gamma\left(n-\alpha\right)}\frac{\ud^n}{\ud x^n}
        \int_{a}^{x}\frac{f(t)}{(x-t)^{n-\alpha}} \ud t, \quad n = \Ceil{\alpha},
  \end{align*}
  and when $\alpha= n\in \mathbb{N}$, $\left(D_{a+}^{\alpha}f\right)(x) = \dfrac{\ud^n}{\ud
  t^n}f(t)$.  We use the convention that $D_{a+}^{\alpha}f:= I_{a+}^{-\alpha}f$, when $\alpha<0$.
\end{definition}
For $1\leq p\leq\infty$, we denote by $L^p(a,b)$ the set of complex-valued functions $f$ on $[a,b]$
with finite $L^p$-norm $\Vert f\Vert_p$, where
\begin{equation*}
 \Vert f\Vert_p =
 \begin{cases}\displaystyle{\left( \int_{a}^{b}|f(x)|^pdt\right)^\frac{1}{p}}, & 1\leq p<\infty, \vspace{0.3cm}  \\
\displaystyle{ \esssup_{a\leq x\leq b}|f(x)|}, & p=\infty.
  \end{cases}
\end{equation*}

\begin{lemma}[Property 2.2 on p. 74 of \cite{kilbas.srivastava.ea:06:theory}]\label{L:I-D}
  For $\alpha>\beta>0$ and $f(x) \in L^p(a,b), 1\leq p \leq\infty$, we have
  \begin{align*}
    \left(D_{a+}^{\beta}I_{a+}^{\alpha}f\right)(x) = I_{a+}^{\alpha-\beta}f(x), \quad \text{for $x\in[a,b]$ almost everywhere.}
  \end{align*}
\end{lemma}

\begin{lemma}[Property 2.5 on p. 81 of \cite{kilbas.srivastava.ea:06:theory}] \label{L:ItDt}
  For $\alpha, \beta>0$, we have
  \begin{equation*}
     \left(I_{0+}^\alpha t^{\beta-1}\right)(x) = \frac{\Gamma(\beta)}{\Gamma(\beta+\alpha)}x^{\beta+\alpha-1} \quad \text{and} \quad
     \left(D_{0+}^\alpha t^{\beta-1}\right)(x) = \frac{\Gamma(\beta)}{\Gamma(\beta-\alpha)}x^{\beta-\alpha-1}.
  \end{equation*}
\end{lemma}

\begin{definition}[(2.4.1) on p. 91 of \cite{kilbas.srivastava.ea:06:theory}]\label{def:CD}
  The {\em Caputo fractional derivative} of order $\alpha$ on $[a,b]$ can be defined via the {\em
  Riemann-Liouville derivative} as follows,
  \begin{equation}\label{E:CD}
    \left({}^CD_{a+}^\alpha f\right)(x) :=
    \left(D_{a+}^{\alpha}\left[f(\cdot)-\sum_{k=0}^{\Ceil{\alpha}-1}\frac{f^{(k)}(a)}{k!}(\cdot-a)^k\right]\right)(x), \quad x\in[a,b].
  \end{equation}
\end{definition}

We are ready to recall the formulas of the solutions to Cauchy problems for differential equations
with the Caputo fractional derivatives. For $\gamma\in[0,1)$, we define the {\em weighted space
$C_{\gamma}[a,b]$ of continuous functions} as follows, $$ C_{\gamma}[a,b] := \big\{ f(x) :
(x-a)^{\gamma} f(x)\in C[a,b]\big\}. $$ Consider the following Cauchy Problem, for $\lambda \in \R,
n\in \mathbb{N}$ and $n-1<\beta<n$,
\begin{equation}\label{E:cfde}
  \begin{cases}
    (^CD_{0+}^\beta f)(x)-\lambda f(x) = y(x), &x\in[0,b], \\
    f^{(k)}(0)=b_k, \quad &b_k\in \bR \text{ for } k=0, 1, \dots, n-1.
  \end{cases}
\end{equation}
We suppose that $y(x)\in C_\gamma [0,b]$ with $0\leq \gamma < 1$ and $\gamma \leq \beta$. Then
\eqref{E:cfde} has a unique solution given by (see \cite[(4.1.62)]{kilbas.srivastava.ea:06:theory}):
\begin{equation}\label{y1}
  f(x) = \sum_{j=0}^{n-1}b_jx^jE_{\beta,j+1}(\lambda x^\beta)+ \int_{0}^{x}(x-t)^{\beta-1}E_{\beta,\beta}(\lambda (x-t)^\beta)y(t)\ud t,
\end{equation}
where $E_{a,b}(z)$ is the Mittag-Leffler function; see \eqref{E:ML}. One may get more explicit
expressions for special values of $a$ and $b$, which will be used in this paper:
\begin{equation} \label{E:ML-ex}
  E_{\frac{1}{2}}(z) = 2e^{z^2}\Phi\left(\sqrt{2}z\right), \quad
  E_{1}(z)   = e^z, \quad
  E_{2}(z)   = \cosh(\sqrt{z}), \quad
  E_{2,2}(z) = \frac{\sinh(\sqrt z)}{\sqrt z},
\end{equation}
where $\Phi(x) = \frac1{\sqrt{2\pi}}\int_{-\infty}^x e^{-\frac{x^2}2 } \ud x$ is the cumulative
distribution function of standard normal distribution. Another formula that will be useful in this
paper is
\begin{align} \label{E:ML_a+b}
  E_{\alpha,\beta}\left(|z|\right) - \frac{1}{\Gamma(\beta)} = |z|
  E_{\alpha,\alpha+\beta}\left(|z|\right),
\end{align}
which can be obtained immediately using the definition of the Mittag-Leffler function in
\eqref{E:ML}; see also (1.8.38) on p.45 of \cite{kilbas.srivastava.ea:06:theory}. The asymptotic
behavior of the Mittag-Leffler function along the positive and negative real lines plays an
important role in the paper, which has been summarized in the following lemma:

\begin{lemma}\label{L:ML-Asymp}
  If all $a>0$ and $b\in \mathbb{C}$, we have that
  \begin{itemize}[wide=0pt]
    \item if $a<2$, as $z\to+\infty$,
      \begin{equation*}
        E_{a,b}(z) = \dfrac{1}{a}z^{(1-b)/a}\exp(z^{1/a})-\dfrac{1}{\Gamma(b-a)}\dfrac{1}{z}+O(z^{-2});
      \end{equation*}
    \item if $a\geq2$, as $z\to+\infty$,
      \begin{equation*}
        E_{a,b}(z)
        = \dfrac{1}{a}\sum_{n\in\mathbb Z: |n|\le a/4}\left( z^{1/a}\exp\left[\dfrac{2n\pi i}{a}\right]\right)^{1-b}\exp\left[\exp\left(\dfrac{2n\pi i}{a}\right) z^{1/a} \right]
        - \frac{1}{\Gamma(b-a)}\dfrac{1}{z}
        + O(|z|^{-2});
      \end{equation*}
    \item if $a<2$, as $z\to-\infty$,
      \begin{equation*}
        E_{a,b}(z) = -\dfrac{1}{\Gamma\left(b-a\right)}\dfrac{1}{z} + O\left(z^{-2}\right);
      \end{equation*}
    \item if $a=2$, as $z\to-\infty$,
      \begin{equation*}
        E_{a,b}(z) = |z|^{(1-b)/2}\cos\left(\sqrt{|z|}+\dfrac{(1-b)\pi}{2}\right)-\dfrac{1}{\Gamma(b-2)}\dfrac{1}{z}+O(z^{-2}).
      \end{equation*}
  \end{itemize}
  In particular, for all $C>0$, $\lim_{t\rightarrow\infty}1/t\log E_{a,b}\left(Ct^a\right) =
  C^{1/a}$.
\end{lemma}
\begin{proof}
  The case when $z\to\infty$ is derived from 1.8.27 (resp. 1.8.29) of
  \cite{kilbas.srivastava.ea:06:theory} when $a<2$ (resp. $a\ge 2$). The case when $a<2$ (resp. $a=2$) and $z\to-\infty$ is a consequence of 1.8.28 (resp. 1.8.31) ({\it
  ibid.}). When $a<2$, the statement for the limit is a direct consequence of the asymptotics at $+\infty$. When $a\geq2$, denoting $z=Ct^a$, we have
  \begin{equation*}
  \begin{split}
     &\lim_{t\rightarrow\infty}\frac{1}{t}\log E_{a,b}(z)
      = \lim_{t\rightarrow\infty}\frac{1}{t}\log \sum_{n\in\mathbb Z: |n|\le a/4} \left( z^{\frac{1}{a}}\exp\left[\frac{2n\pi i}{a}\right]\right)^{1-b}\exp\left[\exp\left(\frac{2n\pi i}{a}\right) z^{\frac{1}{a}} \right]                             \\
     & = \lim_{t\rightarrow\infty}\frac{1}{t}\log \sum_{n\in\mathbb Z: |n|\le a/4} z^\frac{1-b}{a} \exp\left[z^\frac{1}{a}\cos\frac{2n\pi}{a}\right] \exp \left[i \left(\frac{(1-b)2n\pi}{a}+z^\frac{1}{a}\sin\left(\frac{2n\pi}{a}\right)\right)\right] \\
     & = \lim_{t\rightarrow\infty}\frac{1}{t}\log \left( z^\frac{1-b}{a} \exp\left( z^\frac{1}{a}\right) \right)
       = C^{1/a}.
  \end{split}
  \end{equation*}
\end{proof}

We will use the reflection formula for the Gamma function (see, e.g., \cite[5.5.3 on p.
138]{olver.lozier.ea:10:nist}), namely,
\begin{align} \label{E:Reflection}
  \Gamma(z) \Gamma(1-z)=\pi /\sin\left(\pi z\right),\quad z\ne 0,\pm 1,\cdots.
\end{align}

\section{Some miscellaneous lemmas} \label{S:Lemmas}
In this section, we provide the technical lemmas. Lemma \ref{L:integrability} below will be used
used to prove Dalang's condition \eqref{E:Dalang'} in Theorem \ref{T:Exist}.

\begin{lemma}\label{L:integrability}
  For all $\varepsilon>0$, $a$, $b>0$, and $c\in \R$, it holds that
  \begin{equation}\label{E:finite-int}
    \int_{B_\varepsilon^c(0)} \frac{\cos^2(|x|^a+c)}{|x|^b} \ud x<\infty \quad \text{  if and only if $b>d$.}
  \end{equation}
\end{lemma}
\begin{proof}
  We only consider the case $d\ge 2$ while the case $d=1$ is similar but easier. Denote the integral
  in \eqref{E:finite-int} by $I$. Since the integrand is radial,
  \begin{align*}
    I & = C \int_\varepsilon^\infty  \frac{\cos^2(r^a+c)}{r^b} r^{d-1}\ud r,
    \quad \text{with $C=\frac{2\pi^{d/2}}{\Gamma\left(d/2\right)}$.}
  \end{align*}
  Clearly, $b>d$ is a sufficient condition for \eqref{E:finite-int}. To get the necessity, observe
  that
  \begin{align*}
    I & =   \frac{C}{a}\int_{\varepsilon^a}^\infty  \cos^2(s+c) s^{-\frac1a(b-d) -1} \ud s                               \\
      & \ge \frac{C}{a} \sum_{n=N}^\infty \int_{n\pi-c}^{ \left(n+\frac14\right)\pi-c} \cos^2(s+c) s^{-(b-d)/a -1} \ud s \\
      & \ge \frac{C \pi^{-(b-d)/a}}{8a} \sum_{n=N+1}^\infty n^{-(b-d)/a-1},
  \end{align*}
  where $N=N(\varepsilon^a,c)$ is a finite positive integer. The series on the right-hand side is
  convergent if and only if $b>d$ and thus $b>d$ is also necessary for \eqref{E:finite-int}.
\end{proof}

The following lemma is a convolution-type Gronwall lemma, which was proved in Lemma A.2 of
\cite{chen.hu.ea:21:regularity} for $\theta\in(-1,0)$. But indeed, the same proof can be extended
directly to all $\theta>-1$. One can use this lemma to obtain the moment formulas in Theorem
\ref{T:fde} as pointed out in Remark \ref{R:method2}.
\begin{lemma} \label{L:f(t)}
  Suppose that $\theta>-1$, $\kappa>0$ and that $g(\cdot):\R_+\rightarrow\R$ is a locally integrable
  function. If $f$ satisfies
  \begin{equation}\label{f1}
    f(t)=g(t)+\kappa\int_{0}^{t}(t-s)^\theta f(s)\ud s,\quad \text{for }t\geq 0,
  \end{equation}
  then
  \begin{equation}\label{f2}
    f(t)=g(t)+\int_{0}^{t}g(s)K(t-s)\ud s,
  \end{equation}
  with $K(t)=\kappa \Gamma(\theta+1)t^\theta E_{\theta+1,\theta+1}(\kappa\Gamma(\theta+1)t^{\theta+1})$.
  Moreover, if we further assume $g(\cdot)\geq0$ and the equality in \eqref{f1} is replaced by
  $\leq$ (resp. $\geq$), then the equality in \eqref{f2} is replaced by $\leq$ (resp. $\geq$)
  accordingly.
\end{lemma}

The following lemma will be used to obtain the explicit second moment formulas for stochastic wave
equation (i.e., $\beta=2$) in Example \ref{Ex:SFWE}.
\begin{lemma} \label{L:sin}
  For $\alpha>1$ and $b>0$, it holds that
  \begin{align*}
       \int_0^\infty \frac{\sin^2\left(b\: \xi^{\alpha/2}\right)}{\xi^\alpha}\ud \xi
       =
       \begin{cases}
         2^{2(1-1/\alpha)} \alpha^{-1} \cos\left(\pi/\alpha\right) \Gamma\left(2(1/\alpha-1)\right) b^{2-2/\alpha} & \text{if $\alpha\in (1,2)\cup (2,\infty)$,} \\
         2^{-1}b\pi,                                                                                               & \text{if $\alpha=2$.}
       \end{cases}
  \end{align*}
\end{lemma}
\begin{proof}
  Denote the integral by $I$. By change of variable $z=\xi^{\alpha/2}$, we see that
  \begin{align*}
    I = \frac{2}{\alpha} \int_0^\infty \frac{\sin^2\left(b z\right)}{z^{3-2/\alpha}} \ud z.
  \end{align*}
  Let $f(x)= \one_{[-b,b]}(x)$ and $g(x)$ be an even function defined as, for $x>0$,
  \begin{align*}
    g(x) = \frac{\pi}{4\Gamma\left(2(1-1/\alpha)\right) \sin\left(\pi/\alpha\right)}
    \left[\left(x+b\right)^{1-2/\alpha} + |b-x|^{1-2/\alpha}\text{sgn}(b-x)\right].
  \end{align*}
  Now we compute the Fourier transforms for these two functions. It is clear that
  \begin{align*}
    \widehat{f}(\xi) = \frac{2 \sin(b\: \xi)}{\xi}.
  \end{align*}
  Let $h(\xi) = |\xi|^{-2(1-1/\alpha)} \sin\left(b\: |\xi|\right)$. By (2) on p. 19 of \cite{erdelyi.magnus.ea:54:tables}, we see that
  \begin{align*}
    \mathcal{F}^{-1}h(x) = \frac{1}{2\pi} \int_\R e^{ix\xi} h(\xi) \ud \xi
    = \frac{1}{\pi}\int_0^\infty h(\xi) \cos(x\xi)\ud\xi
    = \frac{1}{\pi} g(x),
  \end{align*}
  under the following condition:
  \begin{align*}
    \left|\frac{2}{\alpha} -1\right| < 1 \quad  \Longleftrightarrow \quad \alpha > 1.
  \end{align*}
  Hence, when $\alpha>1$, we have $\widehat{g}(\xi) = \pi h(\xi)$. Then by the Plancherel theorem,
  \begin{align*}
    \int_\R f(x) g(x)\ud x = \frac{1}{2\pi} \int_\R \widehat{f}(\xi) \overline{\widehat{g}(\xi)} \ud \xi
    = \frac{1}{\pi}\int_0^\infty \frac{2 \sin(b\: \xi)}{\xi} \pi \xi^{-2(1-1/\alpha)} \sin\left(b\: \xi\right) \ud \xi
    = \alpha I.
  \end{align*}
  On the other hand,
  \begin{align*}
    \int_\R f(x) g(x)\ud x
    = & 2 \int_0^b g(x)\ud x                                                                                                                                                                 \\
    = & \frac{\pi}{2\Gamma\left(2(1-1/\alpha)\right) \sin\left(\pi/\alpha\right)} \int_0^b \left[\left(x+b\right)^{1-2/\alpha} + (b-x)^{1-2/\alpha}\right]\ud x                              \\
    = & \frac{\pi 2^{1- 2/\alpha} b^{2-2/\alpha}}{\Gamma\left(2-2/\alpha\right)\left(2-2/\alpha\right)\sin\left(\pi/\alpha\right) }                                                          \\
    = & \frac{\pi 2^{1- 2/\alpha} b^{2-2/\alpha}}{\Gamma\left(3-2/\alpha\right)\sin\left(\pi/\alpha\right)}.
  \end{align*}
  Hence, if $\alpha=2$, the above expression becomes  $b\pi$. This proves the lemma for the case  $\alpha=2$. Now if  $\alpha\ne 2$, we have
  \begin{align*}
    \int_\R f(x) g(x)\ud x
    = & \frac{\pi 2^{1- 2/\alpha} b^{2-2/\alpha}}{\Gamma\left(3-2/\alpha\right)\sin\left(\pi/\alpha\right)} \times \frac{\Gamma\left(2(1/\alpha-1)\right)}{\Gamma\left(2(1/\alpha-1)\right)} \\
    = & \frac{2^{1- 2/\alpha} b^{2-2/\alpha}\Gamma\left(2(1/\alpha-1)\right)}{\sin\left(\pi/\alpha\right)} \sin\left(2\pi/\alpha\right)                                                      \\
    = & 2^{2- 2/\alpha} b^{2-2/\alpha}\Gamma\left(2(1/\alpha-1)\right)\cos\left(\pi/\alpha\right).
  \end{align*}
  where we have applied the reflection formula for Gamma function \eqref{E:Reflection}. This proves
  the lemma.
\end{proof}

The following lemma will be used to prove the lower bound of moment estimates in Theorem
\ref{T:lower-bd}. For two sequence of positive numbers $a_n, b_n, n\in\mathbb{N}$, we denote
$a_n\sim b_n$ if $\lim_{n\to \infty}a_n/b_n=1$.

\begin{lemma}\label{L:exp-1}
  (1) As $n\to\infty$, we have
  \begin{gather}\label{E:gamma-n}
    \int_n^\infty t^n e^{-t} \ud t \sim \frac{n^n}{e^n} \sqrt{\frac{\pi n}{2}} \quad \text{and}\\
    \sum_{m=0}^{n-1} \frac{n^m}{m!} \sim \sum_{m=n}^{2n-1} \frac{n^m}{m!}\sim \frac12 e^n.
    \label{E:exp-n}
  \end{gather}
  (2) Given $\alpha>0$, for $n$ sufficiently large, there exist two positive constants $c_1$ and $
  c_2$ depending only on $\alpha$ such that
  \begin{align*}
   \sum_{m=n}^\infty \left( \frac{n^m}{m!}\right)^\alpha \ge c_1 \exp(c_2n).
  \end{align*}
\end{lemma}
\begin{proof}
  (1) Denote the integral in \eqref{E:gamma-n} by $I$. By change of variable
  $x=t/\sqrt{n}-\sqrt{n}$, we see that
  \begin{equation*}
    I= \frac{n^n\sqrt{n}}{e^n}\int_0^\infty e^{-\sqrt{n} x} \left(1+\frac{x}{\sqrt{n}}\right)^n \ud x.
  \end{equation*}
  Notice that $\left(1+ \frac{x}{\sqrt{n}}\right)^n e^{-\sqrt{n} x} \le (1+x)e^{-x}$ for all $t>0$
  and $n\ge 1$ with the upper bound being integrable. Hence, by the dominated convergence theorem
  and L'Hospital's rule, we conclude that
  \begin{align*}
    \lim_{n \to \infty}I \frac{ e^n}{n^n \sqrt{n}}
    = \int_0^\infty \lim_{n \to \infty}e^{-\sqrt{n} x} \left(1+\frac{x}{\sqrt{n}}\right)^n \ud x
    = \int_0^\infty e^{-x^2 /2}\ud x = \sqrt{\frac{\pi}{2}},
  \end{align*}
  which proves \eqref{E:gamma-n}.

  To prove \eqref{E:exp-n}, it suffices to show $R_n(n) \sim \sfrac12\: e^n$ and $\lim_{n\to \infty}
  e^{-n}R_{2n}(n)=0$, where $R_k(x)$ is the remainder function for the Taylor expansion of
  $e^x$:
  \begin{align*}
    R_k(x) = \sum_{m=k}^{\infty} \frac{x^m}{m!}=\int_0^x \frac{(x-t)^k}{k!} e^ t \ud t.
  \end{align*}
  For $R_n(n)$, by change of variable we get
  \begin{align*}
    R_n(n) = \int_0^n \frac{(n-t)^n}{n!} e^t \ud t = \frac{e^n}{n!} \int_0^n x^n e^{-x}\ud x.
  \end{align*}
  By Stirling's formula $n! \sim \sqrt{2\pi n} e^{-n} n^n$ and \eqref{E:gamma-n}, we can show
  \begin{equation}\label{E:int-0-n}
    \int_0^n x^ne^{-x}\ud x = \left(\Gamma(n+1) -\int_n^\infty x^n e^{-x}\ud x\right)\sim \frac12 n!.
  \end{equation}
  For $R_{2n}(n)$, we have
  \begin{align*}
    R_{2n}(n) = \int_0^n \frac{(n-t)^{2n}}{(2n)!} e^t \ud t\le \frac{n^n}{(2n)!} \int_0^n (n-t)^n e^{t} \ud t=\frac{e^n n^n}{(2n)!}\int_0^n x^ne^{-x}\ud x.
  \end{align*}
  Thus, by \eqref{E:int-0-n}, we have
  \begin{align*}
        \lim_{n\to\infty} \frac{R_{2n}(n)}{e^n}
    \le \lim_{n\to\infty} \frac12\frac{n^n (n!)}{(2n)!}
    =   \lim_{n\to\infty} \frac12 \frac{n^{2n} e^{-n} \sqrt{2\pi n} }{ (2n)^{2n} e^{-2n}\sqrt{4\pi n}}
    =   0,
  \end{align*}
  which proves \eqref{E:exp-n}. \bigskip

  (2) The desired result follows directly from the fact
  \begin{align*}
    \sum_{m=n}^\infty \left( \frac{n^m}{m!} \right)^\alpha
    \ge \sum_{m=n}^{2n-1} \left( \frac{n^m}{m!} \right)^\alpha
    \ge
    \begin{cases}
      \displaystyle \left( \sum_{m=n}^{2n-1} \frac{n^m}{m!} \right)^\alpha,             & \text{if $\alpha\in(0,1]$}, \\
      \displaystyle n^{1-\alpha} \left( \sum_{m=n}^{2n-1} \frac{n^m}{m!}\right)^\alpha, & \text{if $\alpha>1$}.
    \end{cases}
  \end{align*}
  Then an application of \eqref{E:exp-n} proves (2).
\end{proof}

\section{Fundamental solutions} \label{S:Funamental}

The fundamental solutions to \eqref{E:fde} in case when $d=1$, $\beta=1$, $\gamma=0$, and $\alpha\in
2 \mathbb{N}$ (i.e., $\alpha$ is an even integer) have been studied in \cite{krylov:60:some} and
\cite{hochberg:78:signed}. In \cite{debbi:06:explicit}, Debbi studied the fundamental solutions to
\eqref{E:fde} when $d=1$, $\beta=1$, $\gamma=0$, and $\alpha\in (1,\infty)\setminus \mathbb{N}$ and
then, with Dozzi \cite{debbi.dozzi:05:on}, they studied the corresponding SPDEs with space-time
white noise. This part can be viewed as a generalization of their results to a class of more general
of SPDEs. The Fox H-functions \cite{kilbas.saigo:04:h-transforms} allow us to study the fundamental
solutions to \eqref{E:fde} with the much more general parameters in a unified way.

The following theorem generalizes Theorem 4.1 of \cite{chen.hu.ea:19:nonlinear} from $\alpha\in
(0,2]$ and $\beta\in (0,2)$ to the case $\alpha>0$ and $\beta\in (0,2]$. The statement of the
theorem remains almost the same except the conditions on $\alpha$ and $\beta$. The proof also
follows the same lines of arguments as those in \cite{chen.hu.ea:19:nonlinear}; one may also check
the proof of Theorem 3.1 in \cite{chen.hu.ea:17:space-time} for the case when $\gamma=0$. The case
when $\beta=2$ is new. For the readers' convenience, we state the theorem below and present its
proof to explicitly show why the ranges of $\alpha$ and $\beta$ can be generalized.

\begin{theorem}\label{T:PDE}
  For $\alpha\in (0,\infty)$, $\beta\in (0,2]$, and $\gamma\ge 0$, the solution to
  \begin{align} \label{E:PDE}
    \begin{cases}
      \left(\partial_t^\beta + \dfrac{\nu}{2} (-\Delta)^{\alpha/2} \right) u(t,x)= I_t^\gamma\left[f(t,x)\right], & \qquad t>0,\: x\in\R^d, \\[1em]
      \left.\dfrac{\partial^k}{\partial t^k} u(t,x)\right|_{t=0}=u_k(x),                                          & \qquad 0\le k\le \Ceil{\beta}-1, \:\: x\in\R^d,
    \end{cases}
  \end{align}
  is
  \begin{align}\label{E:Duhamel}
    u(t,x) = J_0(t,x) + \int_0^t \ud s \int_{\R^d} \ud y\: f(s,y) \: \DRL{0+}{\Ceil{\beta}-\beta-\gamma} Z(t-s,x-y),
  \end{align}
  where $\DRL{0+}{\Ceil{\beta}-\beta-\gamma}$ denotes the {\em Riemann-Liouville derivative}
  $D_{0+}^{\Ceil{\beta}-\beta-\gamma}$ acting on the time variable,
  \begin{align} \label{E:J0}
    J_0(t,x):= \sum_{k=0}^{\Ceil{\beta}-1}\int_{\R^d} u_{k}(y) \partial_t^{\Ceil{\beta}-1-k} Z(t,x-y) \ud y
  \end{align}
  is the solution to the homogeneous equation and $Z(t,x) := Z_{\alpha,\beta,d}(t,x)$ is the
  corresponding fundamental solution. If we denote
  \begin{gather*}
    Y(t,x) := Y_{\alpha,\beta,\gamma,d}(t,x) = \DRL{0+}{\Ceil{\beta}-\beta-\gamma}
    Z_{\alpha,\beta,d}(t,x), \\
    Z^*(t,x) := Z^{*}_{\alpha,\beta,d}(t,x) = \frac{\partial}{\partial t} Z_{\alpha,\beta,d}(t,x),
    \quad \text{if $\beta\in(1,2]$,}
  \end{gather*}
  then we have the following Fourier transforms:
  \begin{align} \label{E:FZ}
   \cF Z_{\alpha,\beta,d}(t,\cdot)(\xi)        & = t^{\Ceil{\beta}-1} E_{\beta,\Ceil{\beta}}(-2^{-1}\nu t^\beta |\xi|^\alpha),              \\
   \cF Y_{\alpha,\beta,\gamma,d}(t,\cdot)(\xi) & = t^{\beta+\gamma-1} E_{\beta,\beta+\gamma}(-2^{-1}\nu t^\beta |\xi|^\alpha), \label{E:FY} \\
   \cF Z^*_{\alpha,\beta,d}(t,\cdot)(\xi)      & = t^{k} E_{\beta,k+1}(-2^{-1}\nu t^\beta |\xi|^\alpha), \quad \text{if $\beta\in (1,2]$}.
   \label{E:FZ*}
  \end{align}
  Moreover, when $\beta\in (0,2)$, we have the following explicit expressions:
  \begin{align} \label{E:Zab}
     Z(t,x) = \pi^{-d/2} t^{\Ceil{\beta}-1} |x|^{-d}
     \FoxH{2,1}{2,3}{\frac{ |x|^\alpha}{2^{\alpha-1}\nu t^\beta}}{(1,1),\:(\Ceil{\beta},\beta)}
     {(d/2,\alpha/2),\:(1,1),\:(1,\alpha/2)},
  \end{align}
  \begin{align} \label{E:Yab}
    \begin{aligned}
      Y(t,x) =  \pi^{-d/2} |x|^{-d}t^{\beta+\gamma-1}
       \FoxH{2,1}{2,3}{\frac{ |x|^\alpha}{2^{\alpha-1}\nu t^\beta}}
       {(1,1),\:(\beta+\gamma,\beta)}{(d/2,\alpha/2),\:(1,1),\:(1,\alpha/2)},
    \end{aligned}
  \end{align}
  and, if $\beta\in (1,2)$,
  \begin{align} \label{E:Z*ab}
    Z^{*}(t,x)=
    \pi^{-d/2} |x|^{-d}
    \FoxH{2,1}{2,3}{\frac{ |x|^\alpha}{2^{\alpha-1}\nu t^\beta}}{(1,1),\:(1,\beta)}
    {(d/2,\alpha/2),\:(1,1),\:(1,\alpha/2)}.
  \end{align}
\end{theorem}
\begin{proof}
  The proof follows a standard argument using the Fourier and Laplace transforms in the space and
  time variables, respectively, which are denoted by $\widehat{f}$ and $\widetilde{g}$. Let us apply
  the Fourier transform to \eqref{E:PDE} first to obtain
  \[
  \begin{cases}
  \displaystyle \partial_t^\beta \widehat{u}(t,\xi)+\frac{\nu}{2}|\xi|^\alpha
  \widehat{u}(t,\xi)=I_t^\gamma \left[\widehat{f}(t,\xi)\right]\;,& \xi\in\R^d\\[0.5em]
  \displaystyle
  \left.\frac{\partial^k}{\partial t^k} \widehat{u}(t,\xi)\right|_{t=0} = \widehat{u}_k(\xi)\;,
  &\text{$0\le k\le \Ceil{\beta}-1$, $\xi\in\R^d$\:.}
  \end{cases}
  \]
  Apply the Laplace transform on the Caputo derivative using \cite[Theorem 7.1 on p.
  134]{diethelm:10:analysis}:
  \[
    \cL\left[\partial_t^\beta\: \widehat{u}(t,\xi)\right](s) = s^\beta \;\widetilde{\widehat{u}}(s,\xi) -
    \sum_{k=0}^{\Ceil{\beta}-1} s^{\beta-1-k}\; \widehat{u}_k(\xi).
  \]
  On the other hand, it is known that (see, e.g., \cite[(7.14) on p.
  140]{samko.kilbas.ea:93:fractional}),
  \[
    \cL I_t^\gamma\left[\widehat{f}(t,\xi)\right] = s^{-\gamma} \widetilde{\widehat{f}}(s,\xi),\quad \Re(\gamma)>0.
  \]
  Thus,
  \[
    \widetilde{\widehat{u}}(s,\xi)
    = \left(s^\beta +\frac\nu2 \:|\xi|^\alpha \right)^{-1}
      \left[\sum_{k=0}^{\Ceil{\beta}-1} s^{\beta-1-k}\; \widehat{u}_k(\xi)
            +s^{-\gamma} \widetilde{\widehat{f}}(s,\xi)
      \right].
  \]
  Notice that (see, e.g., \cite[(1.80) on p. 21]{podlubny:99:fractional})
  \[
    \cL\left[t^{\beta-1}E_{\alpha,\beta}(-\lambda t^\alpha) \right](s) = \frac{s^{\alpha-\beta}}{s^\alpha+\lambda},
    \quad\text{for $\Re(s)>|\lambda|^{1/\alpha}$.}
  \]
  Hence,
  \[
    \widehat{u}(t,\xi)=\sum_{k=0}^{\Ceil{\beta}-1}t^k E_{\beta,k+1}\left(-\frac{\nu}{2}|\xi|^\alpha t^\beta\right)\widehat{u}_k(\xi)
      + \int_0^t \ud \tau \: \tau^{\beta+\gamma-1} E_{\beta,\beta+\gamma}\left(-\frac{\nu}{2}|\xi|^\alpha\tau^\beta\right)\widehat{f}(t-\tau,\xi).
  \]
  Now if we denote
  \begin{align} \label{E:Uab}
    U(t,\xi):=t^{\Ceil{\beta}-1} E_{\beta,\Ceil{\beta}}\left(-\frac{\nu}{2}|\xi|^\alpha t^\beta\right),
  \end{align}
  using the fact that $\lMr{t}{D}{0+}^\gamma = \frac{\ud^\gamma}{\ud t^\gamma}$ when $\gamma\in
  \mathbb{Z}$ and for all $\gamma \in \R$ (see \cite[(1.82) on p. 21]{podlubny:99:fractional})
  \[
    \lMr{t}{D}{0+}^\gamma \left(t^{\beta-1}E_{\alpha,\beta}(\lambda t^\alpha)\right) = t^{\beta-\gamma-1}
    E_{\alpha,\beta-\gamma}(\lambda t^\alpha),
  \]
  we see that
  \[
    \widehat{u}(t,\xi)=\sum_{k=0}^{\Ceil{\beta}-1} \left(\frac{\ud^k}{\ud t^k} U(t,\xi) \right) \widehat{u}_{\Ceil{\beta}-1-k}(\xi)
    + \int_0^t \ud \tau \: \left(\lMr{t}{D}{0+}^{\Ceil{\beta}-\beta-\gamma} U(\tau,\xi)\right)\widehat{f}(t-\tau,\xi).
  \]

  It remains to prove the expressions in \eqref{E:Zab} -- \eqref{E:Z*ab} under the assumption that
  $\beta\in (0,2)$. A key observation is that for $Z_{\alpha,\beta,d}(t,x)$ defined in
  \eqref{E:Zab}, its Fourier transform is given by $U(t,\xi)$ in \eqref{E:Uab}, namely,
  \begin{align} \label{E:FU}
    \mathcal{F} Z_{\alpha,\beta,d}(t,\cdot) (\xi) = U(t,\xi), \quad
    \text{for all $\alpha>0$,  $\beta\in(0,2)$, and  $d\ge 1$.}
  \end{align}
  Indeed, \eqref{E:FU} is proved in Lemma 4.2 of \cite{chen.hu.ea:17:space-time}, but only for the
  case of $\alpha\in (0,2]$. Here we claim that the restriction of $\alpha\in (0,2]$ is not
  necessary. In the proof of this lemma, one needs to consider two cases separately: $d=1$ and $d\ge
  2$. In the case of $d=1$, the conditions we need are
  \begin{align*}
    \frac{2-\beta}{\alpha} >0 \quad \text{and} \quad 1\wedge \alpha >0.
  \end{align*}
  For the second case -- $d\ge 2$, the proof is a direct application of Corollary 2.5.1 of
  \cite{kilbas.saigo:04:h-transforms}, where one needs to verify the following conditions:

  \begin{center}
    \renewcommand{\arraystretch}{1.2}
    \begin{tabular}{|c|c|} \hline
      Conditions in \cite{kilbas.saigo:04:h-transforms} & the corresponding conditions in our setting \\ \hline
      $a^*>0$                                           & $2-\beta>0$                                 \\
      (2.6.8)                                           & $\min(\alpha,d)>0$                          \\
      (2.6.9)                                           & $d>1$                                       \\
      (2.6.10)                                          & $d>1$                                       \\ \hline
    \end{tabular}
  \end{center}

  Apparently, the above two conditions hold for all $\alpha>0$ and $\beta\in (0,2)$. Hence, Lemma
  4.2 of \cite{chen.hu.ea:17:space-time} is true for all $\alpha>0$ and $\beta\in (0,2)$. This
  proves both \eqref{E:FZ} and \eqref{E:Zab}. Once one obtains the expressions for
  $Z_{\alpha,\beta,d}(t,x)$ and $\mathcal F Z_{\alpha,\beta,d}(t,\cdot)(\xi)$, it is routine to
  obtain the corresponding expressions of their fractional or integer derivatives/integrals; see
  \cite{chen.hu.ea:17:space-time} for more details. This completes the proof of Theorem \ref{T:PDE}.
\end{proof}
\begin{remark}
  For the case $\beta=2$, the expression in \eqref{E:FZ} can be simplified using the fourth
  expression in \eqref{E:ML-ex}.
\end{remark}

\bigskip

\paragraph{\textbf{Acknowledgment.}} The authors wish to thank Xiong Wang for helpful discussions.   J. Song is partially supported by Shandong University grant 11140089963041 and National Natural Science Foundation of China grant 12071256.

\bigskip

\printbibliography[title={References}]
\end{document}